\documentclass[11pt]{amsart}
\usepackage{graphicx,amssymb,amscd,amsfonts,pinlabel,color,bm,times,overpic,psfrag}

\theoremstyle{plain}
\newtheorem{thm}{Theorem}[section]
\newtheorem{claim}[thm]{Claim}

\newtheorem{prop}[thm]{Proposition}
\newtheorem{lemma}[thm]{Lemma}
\newtheorem{cor}[thm]{Corollary}

\newtheorem{q}[thm]{Question}

\theoremstyle{definition}
\newtheorem{defn}[thm]{Definition}

\theoremstyle{remark}
\newtheorem{rmk}[thm]{Remark}
\newtheorem{example}[thm]{Example}

\newcommand{\C}{\mathbb{C}}
\newcommand{\R}{\mathbb{R}}
\newcommand{\Z}{\mathbb{Z}}

\newcommand{\F}{\mathbb{F}}

\newcommand{\bdry}{\partial}
\newcommand{\pa}{\partial}

\newcommand{\s}{\vskip.1in}
\newcommand{\n}{\noindent}
\newcommand{\krn}{\operatorname{ker}}
\newcommand{\Ordo}{O} 
\newcommand{\M}{\mathcal{M}}
\newcommand{\MM}{\mathcal{M}}
\newcommand{\TT}{\mathcal{T}}
\newcommand{\A}{\mathcal{A}}
\newcommand{\CC}{\mathcal{C}}
\newcommand{\bb}{\mathbf{b}}

\newcommand{\ind}{\operatorname{ind}}
\newcommand{\be}{\begin{enumerate}}
\newcommand{\ee}{\end{enumerate}}
\newcommand{\Mo}{{\rm Mo}}
\newcommand{\co}{{\rm co}}
\newcommand{\op}{\operatorname}

\numberwithin{equation}{section}

\title{Legendrian knots and exact Lagrangian cobordisms}

\author{Tobias Ekholm}
\address{Uppsala Universitet, Box 480, 751 06 Uppsala, Sweden;\newline\indent
Institut Mittag-Leffler, Aurav.\,17, 182 60 Djursholm, Sweden}
\email{tobias.ekholm@math.uu.se}
\urladdr{http://katalog.uu.se/empInfo/?id=N94-2099}

\author{Ko Honda}
\address{University of Southern California, Los Angeles, CA 90089}
\email{khonda@usc.edu} \urladdr{http://www-bcf.usc.edu/\char126 khonda}

\author{Tam\'as K\'alm\'an}
\address{Tokyo Institute of Technology, 152-8551 Meguro, Tokyo, Japan}
\email{kalman@math.titech.ac.jp} \urladdr{http://math.titech.ac.jp/\char126kalman}

\date{This version: December 12, 2012.}



\begin{document}

\begin{abstract}
We introduce constructions of exact Lagrangian cobordisms with cylindrical Legendrian ends and study their invariants which arise from Symplectic Field Theory. A pair $(X,L)$ consisting of an exact symplectic manifold $X$ and an exact Lagrangian cobordism $L\subset X$ which agrees with cylinders over Legendrian links $\Lambda_+$ and $\Lambda_-$ at the positive and negative ends induces a differential graded algebra (DGA) map from the Legendrian contact homology DGA of $\Lambda_+$ to that of $\Lambda_-$. We give a gradient flow tree description of the DGA maps for certain pairs $(X,L)$, which in turn yields a purely combinatorial description of the cobordism map for elementary cobordisms, i.e., cobordisms that correspond to certain local modifications of Legendrian knots. As an application, we find exact Lagrangian surfaces that fill a fixed Legendrian link and are not isotopic through exact Lagrangian surfaces.
\end{abstract}

\thanks{TE supported by G\"oran Gustafsson Foundation for Research in Natural Sciences and Medicine; KH supported by NSF Grants DMS-0805352 and DMS-1105432 and a Simons Fellowship; TK supported by a Japan Society for the Promotion of Science (JSPS) Grant-in-Aid for Young Scientists B (no.\ 21740041).}

\maketitle


\section{Introduction}
\label{Sec:intr}

The goal of this paper is to introduce constructions of exact Lagrangian cobordisms with cylindrical Legendrian ends and study their invariants which arise from the Symplectic Field Theory (SFT) of Eliashberg-Givental-Hofer~\cite{EGH}.

\subsection{Preliminaries}

The \emph{standard contact structure} on $\R^{3}$ is the $2$-plane field $\xi_0=\krn(\alpha_0)$, where $\alpha_0=dz-ydx$ with respect to the standard coordinates $(x,y,z)$. A $1$-dimensional submanifold $\Lambda\subset \R^3$ is \emph{Legendrian} if it is everywhere tangent to $\xi_0$. We assume that Legendrian knots and links are closed and oriented, unless stated otherwise.

The \emph{Reeb vector field} $R_{\alpha_0}$ of $\alpha_0$ is given by $R_{\alpha_0}=\bdry_z$. A \emph{Reeb chord} of a Legendrian link $\Lambda\subset \R^3$ is an integral curve of $R_{\alpha_0}$ with initial and terminal points on $\Lambda$. We write $\CC(\Lambda)$ for the set of Reeb chords of $\Lambda$.

An {\em exact symplectic manifold} is a triple $(X,\omega,\beta)$ consisting of a $2n$-dimensional symplectic manifold $(X,\omega)$ and a $1$-form $\beta$ satisfying $d\beta=\omega$. We will often abbreviate an exact symplectic manifold as $(X,\beta)$ or $(X,d\beta)$.  An $n$-dimensional submanifold $L\subset (X,\omega,\beta)$ is {\em Lagrangian} if $\omega\big|_{L}=0$ and is {\em exact Lagrangian} if $\beta\big|_{L}$ is exact in addition.

The \emph{symplectization} of $(\R^{3},\alpha_0)$ is the exact symplectic manifold $(\R\times\R^{3},e^{t}\alpha_0)$, where $t$ is the coordinate of the first $\R$-factor. A Legendrian link $\Lambda\subset\R^{3}$ gives rise to a cylindrical Lagrangian submanifold $\R\times\Lambda\subset\R\times\R^{3}$.

\begin{defn} \label{defn: exact Lagrangian with cylindrical ends}
Let $\Lambda_+$ and $\Lambda_-$ be Legendrian links in $\R^3$, $(X,\beta)$ an exact symplectic manifold whose positive and negative ends agree with the positive and negative ends of $(\R\times \R^3,e^t\alpha_0)$, and $L\subset (X,\beta)$ an oriented exact Lagrangian submanifold. Then the pair $((X,\beta),L)$ is an {\em exact Lagrangian cobordism from $\Lambda_+$ to $\Lambda_-$ with cylindrical Legendrian ends $\mathcal{E}_\pm(L)$} if there exists $T>0$ such that the following hold:
\begin{equation}\begin{split}
\label{eq:legends}
\mathcal{E}_+(L)&= L\cap((T,\infty)\times\R^{3}) = (T,\infty)\times\Lambda_+;\\
\mathcal{E}_-(L)&= L\cap((-\infty,-T)\times\R^{3}) = (-\infty,-T)\times\Lambda_-;
\end{split}\end{equation}
and
\begin{itemize}
\item[(i)] $f$ is constant on each of $\mathcal{E}_+(L)$ and $\mathcal{E}_-(L)$ whenever $df=\beta|_L$.\footnote{It is automatic that $f$ is constant on each component of $\mathcal{E}_\pm(L)$. If $f$ is not the same constant on all the components of $\mathcal{E}_-(L)$ or on all the components of $\mathcal{E}_+(L)$, then the composition of exact Lagrangians is not necessarily exact Lagrangian and there are problems defining the DGA morphism $\Phi_{(X,L)}$ (see Section~\ref{subsection: functor}). These were communicated to the authors by Chantraine and Ghiggini; see~\cite{Cha2}.}
\item[(ii)] $L$ is compact with boundary $\Lambda_+-\Lambda_-$ after removing the cylindrical ends $\mathcal{E}_\pm(L)$.
\end{itemize}
An exact Lagrangian cobordism $((X,\beta),L)$ from $\Lambda$ to $\varnothing$ is an {\em exact Lagrangian filling of $\Lambda$}.
\end{defn}


The most general ambient exact symplectic manifold $(X,\omega)$ that we consider in this paper is the {\em completion of the cotangent bundle $T^{\ast}F$ of a surface $F$,} where $F=(\R\times[a_-,a_+]) \# \Sigma$ is the connected sum of a closed surface $\Sigma$ and a strip $\R\times[a_-,a_+]$ and $\omega|_{T^{\ast} F}$ is the canonical symplectic form. Let $\pa_{-}F=\R\times\{a_-\}$ and $\pa_{+}F=\R\times\{a_+\}$. Then $X$ is the exact symplectic manifold obtained by attaching half-symplectizations $(-\infty,0]\times\R^{3}$ and $[0,\infty)\times\R^{3}$ to $T^{\ast}F$ along $T^{\ast} F\big|_{\pa_- F}$ and $T^{\ast} F\big|_{\pa_+ F}$; for more details, see Section~\ref{sec:lagcob}.

\subsection{The Legendrian contact homology functor $\Phi$} \label{subsection: functor}

Let $\mathfrak{cob}$ be the category whose objects are ``chord generic'' Legendrian links $\Lambda\subset \R^3$ (see Definition~\ref{chord generic}) and whose morphisms $\op{Hom}(\Lambda_+,\Lambda_-)$ are exact Lagrangian cobordisms from $\Lambda_+$ to $\Lambda_-$ with cylindrical Legendrian ends.  The pair $(\R\times\R^{3},\R\times\Lambda)$ is the identity in $\op{Hom}(\Lambda,\Lambda)$. If $(X_{21},L_{21})\in \op{Hom}(\Lambda_2,\Lambda_1)$ and $(X_{10},L_{10})\in \op{Hom}(\Lambda_1,\Lambda_0)$, then their product $(X_{20},L_{20})=(X_{21},L_{21})\cdot (X_{10},L_{10})\in \op{Hom}(\Lambda_2,\Lambda_0)$ is obtained from $(X_{21},L_{21})$ and $(X_{10},L_{10})$ by concatenating the two along their common $\Lambda_1$-end.\footnote{The process of concatenation is noncanonical and some care is needed to make the concatenation associative. One way to do this is to remember the data of the cylindrical ends $\mathcal{E}_\pm(L)$.  When we concatenate we truncate the ends so that we have a fixed width $C$ that is left from each end, i.e., $(T,T+C)\times \Lambda_+$ and $(-T-C,-T)\times \Lambda_-$.}
Also let $\mathfrak{dg}$ be the category of unital differential graded algebras (DGAs) over a field $\F$.

We now describe the Legendrian contact homology functor $\Phi: \mathfrak{cob}\to \mathfrak{dg}$. The functor $\Phi$ associates a unital DGA $\A(\Lambda)$ to a generic Legendrian link $\Lambda\subset\R^{3}$ which is freely generated by the union of $\CC(\Lambda)$ and a basis for $H_1(\Lambda;\Z)$. Here both the Reeb chords and the elements of $H_1(\Lambda;\Z)$ are graded by a Maslov index. The differential of $\A(\Lambda)$ is defined by a count of punctured $J$-holomorphic disks in $\R\times\R^{3}$, where $J$ is an ``adjusted'' almost complex structure; in particular, $J$ is invariant with respect to the $t$-translation. (See Section~\ref{subsection: moduli of holomorphic disks} for more details.) The disks have boundary on $\R\times\Lambda$ and are asymptotic to strips over Reeb chords near the punctures; see the left panel of Figure~\ref{fig:sft}. Furthermore, the disks that contribute to the differential are required to be rigid up to translation in the $t$-direction. When $\Lambda=\varnothing$, then $\A(\Lambda)=\F$ with the trivial differential $\bdry=0$ and the degree of elements in $\F$ is zero.

\begin{figure}[ht]
\labellist
\small
\pinlabel $t\ll0$ at -60 0
\pinlabel $t\gg0$ at -60 330
\pinlabel $t\ll0$ at 1180 170
\pinlabel $t\gg0$ at 1180 500
\pinlabel $a$ at 240 390
\pinlabel $a$ at 840 410
\pinlabel $b$ at 140 30
\pinlabel $b$ at 770 30
\pinlabel $c$ at 270 60
\pinlabel $c$ at 890 60
\pinlabel $\Lambda$ at 370 400
\pinlabel $\Lambda$ at 370 70
\pinlabel $\Lambda_+$ at 920 450
\pinlabel $\Lambda_-$ at 990 30
\pinlabel $L$ at 770 280
\endlabellist
\centering
\includegraphics[height=1.7in]{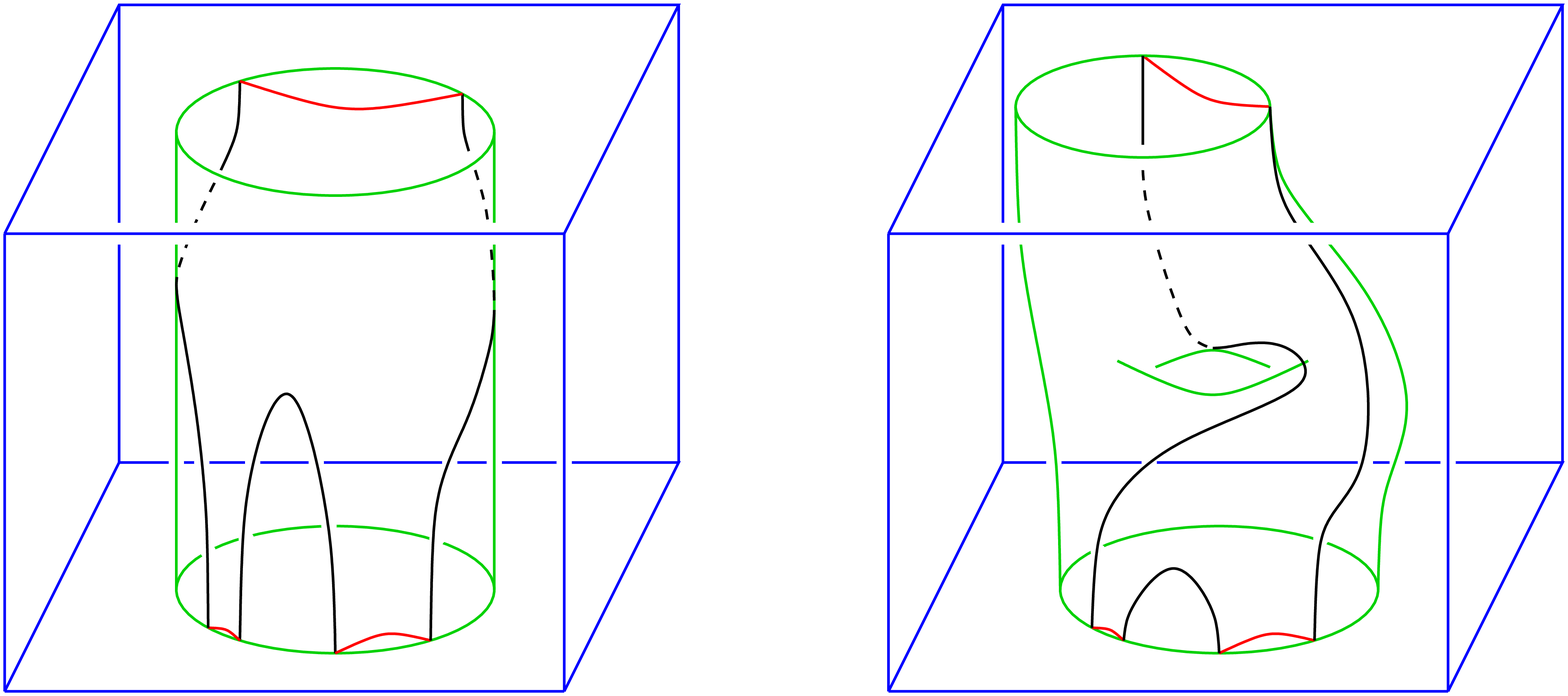}
\caption{Schematic picture of: $\R^3$ and $\R\times\R^3$ in blue; the Legendrians $\Lambda$, $\Lambda_+$, $\Lambda_-$ and Lagrangians $\R\times \Lambda$ and $L$ in green; the Reeb chords $a$, $b$, and $c$ in red; and the holomorphic disks in black, with a positive puncture at $a$ and negative punctures at $b$ and $c$.}
\label{fig:sft}
\end{figure}

The functor $\Phi$ associates to $(X,L)\in \op{Hom}(\Lambda_+,\Lambda_-)$ a DGA morphism
$$\Phi_{(X,L)}\colon\A(\Lambda_+)\to\A(\Lambda_{-}),$$
which is defined by a count of rigid punctured holomorphic disks in $X$ with boundary on $L$. (See Section~\ref{subsection: moduli of holomorphic disks} for more details.)  The right panel of Figure \ref{fig:sft} depicts a holomorphic disk that is counted in $\Phi_{(X,L)}$, noting that our statements hold not just in $\R\times\R^3$, but for more general exact Lagrangian cobordisms $(X,L)$ as well.

The usual compactness and gluing results for holomorphic disks imply that the cobordism maps satisfy the following properties (see \cite{E2} and Section \ref{sec:diffandcob}):

\begin{thm}
$\Phi:\mathfrak{cob}\to \mathfrak{dg}$ is a functor, i.e.,
\begin{enumerate}
\item If $(X,L)=id$, then $\Phi_{(X,L)}=id$.
\item If $(X_{20},L_{20})=(X_{21},L_{21})\cdot (X_{10},L_{10})$, then
\[
\Phi_{(X_{20},L_{20})}=\Phi_{(X_{10},L_{10})}\circ\Phi_{(X_{21},L_{21})}.
\]
\end{enumerate}
\end{thm}

\begin{thm}
If $(X_t,L_t)$, $0\le t\le 1$, is an isotopy of exact Lagrangian cobordisms from $\Lambda_+$ to $\Lambda_-$ with cylindrical ends, then the DGA maps $\Phi_{(X_0,L_{0})}$ and $\Phi_{(X_1,L_{1})}$ are chain homotopic.
\end{thm}

For simplicity we use $\F=\Z/2$-coefficients throughout, so that $\A(\Lambda_\pm)$ are algebras over $\F$ and $\Phi_{(X,L)}$ is an $\F$-algebra morphism. In particular, all curve counts are mod 2 counts. The calculations of $\Phi_{(X,L)}$ are expected to carry over to the setting of more general coefficient rings using the orientation scheme developed in \cite{EES3}, for example.

An {\em augmentation of $\Lambda$} is a DGA morphism $\varepsilon: \mathcal{A}(\Lambda)\to \mathcal{A}(\varnothing)=\F$.
Given an exact Lagrangian filling $(X,L)$ of $\Lambda$, the DGA morphism $\Phi_{(X,L)}$ is an augmentation of $\Lambda$; such augmentations will be called {\em geometric}.

\subsection{Main results}

The goal of this paper is to give Morse-theoretic and combinatorial descriptions of DGA maps induced by exact Lagrangian cobordisms $(X,L)\in \mbox{Hom}(\Lambda_+,\Lambda_-)$, where $X$ is the completion of $T^*F$.  This is done through an intermediary called a {\em Morse cobordism}, a compact exact Lagrangian submanifold $L^{\rm Mo}$ in $T^{\ast}F$ with a Legendrian lift $\tilde L^{\rm Mo}\subset J^{1}F$, which is described in more detail in Section~\ref{sec:lagcob}.
Here $J^{1}F\simeq T^{\ast}F\times\R$ denotes the $1$-jet space of $F$ with its standard contact form $d\zeta-\theta$, where $\zeta$ is the $\R$-coordinate. We assume that $\tilde L^{\rm Mo}\cap J^{1}F\big|_{\pa F}$ agrees with the Legendrian links $\Lambda_{\pm}\subset J^{1}\pa_{\pm} F\subset J^{1}F$ and all the Reeb chords of $\tilde L^{\rm Mo}$ are contained in $J^{1}F\big|_{\pa F}$ so that they are in a natural one-to-one correspondence with the Reeb chords of $\Lambda_{\pm}$.

In Section~\ref{sec:lagcob} we associate to $L^{\rm Mo}$ a $2$-parameter family of exact Lagrangian cobordisms $(X_{\delta},L_{\delta;\sigma})\in \op{Hom}(\Lambda_+,\Lambda_-)$, $\delta,\sigma>0$, where $X_{\delta}$ is a completion of $T^{\ast}F$, and prove the following:

\begin{lemma} \label{back and forth}
Given $(X,L)\in \mbox{Hom}(\Lambda_+,\Lambda_-)$, where $X$ is the completion of $T^*F$, there exists a Morse cobordism $L^{\rm Mo}$ in $T^*F$
such that the exact Lagrangian cobordisms $(X_{\delta},L_{\delta;\sigma})$ associated to $L^{\rm Mo}$ are exact Lagrangian isotopic to $(X,L)$.
\end{lemma}

As in \cite{E1} we define Morse flow trees of a Morse cobordism $L^{\rm Mo}\subset T^*F$ from $\Lambda_{+}$ to $\Lambda_{-}$. If $a$ is a Reeb chord of $\Lambda_{+}$ and $\bb$ a monomial in $\A(\Lambda_-)$, then let $\TT(a;\bb)$ denote the space of Morse flow trees of $\tilde L^{\rm Mo}$ with a positive puncture at $a$ and negative punctures (as well as homotopy data) determined by $\bb$; see Section \ref{sec:flowtrees} for details.

The following main technical result, proved in Section~\ref{subsection: trees to disks}, gives a Morse-theoretic expression for the DGA morphism induced by $(X_{\delta},L_{\delta;\sigma})$.

\begin{thm}\label{thm:flowtreecomp}
For all sufficiently small $\delta,\sigma>0$, $\Phi_{(X_{\delta},L_{\delta;\sigma})}\colon \A(\Lambda_+)\to\A(\Lambda_-)$ is given as follows: if $a\in \mathcal{C}(\Lambda_+)$ then
\[
\Phi_{(X_{\delta},L_{\delta;\sigma})}(a)  \ \ =\sum_{\dim(\TT(a;\bb))=0}|\TT(a;\bb)|\bb,
\]
where $|\TT(a;\bb)|$ is the mod $2$ count of points in the compact $0$-dimensional moduli space $\TT(a;\bb)$.
\end{thm}

By combining Lemma~\ref{back and forth} and Theorem \ref{thm:flowtreecomp}, we obtain a Morse-theoretic description of any $\Phi_{(X,L)}$, where $(X,L)\in \mbox{Hom}(\Lambda_+,\Lambda_-)$ and $X$ is the completion of $T^*F$. Although Theorem~\ref{thm:flowtreecomp} is stated for general $F$, in applications we take $F=\R\times[a,b]$. In particular we use Theorem \ref{thm:flowtreecomp} to find explicit combinatorial formulas for $\Phi_{(\R\times \R^3,L)}$ in the following special cases:
\begin{itemize}
\item[(a)] $L$ is induced by a Legendrian isotopy (see Sections \ref{sec:nomovescob} and \ref{sec:movescob});
\item[(b)] $L$ is a minimum cobordism, i.e., a disk bounding the standard unknot (see Definition \ref{def:mincob}); and
\item[(c)] $L$ is a saddle cobordism that is given by a $0$-resolution of a contractible Reeb chord in the Lagrangian projection  (see Definition \ref{def:saddlecob}).
\end{itemize}
The explicit formulas for (a) are given in Lemmas~\ref{lem:triple}, \ref{lem:death}, and \ref{lem:birth}; the formula for (b) is given in Lemma~\ref{lem:minimum}; and the formula for (c) is given in Proposition~\ref{prop:saddle}.

As an application of these formulas we prove the following result:

\begin{thm}\label{thm:appl}
Let $\Lambda_n$ be the Legendrian $(2,n)$-torus link in Figure \ref{fig:2n}, let $A_n=(2^{n+1}-(-1)^{n+1})/3$, and let $g_n=\frac{n-1}{2}$ if $n$ is odd and $g_n=\frac{n-2}{2}$ if $n$ is even. Then there are at least $A_n$ (resp.\ $A_n-1$) smoothly isotopic exact Lagrangian fillings of $\Lambda_n$ of genus $g_n$ that are pairwise non-isotopic through exact Lagrangian surfaces when $n$ is odd (resp.\ even).
\end{thm}

\begin{proof}
This is an immediate corollary of Proposition~\ref{prop: torus link}.
\end{proof}

\begin{figure}[ht]
\includegraphics[width=.7\linewidth]{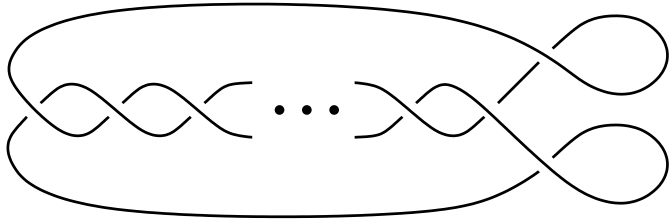}
\caption{Lagrangian diagram of a Legendrian $(2,n)$-torus knot or link, oriented as the closure of a positive braid on $2$ strands. The total number of crossings is $n+2$.}\label{fig:2n}
\end{figure}

Theorem \ref{thm:appl} is in sharp contrast to the situation of the Legendrian unknot~$U$ with rotation number $r(U)=0$ and Thurston--Bennequin invariant $tb(U)=-1$. By the isotopy uniqueness for Lagrangian planes in $\C^{2}$ that are standard at infinity \cite{EP}, there is a unique Lagrangian disk which fills $U$, up to exact Lagrangian isotopy.  (See Figure~\ref{fig:mincob} for such a disk.)

In Section~\ref{sec:Khovanov} we discuss some relationships with Khovanov homology~\cite{Kh}.

\s\n
{\em Outline of the paper.} In Section~\ref{sec:lagcob} we construct and relate the various forms of exact Lagrangian cobordisms that are used in this paper. In particular we prove Lemma~\ref{back and forth}.  We then review Legendrian contact homology in Section~\ref{section: Legendrian contact homology} and gradient flow trees from \cite{E1} in Section~\ref{sec:flowtrees}. Section~\ref{Sec:disktree} is the main technical part of this paper where we prove Theorem~\ref{thm:flowtreecomp}.  We then use Theorem~\ref{thm:flowtreecomp} to compute the DGA maps of elementary exact Lagrangian cobordisms in Section~\ref{Sec:DGAmaps}.  In Section~\ref{section: exact Lagrangian fillings} we collect some observations about exact Lagrangian fillings and augmentations. In Section~\ref{sec:examples} we give some applications including the proof of Theorem~\ref{thm:appl}. The connections with Khovanov homology are given in Section~\ref{sec:Khovanov}.

\section{Lagrangian cobordisms with Legendrian ends}
\label{sec:lagcob}

In this section we introduce the various forms of exact Lagrangian cobordisms that we use in this paper.

\subsection{Completion of $T^\ast F$} \label{subsection: completion}

As in Section \ref{Sec:intr}, let $F=(\R\times[a_-,a_+])\# \Sigma$, where $\Sigma$ is a compact surface and $a_+>a_->0$. Let $\theta$ be the canonical $1$-form on $T^{\ast}F$ and let $\omega=-d\theta$. Let $U_+=\R\times (a_+-\epsilon,a_+]$ be a neighborhood $\pa_+ F$ and let $(\xi_1,\eta_1,\xi_2,\eta_2)$ be the coordinates on $T^{\ast} F\big|_{U_{+}}$ such that $(\xi_1,\xi_2)$ are the coordinates on $U_+$ and $(\eta_1,\eta_2)$ are the dual coordinates on $\R^2$.  Then $\theta=\eta_1d\xi_1+\eta_2d\xi_2$ on $T^{\ast} F\big|_{U_{+}}$.

Consider the map
$$\Phi_{+}\colon (\log (a_+-\epsilon),\log a_+]\times\R^{3}\to T^{\ast} F\big|_{U_+},$$
$$(t,x,y,z) \mapsto \left( x , e^{t} y, e^{t}, z\right).$$
Then
\[
\Phi^{\ast}_{+}(-\theta)= -e^{t}(ydx + zdt)= e^{t}(dz-ydx)-d(e^{t}z).
\]
Hence $\Phi_{+}$ is a symplectomorphism with respect to $d(e^t(dz-ydx))$ on $(\log (a_+-\epsilon),\log a_+]\times\R^{3}$ and the canonical symplectic form $\omega=-d\theta$ on $T^{\ast} F\big|_{U_+}$.

Next, on $(\log(a_+-\epsilon),\infty)\times\R^{3}$, we consider the $1$-form
$$\alpha=e^{t}(dz-ydx)-d(\beta(t)e^{t}z),$$ where $\beta(t)=1$ for $t\in (a_+-\epsilon,a_+]$ and $\beta(t)=0$ for $t\gg a_+$. Then attaching $((\log(a_+-\epsilon),\infty)\times\R^{3},\alpha)$ to $T^{\ast}F$ using $\Phi_+$ yields an exact symplectic manifold which extends $(T^{\ast}F,-\theta)$.

Similarly, let $U_-$ be a neighborhood of $\pa_{-}F$ with coordinates $(\xi_1,\xi_2)\in \R\times [a_-,a_-+\epsilon)$.  We attach $(-\infty,\log(a+\epsilon))\times\R^{3}$ to $T^{\ast}F\big|_{U_-}$ using the map
\[
\Phi_{-}(t,x,y,z) = \left( x , e^{t} y, e^{t}, z\right).
\]
The result of these two attachments is an exact symplectic manifold $X$ with cylindrical ends, called the {\em completion of $T^\ast F$}.

\subsection{Conical cobordisms}

We now explain how to translate the notion of a cylindrical Lagrangian cobordism to the notion of a {\em conical Lagrangian cobordism} on $T^\ast F$.  Conical Lagrangian cobordisms serve as the intermediary between cylindrical Lagrangian cobordisms and Morse cobordisms, where the latter is particularly suited for holomorphic disk counting.

Consider the ends
\begin{align*}
\mathcal{E}_+(L)&=(\log(a_+-\epsilon),\infty)\times\Lambda_{+}\subset [\log(a_+-\epsilon),\infty)\times\R^{3},\\
\mathcal{E}_-(L)&=(-\infty,\log(a_-+\epsilon)]\times\Lambda_{-}\subset (-\infty,\log(a_-+\epsilon)]\times\R^{3},
\end{align*}
of a cylindrical Lagrangian cobordism $L$, where $\Lambda_{\pm}$ are Legendrian links with parametrizations
\begin{equation}\label{eq:legparam}
(x_{\pm}(s),y_{\pm}(s),z_{\pm}(s)),\quad s\in S_{\pm}
\end{equation}
and $S_{\pm}$ are closed $1$-manifolds. Then $\Phi_{\pm}(\mathcal{E}_\pm(L))\cap T^{\ast}F\big|_{U_\pm}$ are parametrized by
$$(x_{\pm}(s),\xi_2y_{\pm}(s),\xi_2,z_{\pm}(s)),$$
where $(s,\xi_2)\in S_{+}\times(a_+-\epsilon,a_+]$ and $(s,\xi_2)\in S_{-}\times [a_-,a_-+\epsilon)$, respectively.

\begin{defn}
If $\Lambda_+$ and $\Lambda_-$ are Legendrian links, then a Lagrangian submanifold $L^{\co}\subset T^{\ast}F$ is a {\em conical Lagrangian cobordism} from $\Lambda_+$ to $\Lambda_-$ if $L^{\co}\cap U_{\pm}$ admits a parametrization
\begin{equation} \label{eq:conicalend}
(x_{\pm}(s),(\xi_2-a_\pm') y_{\pm}(s),\xi_2,z_{\pm}(s)),
\end{equation}
where $a_-<a_+$ and $a_\pm'<a_\pm$. We do not require $a_->0$.
\end{defn}

Let $\hat F$ be the surface obtained from $F$ by gluing $E_-=\R\times (a_-',a_-]$ to $\pa_{-}F$ and $E_+=\R\times[a_+,\infty)$ to $\pa_{+}F$ so that $X\simeq T^*\hat F$.

\begin{defn}
Given a conical Lagrangian cobordism $L^{\co}\subset T^{\ast}F$, the {\em long conical Lagrangian cobordism $\hat L^{\co}$ corresponding to $L^{\co}$} is a Lagrangian submanifold of $T^{\ast}\hat F$ such that $\hat L^{\co}\cap T^{\ast}\hat F\big|_{E_-\cup E_+ }$ admits a parametrization as in \eqref{eq:conicalend}, where $(s,\xi_2)\in S_{+}\times[a_+,\infty)$ and $(s,\xi_2)\in S_{-}\times(a_-',a_-]$, respectively.
\end{defn}

\subsection{Morse cobordisms}

We now define {\em Morse cobordisms.} Let $\Lambda_{\pm}\subset\R^{3}$ be Legendrian links with parametrizations as in \eqref{eq:legparam}.

\begin{defn} \label{defn: Morse cobordism}
A Lagrangian submanifold $L^{\Mo}$ in $T^{\ast}F$ is a {\em Morse cobordism from $\Lambda_+$ to $\Lambda_-$} if
$L^{\rm Mo}\cap T^\ast F \big|_{U_+\cup U_-}$ admits parametrizations
\begin{equation}\label{eq:Morseend}
(x_\pm(s),\left(A_\pm \mp(\xi_2-a_\pm)^{2}\right)y_\pm(s),\xi_2, \mp 2(\xi_2-a_\pm)z_\pm(s)),
\end{equation}
where $(s,\xi_2)\in S_{+}\times (a_+-\epsilon,a_+]$ and $(s,\xi_2)\in S_{-}\times [a_-,a_-+\epsilon)$, respectively, and $A_\pm>0$.
\end{defn}

A Morse cobordism $L^{\Mo}$ has a Legendrian lift~$\tilde L^{\Mo}\subset J^{1}F$ which is unique up to $\R$-translation and the Reeb chords of $\tilde L^{\Mo}$ over $\pa_{+} F$ and $\pa_{-}F$ are canonically identified with the Reeb chords of $\Lambda_{+}$ and $\Lambda_{-}$, respectively.

\subsection{From Morse to conical} \label{subsubsection: morse to conical}

We explain how to pass from a Morse cobordism to a conical cobordism.

Let $h_{\delta}\colon [0,\epsilon)\to\R$ be a family of increasing functions parametrized by $\delta\in(0,\frac{1}{5}\epsilon)$ such that
\[
h_{\delta}(u)=
\begin{cases}
\delta u-\tfrac12\delta^{2} &\text{for }u\in [0,\frac12\delta],\\
u^{2} &\text{for } u\in [\delta,\epsilon].
\end{cases}
\]
Let $L^{\Mo}$ be a Morse cobordism and let $L'_{\delta}$ be an exact Lagrangian on $T^*F$ which agrees with $L^{\Mo}$ on $T^{\ast} F\big|_{F-(U_+\cup U_-)}$ and is parametrized by
\begin{equation}\label{eq:Morseend2}
\begin{split}
\left( x_+(s), \left(A_+ - h_{\delta}(a_+-\xi_2)\right)y_+(s),\xi_2,h_{\delta}'(a_+-\xi_2) z_+(s)\right)\\
\left( x_-(s), \left(A_- + h_{\delta}(\xi_2-a_-)\right)y_-(s),\xi_2,h_{\delta}'(\xi_2-a_-)z_-(s)\right)
\end{split}
\end{equation}
on $T^{\ast} F\big|_{U_{\pm}}$.

By adding ends $\R\times[a_-'(\delta),a_-]$ (for an appropriate $a_-'(\delta)$) and $\R\times[a_+,\infty)$ to $\pa_{\pm} F$ we obtain a surface $F_{\delta}$ and by adding conical ends to $L'_\delta$ we obtain a long conical Lagrangian $\hat L^{\co}_\delta\subset X_\delta:=T^{\ast} F_{\delta}$.

Given $\sigma>0$, let $s_\sigma$ be the {\em fiber scaling map}:
$$s_\sigma : T^*M \to T^*M,\quad   s_\sigma(q,p) = (q,\sigma p),$$
where $q\in M$ and $p\in T^*_qM$. Taking $M=F_\delta$, we define $L_{\delta;\sigma}:=s_\sigma(\hat L^{\co}_\delta)$. We refer to $(X_\delta,L_{\delta;\sigma})$ as the {\em conical exact Lagrangian cobordisms associated to $L^{\Mo}$}.

\begin{proof}[Proof of Lemma~\ref{back and forth}]
Given an exact Lagrangian cobordism $L$ with cylindrical Legendrian ends, we apply $\Phi_\pm$ to transform it into a long conical Lagrangian cobordism $\hat L^{\co}$. By truncating the ends of $\hat L^{\co}$ we obtain a conical Lagrangian cobordism $L^{\co}$.  Rounding $L^{\co}$ near its boundary yields a Morse cobordism $L^{\Mo}$ (details left to the reader).  Now applying the above procedure gives conical exact Lagrangian cobordisms $L_{\delta;\sigma}$.  One easily verifies that $\hat L^{\co}$ is exact Lagrangian isotopic to $L_{\delta;\sigma}$.
\end{proof}

\section{Legendrian contact homology}
\label{section: Legendrian contact homology}

In this section we review the Legendrian contact homology of links in standard contact $\R^{3}$, first from the SFT perspective of \cite{EGH} and then from the combinatorial perspective of \cite{Ch}.

\subsection{The graded algebra of Reeb chords} \label{subsection: graded algebra of Reeb}

We use the notation from Section \ref{Sec:intr}. Let $\Lambda$ be a Legendrian link in $(\R^3,\xi_0)$, $\Phi^t=\Phi_{R_{\alpha_0}}^{t}$ the time-$t$ flow of the Reeb vector field $R_{\alpha_0}=\pa_{z}$, and $\CC(\Lambda)$ the set of Reeb chords of $\Lambda$.

The $\alpha_0$-{\em action} of $c\in \CC(\Lambda)$ is given by
$$\mathfrak{A}_{\alpha_0}(c)=\int_c \alpha_0.$$
If $c^\pm$ are the endpoints of $c$, then
$$\mathfrak{A}_{\alpha_0}(c)=z(c^+)-z(c^-)>0.$$
Hence we also refer to $\mathfrak{A}_{\alpha_0}(c)$ as the {\em length of $c$}.

\begin{defn}[Chord genericity] \label{chord generic}
A Legendrian link $\Lambda\subset (\R^3,\xi_0)$ is {\em chord generic} if, for any $c\in \CC(\Lambda)$, $d\Phi^{\ell}(T_{c^{-}}\Lambda)$ is transverse to $T_{c^{+}}\Lambda$ in the contact plane $(\xi_0)_{c^{+}}$ at $c^{+}$ and $\ell=\mathfrak{A}_{\alpha_0}(c)$.
\end{defn}

We assume that $\Lambda$ is chord generic, since chord genericity can be attained by perturbing $\Lambda$ in its isotopy class. In particular, $\CC(\Lambda)$ is a finite set and is in one-to-one correspondence with the double points of the restriction of the Lagrangian projection $$\Pi_{\C}\colon \R^{3}\to\R^{2},\quad (x,y,z)\mapsto(x,y),$$
to $\Lambda$.

Let $\Lambda=\Lambda_1\cup\dots\cup\Lambda_k$, where $\Lambda_j$, $j=1,\dots,k$, are the connected components of $\Lambda$. Fix a reference point $p_j\in\Lambda_j$ for each $j$.  For each pair $i\ne j$, pick a path~$\delta_{ij}$ in $\R^{3}$ that connects $p_i$ to $p_j$ and a path of lines in the contact planes along $\delta_{ij}$ that connects $T_{p_i}\Lambda_i$ to $T_{p_j}\Lambda_j$.  When $i=j$, we let $\delta_{ii}$ be the constant path at $p_i$. For each Reeb chord $c\in\CC(\Lambda)$ with $c^{\pm}\in \Lambda_{i^{\pm}}$, $i^{\pm}\in\{1,\dots,k\}$, pick capping paths $\gamma^{\pm}_{c}$ in $\Lambda_{i^{\pm}}$ that connect $p_{i^{\pm}}$ to $c^{\pm}$.
We use the complex trivialization ${\frak T}$ of $\xi_0$ that is induced from the Lagrangian projection $\Pi_{\C}\colon\R^{3}\to\R^{2}$ followed by the standard identification $\R^{2}\simeq\C$. We also write homology classes in $H_1(\Lambda_j)\simeq \Z$ multiplicatively: the class which equals $m$ times the generator is denoted by $\tau^{m}_j$ with $\tau_j^{0}=1_{j}$.

Let $\Gamma_{\Lambda_j}$ be a loop of tangent lines of $\Lambda_j$ induced by a loop which traverses $\Lambda_j$ once in the direction determined by $\tau_j$. Then the {\em degree of $\tau_j$} is given as follows:
\[
|\tau_j|=\mu(\Gamma_{\Lambda_j}),
\]
where the Maslov index $\mu$ is measured with respect to ${\frak T}$. 

Let $c\in\CC(\Lambda)$ be a Reeb chord of length $\ell$. Let $\hat\Gamma_c$ denote the loop of tangent lines in $\xi_0$ which is the concatenation of the following four paths:
\begin{itemize}
\item[(i)]  the chosen path of tangent lines along $-\delta_{i^-i^+}$,
\item[(ii)]  the path of tangent lines of $\Lambda$ along $\gamma_{c}^{-}$,
\item[(iii)]  the smallest positive rotation in the contact plane $\xi_{c^{+}}$ that takes $d\Phi^{\ell}(T_{c^{-}}\Lambda)$ to $T_{c^{+}}\Lambda$, and
\item[(iv)]  the path of tangent lines of $\Lambda$ along $-\gamma_{c}^{+}$.
\end{itemize}
Here $-\gamma_c^+$ is $\gamma_c^+$ with the opposite orientation. Then define the {\em degree of $c$} as:
\[
|c|=\mu(\hat\Gamma_c)-1,
\]
where the Maslov index $\mu$ is measured with respect to ${\frak T}$.

The graded algebra $\A(\Lambda)$ underlying the Legendrian DGA of $\Lambda$ is the unital algebra over $\F$ which is freely generated by the Reeb chords in $\CC(\Lambda)$ and the homology classes $\tau_j^{\pm 1}\in H_1(\Lambda_j)$, $j=1,\dots,k$, with the grading given above.

\begin{rmk}\label{rmk:Morsegrading}
If $L^{\Mo}$ is a Morse cobordism in $T^{\ast} F$ from $\Lambda_{+}$ to $\Lambda_{-}$, then
the Reeb chords of $\tilde L^{\Mo}\subset J^{1}F=T^{\ast}F\times\R$ over $\pa_{-}F$ (resp.\ $\pa_+ F$) are in natural one-to-one correspondence with the Reeb chords of $\Lambda_-$ (resp.\ $\Lambda_+$). By abuse of notation, the Reeb chords of $\mathcal{C}(\tilde L^{\Mo})$ corresponding to $c_\pm \in \mathcal{C}(\Lambda_\pm)$ will be denoted by $c_\pm$.  Then $|c_-|_{\tilde L^{\Mo}}=|c_-|_{\Lambda_-}$ and $|c_+|_{\tilde L^{\Mo}}=|c_+|_{\Lambda_+}+1$, where $|\cdot|_{\tilde L^{\Mo}}$ refers to the grading of $\A(\tilde L^{\Mo})$ and $|\cdot|_{\Lambda_+}$ refers to the grading of $\A(\Lambda_{\pm})$.
\end{rmk}

\subsection{Moduli spaces of holomorphic disks} \label{subsection: moduli of holomorphic disks}

Consider an exact Lagrangian cobordism $(X,L)$ from $\Lambda_{+}$ to $\Lambda_{-}$ with cylindrical Legendrian ends. For cobordisms $(X,L)$, we use the natural maps $H_{1}(\Lambda_{\pm})\to H_{1}(L;\Z)$ and consider $\A(\Lambda_{\pm})$ as algebras generated by Reeb chords and a basis of $H_{1}(L;\Z)$. When we want to make the ``coefficient ring'' $R$ explicit we write $\A(\Lambda_\pm; R)$.

We describe $\A(\Lambda_+;\F[H_1(L;\Z)])$ more precisely: Pick a reference point $q_i$ in each component of $L$ and for each $p_j\in \Lambda_{+,j}$ a path $\delta_{q_ip_j}$ in $L$ from some $q_i$ to $p_j$. Given $c\in \mathcal{C}(\Lambda_+)$, we define $\overline{\gamma}_c^\pm$ as the concatenation of $\gamma_c^\pm$ with the appropriate $\delta_{p_iq_j}$. Then we repeat the construction of $\A(\Lambda_+)$ from Section~\ref{subsection: graded algebra of Reeb} with $\gamma_c^\pm$ replaced by $\overline{\gamma}_c^\pm$ and $H_1(\Lambda_+;\Z)$ replaced by $H_1(L;\Z)$. This gives us $\A(\Lambda_+;\F[H_1(L;\Z)])$; the construction of $\A(\Lambda_-;\F[H_1(L;\Z)])$ is similar.

Let $a\in\CC(\Lambda_{+})$ and $b_1,\dots,b_m\in\CC(\Lambda_-)$ be Reeb chords, let $\tau_0,\dots,\tau_m$ be elements of $H_1(L;\Z)$, and let
\[
\bb=\tau_0 b_1\tau_1 b_2\tau_2\dots \tau_{m-1} b_m\tau_{m},
\]
where $b_s^{-}$ and $b_{s+1}^{+}$ both lie on $\Lambda_{j_s}$ for $s\ge 1$; $a^{+}$ and $b_{1}^{+}$ both lie on $\Lambda_{j_0}$; and $b_m^{-}$ and $a^{-}$ both lie on $\Lambda_{j_m}$. Fix an almost complex structure~$J$ on $X$ which is {\em adjusted to the symplectic form on $X$}, i.e.,
\begin{itemize}
\item $J$ is compatible with the symplectic form;
\item $J$ is $\R$-invariant on the symplectization ends;
\item $J(\xi)=\xi$ and $J(\bdry_t)= R_{\alpha_0}$ on the symplectization ends.
\end{itemize}
If $c$ is a Reeb chord of $\Lambda_{\pm}$, then the product $\R_{\pm}\times c\subset \R_{\pm}\times\R^{3}$ of $c$ and a half-line $\R_{\pm}$ in the $\R$-direction in the symplectization ends is $J$-holomorphic. We call $\R_{\pm}\times c$ a {\em strip over the Reeb chord $c$.}

Define $\M^{(X,L);J}(a;\bb)$ as the moduli space of $J$-holomorphic disks \[u\colon (D_{m+1},\pa D_{m+1})\to (X,L),\] with the following properties:
\begin{enumerate}
\item $D_{m+1}$ is the closed unit disk with $m+1$ boundary points $\zeta_0,\dots,\zeta_m$ removed. Here $\zeta_0,\dots,\zeta_m$ are arranged in counterclockwise order around the boundary of the disk.
\item $u$ has a positive puncture at $\zeta_0$ where it is asymptotic to the strip over the Reeb chord $a$ at $+\infty$.
\item $u$ has a negative puncture at $\zeta_j$, $j>0$, where it is asymptotic to the strip over the Reeb chord $b_j$ at $-\infty$.
\item The loop in $L$ obtained by concatenating the capping path $\overline{\gamma}_{b_s}^{+}$, the image of the boundary segment from $\zeta_s$ to $\zeta_{s-1}$ under $u$, and the capping path $-\overline{\gamma}_{b_{s-1}}^{-}$ represents the class $\tau_s$. (Here we use the convention $b_0=b_{m+1}=a$.)
\end{enumerate}
The $J$ in the notation of the moduli space will often be suppressed. We will also refer to $u\in \M^{(X,L);J}(a;\bb)$ as a {\em $J$-holomorphic disk in $(X,L)$ from $a$ to $\bb$.}

Similarly we define $\mathcal{M}^{(X,L);J}(a_1,\dots,a_k;\bb_1,\dots,\bb_k)$ as the moduli space of $J$-holomorphic disks in $(X,L)$ with positive punctures at $a_1,\dots,a_k$ and negative punctures given by $\bb_1,\dots,\bb_k$ such that $a_1,\bb_1,a_2,\bb_2,\dots,a_k,\bb_k$ are in counterclockwise order on the boundary of the domain.

\begin{defn}
An {\em $(a+b+1)$-level broken disk} is a union $u_{-b}\cup \dots \cup u_a$, arranged from bottom to top, such that $u_j$, $j\not=0$, maps to $X_0=\R\times\R^3$, $u_0$ maps to $X$, each component of $u_j$, $j=-b,\dots,a$, has one positive puncture and zero or more negative punctures, and the positive punctures of $u_j$, $j=-b,\dots,a-1$, are in one-to-one correspondence with the negative punctures of $u_{j+1}$.
\end{defn}

\subsection{Energy}

Let $((X,d\beta),L)$ be an exact Lagrangian cobordism from $\Lambda_+$ to $\Lambda_-$ with cylindrical Legendrian ends. Let
$$X_+=[T,\infty)\times \R^3,\quad X_-=(-\infty,-T]\times \R^3$$ be the positive and negative ends of $X$ and let
$$X_0=X-int(X_+\cup X_-).$$
We write $\beta_0= \beta|_{X_0}$, $\beta_\pm= \beta|_{\bdry X_\pm}$, and $\omega_0$ for the exact $2$-form on $X$ which agrees with $d\beta_0$ on $X_0$ and with $d\beta_\pm$ on $X_\pm$.

\begin{defn}
If $u\in \M^{(X,L);J}(a;\bb)$, then we define its {\em $\omega_0$-energy} $E_{\omega_0}(u)$ and the {\em $\beta_\pm$-energy} $E_{\beta_\pm}(u)$ as follows:
$$E_{\omega_0}(u)= \int_u \omega_0,\quad E_{\beta_\pm}(u)=\sup_{\phi_-}\int_{u|_{X_\pm}} d( \phi_\pm(t)) \wedge\beta_\pm,$$
where $\phi_+: [T,\infty)\to \R^{\geq 0}$ (resp.\ $\phi_-:(-\infty,-T]\to \R^{\geq 0}$) is a smooth nondecreasing function such that $\phi_+(T)=0$ and $\displaystyle\lim_{t\to \infty} \phi_+(t)=1$ (resp.\ $\displaystyle\lim_{t\to-\infty}\phi_-(t)=0$ and $\phi_-(-T)=0$).
\end{defn}

The following lemma is immediate (cf.\ \cite[Lemma B.3]{E2} for details):

\begin{lemma}\label{lemma: energy bound}
If $u\in \M^{(X,L);J}(a;\bb)$, then $E_{\omega_0}(u)\leq \mathfrak{A}(a)$ and $E_{\beta_\pm}(u)\leq \mathfrak{A}(a)$.
\end{lemma}

\subsection{Fredholm theory and compactness}

The {\em Fredholm index} $\op{ind}(u)$ of $u\in \mathcal{M}^{(X,L);J}(a_1,\dots,a_k;\bb_1,\dots,\bb_k)$ is the full expected dimension count which includes conformal variations.  If $(X,L)$ is $\R$-invariant, then $u$ is {\em rigid} if $\op{ind}(u)=1$; otherwise $u$ is {\em rigid} if $\op{ind}(u)=0$.

\begin{lemma} \label{fredholm index}
If $u\in \mathcal{M}^{(X,L);J}(a_1,\dots,a_k;\bb_1,\dots,\bb_k)$, then
\begin{equation} \label{eqn: fredholm index}
\op{ind}(u)=\sum_{i=1}^k (|a_i|-|\bb_i|)-1+k.
\end{equation}
\end{lemma}

\begin{proof}
Follows from \cite[Theorem A.1]{CEL}.
\end{proof}

\begin{lemma}\label{lem:dim+tv}
For generic $J$, the moduli space $\M^{(X,L);J}(a;\bb)$ is a transversely cut out manifold of dimension $\op{ind}=|a|-|\bb|$
and admits a natural compactification by multiple-level broken disks.
\end{lemma}

\begin{proof}
The transversality can be achieved by perturbing $J$ in the contact planes near the positive puncture as in \cite[Lemma 4.5(1)]{EES4}. The compactness statement is a consequence of \cite[Section~11.3]{BEHWZ}; see also \cite[Lemma B.4]{E2}.
\end{proof}

In general, there are transversality problems when one counts holomorphic disks with more than one positive puncture, since such disks may be multiply-covered.

In the special case of a cylindrical cobordism $L=\R\times\Lambda$ in the symplectization~$X_{0}=\R\times\R^{3}$, we write $\M^J(a;\bb)=\M^{(X_0,\R\times\Lambda);J}(a;\bb)$. Since $J$ is $\R$-invariant, $\R$ acts on $\M^J(a;\bb)$ and we write $\widehat{\M}^J(a;\bb)=\M^J(a;\bb)/\R$ for the reduced moduli space.

\begin{cor}\label{cor:bdrysympl}
For generic $J$, a reduced moduli space $\widehat{\M}^J(a;\bb)$ of dimension $0$ is a compact $0$-dimensional manifold and the boundary of a compactified reduced moduli space of dimension~$1$ is either empty or consists of two-level broken disks $u_0\cup u_1$ with $\op{ind}(u_0)=\op{ind}(u_1)=1$.
\end{cor}

\begin{proof}
Follows from Lemma \ref{lem:dim+tv}.
\end{proof}

When $(X,L)$ is a nontrivial cobordism, we have the following description of the low-dimensional moduli spaces.

\begin{cor}\label{cor:bdrycob}
For generic $J$, the moduli spaces $\M^{(X,L);J}(a;\bb)$ of dimension $0$ are compact $0$-dimensional manifolds and the boundary of a compactified moduli space of dimension~$1$ is either empty or consists of two-level broken disks $u_0\cup u_1$ or $u_{-1}\cup u_0$, where $u_0\in \M^{(X,L);J}(a;\bb)$ with $\op{ind}(u_0)=0$ and $u_{\pm 1}\in\M^J(a',\bb')$ with $\op{ind}(u_{\pm 1})=1$.
\end{cor}

\begin{proof}
Follows from Lemma \ref{lem:dim+tv}.
\end{proof}

If $(X'',L'')$ and $(X',L')$ are exact Lagrangian cobordisms with cylindrical Legendrian ends, where the positive end of $(X',L')$ is the same as the negative end of $(X'',L'')$, then these cobordisms can be concatenated by truncating the ends $\{t\geq\rho\}\subset X'$ and $\{t\leq -\rho\}\subset X''$ for $\rho\gg 0$ and gluing the remaining pieces together. The resulting Lagrangian cobordism is denoted by $(X_{\rho},L_{\rho})$.

\begin{cor}\label{cor:stretch}
For all sufficiently large $\rho$, there is a natural one-to-one correspondence between rigid holomorphic disks in $(X_{\rho},L_{\rho})$ and two-level broken disks $u_0\cup u_1$, where $u_0$ is a rigid disk in $(X',L')$ and $u_1$ is a rigid disk in $(X'',L'')$.
\end{cor}

\begin{proof}
Follows from Lemma \ref{lem:dim+tv} and a stretching argument; see \cite[Section 11.3]{BEHWZ} or \cite[Lemma 3.13]{E2}.
\end{proof}

\subsection{Differentials and cobordism maps}
\label{sec:diffandcob}

The differential in the Legendrian DGA is defined by counting disks in $0$-dimensional reduced moduli spaces in trivial cobordisms. More precisely, if $\Lambda\subset\R^{3}$ is a Legendrian link, then the differential of $\A(\Lambda)$ is the $\F$-linear map $\pa\colon\A(\Lambda)\to\A(\Lambda)$ which satisfies the Leibniz rule, maps homology generators (i.e., monomials of the group ring $\F[H_1(\Lambda;\Z)]$) to $0$, and is defined as follows on the Reeb chord generators:
\[
\pa a =\sum_{\dim(\widehat{\M}(a;\bb))=0}|\widehat{\M}(a;\bb)|\bb.
\]

\begin{lemma}
The map $\pa\colon\A(\Lambda)\to\A(\Lambda)$ is a differential (i.e., $\pa^{2}=0$) of degree~$-1$.
\end{lemma}

\begin{proof}
This is a direct consequence of Corollary \ref{cor:bdrysympl}.
\end{proof}

Similarly, if $(X,L)$ is an exact Lagrangian cobordism from $\Lambda_+$ to $\Lambda_-$, then it induces a DGA map
$$\Phi_{(X,L)}\colon \A(\Lambda_+;\F[H_1(L;\Z)])\to\A(\Lambda_-;\F[H_1(L;\Z)]),$$
defined as follows: on the homology generators $\Phi_{(X,L)}=id$ and on the Reeb chord generators $a$,
\[
\Phi_{(X,L)}(a)=\sum_{\dim(\M^{(X,L)}(a;\bb))=0}|\M^{(X,L)}(a;\bb)|\bb.
\]
We write $\Phi_{(X,L)}^{A}$ when we want to emphasize the ``coefficient ring'' $A$.

\begin{lemma}
The map $\Phi_{(X,L)}$ is a chain map (i.e.~$\pa_-\circ\Phi_{(X,L)}=\Phi_{(X,L)}\circ\pa_+$, where $\pa_{\pm}$ is the differential on $\A(\Lambda_{\pm})$) of degree $0$. Furthermore, if $(X,L)$ is obtained from concatenating $(X',L')$ and $(X'',L'')$, then $$\Phi_{(X,L)}^{\F[H_1(L;\Z)]}=\Phi_{(X'',L'')}^{\F[H_1(L;\Z)]}\circ\Phi_{(X',L')}^{\F[H_1(L;\Z)]}.$$
\end{lemma}

\begin{proof}
This is a direct consequence of Corollaries \ref{cor:bdrycob} and \ref{cor:stretch}.
\end{proof}

Next consider a $1$-parameter family of cobordisms $(X_t,L_t)$, $t\in[0,1]$, together with almost complex structures $J_{t}$. Assume that moduli spaces determined by $(X_0,L_0)$ and $(X_1,L_1)$ are transversely cut out so that the cobordism maps $\Phi_{(X_0,L_0)}$ and $\Phi_{(X_1,L_1)}$ are well-defined.

\begin{lemma}\label{lem:chainhomotopy}
The DGA maps $\Phi_{(X_0,L_0)}$ and $\Phi_{(X_1,L_1)}$ are chain homotopic, i.e., there exists a degree $+1$ map $K\colon\A(\Lambda_{+})\to\A(\Lambda_-)$ such that
\begin{equation}\label{eq:chhomtpy}
\Phi_{(X_1,L_1)}+\Phi_{(X_0,L_0)} =\Omega_K\circ\pa_+ + \pa_-\circ\Omega_K,
\end{equation}
where $\Omega_{K}$ is $\F$-linear and is defined as follows on monomials:
\[
\Omega_{K}(c_1\dots c_m)=\sum_{j=1}^{m}\Phi_{(X_1,L_1)}(c_1\dots c_{j-1})K(c_j)\Phi_{(X_0,L_0)}(c_{j+1}\dots c_m).
\]
\end{lemma}

\begin{proof}
The lemma follows from \cite[Lemma B.15]{E2}. Since the terminology of \cite[Lemma B.15]{E2} is slightly different from that of this paper, we sketch the argument, referring the reader to \cite[Section B.6]{E2} for details.

Consider the moduli space \[\MM(a;\bb)=\coprod_{t\in[0,1]}\MM^{(X_t,L_t);J_t}(a;\bb),\]
where $(X_t,L_t,J_t)$, $t\in[0,1]$, is generic. If $\op{ind}(a;\bb)=-1$, then:
\begin{enumerate}
\item $\MM^{(X_t,L_t);J_t}(a;\bb)\not=\varnothing$ if and only if $t=t_j$ for some $t_j$, $j=1,\dots,\ell$, where $0<t_1<t_2<\dots<t_\ell<1$;
\item $\#\MM^{(X_{t_j},L_{t_j});J_{t_j}}(a;\bb)=1$, where $\#$ denotes the cardinality.
\end{enumerate}
For simplicity we assume that $\coprod_{\op{ind}(a;\bb)=-1}\MM(a;\bb)=\{v\}$ and $v$ occurs at time $t^*$.

We then use $v$ to construct the chain homotopy. A subtlety is that we need to treat the gluing of a broken disk $u_{-b}\cup\dots\cup u_a$, where $u_0$ contains several copies of $v$. This is done using a {\em time-ordered, domain-dependent} abstract perturbation scheme, which we describe now.

Let $u:D_{m+1}\to X$ be a map that is close to breaking into $u_{-b}\cup\dots\cup u_a$.  More precisely,
\begin{itemize}
\item let $c_{i1},\dots,c_{ij_i}$, $i=-b,\dots,a$, be the Reeb chords of the negative ends of $u_i$, arranged in counterclockwise order around the boundary of the disk obtained by pre-gluing $u_i\cup\dots \cup u_a$;
\item let $A_{ij}\subset D_{m+1}$, $i=-b+1,\dots,a$, $j=1,\dots,j_i$, be rectangles biholomorphic to $[0,1]\times[\tau_{ij},\tau_{ij}']$ for some $\tau_{ij}<\tau_{ij}'$ such that $\{0,1\}\times [\tau_{ij},\tau_{ij}']\subset \bdry D_{m+1}$;
\item let $A_{-bj}\subset D_{m+1}$, $j=1,\dots,j_{-b}$, be half-infinite strips biholomorphic to $[0,1]\times(-\infty,\tau_{-bj}']$ for some $\tau_{-bj}'$ such that $\{0,1\}\times(-\infty,\tau_{-bj}']\subset \bdry D_{m+1}$; similarly, let $A_{a+1,1}\subset D_{m+1}$ be the half-infinite strip corresponding to the positive end of $u_a$;
\item the $A_{ij}$ are disjoint and $u|_{A_{ij}}$ is close to a strip over $c_{ij}$; and
\item let $D_{m+1}-\cup_{i,j} A_{ij}=B_{-b}\sqcup \dots \sqcup B_{a}$ such that $u|_{B_i}$, $i=-b,\dots,a$, is close to $u_i$ with ends truncated.
\end{itemize}
We will refer to the subscript $ij$ in $c_{ij}$ or $A_{ij}$ as {\em a subscript at a negative end of $u_i$}.

Next choose $\epsilon>0$ small and $N>0$ large and let
$$\sigma: \CC(\Lambda_+)\cup \CC(\Lambda_-)\to (0,\epsilon)$$
be a map such that ${\sigma(c)\over \sigma(c')}> N$ whenever $\mathfrak{A}(c)>\mathfrak{A}(c')$. We then inductively construct the vector
$$(t_{ij}^*),\quad i=-b,\dots,a+1,\quad j=1,\dots,j_i$$
as follows: First set $t_{a+1,1}^*=0$. Suppose we have constructed $t_{i1}^*<\dots< t_{ij_i}^*$. Then $t_{i-1,j}^*$, $j=1,\dots,j_{i-1}$, is given by $t^*_{\tau(i-1,j)}+\sigma(c_{\tau(i-1,j)})(p_{i-1,j}-1)$, where $\tau(i-1,j)$ is the subscript at the positive end of the component $\widetilde u$ of $u_{i-1}$ which has $(i-1,j)$ as a subscript at the negative end, and $(i-1,j)$ is the $p_{i-1,j}$th negative end of $\widetilde u$, arranged in counterclockwise order.

We use a {\em domain-dependent} almost complex structure $\mathfrak{J}_t$ to define the $\overline\bdry$-operator, i.e., $\mathfrak{J}_t$ depends smoothly on $x\in D_{m+1}$: we set $\mathfrak{J}_t=J_{t+t_{ij}^*}$ on $A_{ij}$ and we extend $\mathfrak{J}_t$ smoothly to $B_i$ so that $\mathfrak{J}_t(x)=J_{t+t(x)}$, $x\in B_i$, and $t^*_{i+1,b}\leq t(x)\leq  t^*_{ia}$, where $(i+1,b)$ is the subscript at the positive end of the component $\widetilde u$ corresponding to $x\in B_i$ and $(i,a)$ is the subscript of the last negative end of $\widetilde u$.

Let $\MM'(a;\bb)$ be the ``perturbation of $\MM(a;\bb)$'', obtained using the $1$-parameter family $\mathfrak{J}_t$, $t\in[\epsilon,1-\epsilon]$.
Observe that new $\op{ind}=-1$ disks might get created when the perturbation is turned on. For example, an $\op{ind}=-1$ disk might be created when gluing $u_1$ with $\op{ind}(u_1)=-1$ and two negative ends to $u_0$ consisting of two copies of $v$. The spacing for $(t_{ij}^*)$ is chosen to ensure that a broken disk of $\bdry\MM'(a;\bb)$ will use an $\op{ind}=-1$ curve at most once. This is a consequence of the fact that if $\op{ind}(\widetilde{u})=-1$, then the domain-dependent $\mathfrak{J}_t$ for $\widetilde{u}$ is close to $J_{t^*}$ at the positive end of $\widetilde{u}$ by Gromov compactness.  Hence we obtain
\begin{equation} \label{moomin}
\bdry\MM'(a;\bb)\simeq \MM^{(X_0,L_0);J_0}(a;\bb)\sqcup \MM^{(X_1,L_1);J_1}(a;\bb)\sqcup \mathcal{M}'',
\end{equation}
where $\mathcal{M}''$ is the set of broken disks of the form $u_0\cup u_1$ or $u_{-1}\cup u_0$ with $\op{ind}(u_0)=-1$ and $\op{ind}(u_{-1})=\op{ind}(u_1)=0$.

Finally, if we define $K(c)$, $c\in \Lambda_+$, as follows:
\[
K(c)=\sum_{\op{ind}(c;\bb)=-1}\left|\MM'(c;\bb)\right|\bb,
\]
then \eqref{eq:chhomtpy} holds in view of \eqref{moomin}.
\end{proof}

\subsection{Combinatorial description of the Legendrian DGA}

Let $\Lambda\subset (\R^{3},\xi_0)$ be a Legendrian link. Consider the Lagrangian projection $\Pi_{\C}\colon\R^{3}\to\R^{2}$, $(x,y,z)\mapsto (x,y)$.  Let $J_0$ be the adjusted almost complex structure on $\R\times \R^3$ which is induced from the complex structure on $\R^{2}\simeq\C$ via $\Pi_\C$, i.e., $J_0$ maps:
\begin{equation}\label{e:translinvJ}
\pa_x\mapsto \pa_y +y\pa_t,\quad \pa_y \mapsto -\pa_x - y \pa_z,\quad \pa_t \mapsto \pa_z,\quad \pa_z\mapsto -\pa_t.
\end{equation}
Then a $J_0$-holomorphic disk in $(\R\times\R^3,\R\times\Lambda)$ from $a$ to $\bb$ projects to a holomorphic disk in $\R^{2}$ whose boundary maps to the Legendrian knot diagram $\Pi_{\C}(\Lambda)$ and whose punctures ``map to $a$ and $\bb$''. On the other hand, by \cite[Section 2.7]{EES2}, each such disk in $\R^{2}$ lifts to a unique $\R$-invariant family of disks in $\R\times\R^{3}$.

Thus we are led to the following combinatorial description of Legendrian contact homology, which is the version of the theory originally defined by Chekanov \cite{Ch}: The set of Reeb chords $\mathcal{C}(\Lambda)$ is in one-to-one correspondence with the set of double points of $\Pi_{\C}(\Lambda)$. The Maslov index used to define the grading on $\A(\Lambda)$ is the Maslov index (i.e., twice the rotation number) of $\Pi_{\C}(\Lambda)$. The rigid holomorphic disks correspond to immersed polygons with convex corners at double points of $\Pi_{\C}(\Lambda)$. The sign of a puncture is positive (resp.\ negative) if the boundary orientation of the disk points towards (resp.\ away from) the double point along the lower strand and points away from (resp.\ towards) the double point along the upper strand. Consequently, the computation of the differential is reduced to the combinatorial problem of finding all the immersed polygons with convex corners, boundary on $\Pi_\C(\Lambda)$, and exactly one positive puncture.

\section{Gradient flow trees of Morse cobordisms}
\label{sec:flowtrees}

\subsection{Gradient flow trees}

We briefly summarize the definitions and notation for {\em gradient flow trees} (or {\em flow trees} for short) from \cite[Sections 2.2 and 3.1]{E1}. The reader is warned that some of the terminology is different in this paper. We will mainly describe the case of a $2$-dimensional Legendrian submanifold $\tilde L\subset J^{1}F$, leaving the simpler case of a $1$-dimensional Legendrian link $\Lambda\subset\R^{3}=J^{1}\R$ to Remark~\ref{leg link}.

Let
$$\Pi_{J^0F}: J^{1}F\to J^{0}F=F\times\R,$$
$$\Pi_{T^*F}: J^1 F\to T^*F,\quad \Pi_{F}: J^{1}F\to F,$$
be the front projection, the Lagrangian projection, and the projection to the base. Let $\tilde L\subset J^{1} F$ be a $2$-dimensional Legendrian submanifold. For generic $\tilde L$ in its isotopy class, the singular set $\Sigma\subset\tilde L$ of $\Pi_{F}$ consists of cusp edges and swallowtails; see \cite[Equation (2-1) and Remark 2.5]{E1}. In particular, for generic $\tilde L$, $\Pi_F(\Sigma)$ admits a stratification:
\[
\Pi_F(\Sigma)=\Sigma_1\supset(\Sigma_2^{\rm dbl}\cup\Sigma_2^{\rm sw}),
\]
where $\Sigma_1$ consists of all the critical values of $\Pi_F$ and has codimension $1$ in $F$, $\Sigma_2^{\rm dbl}$ is the set of transverse double points of $\Sigma_1$, $\Sigma_2^{\rm sw}$ is the set of swallowtail points, and both $\Sigma_2^{\rm dbl}$ and $\Sigma_2^{\rm sw}$ have codimension $2$ in $F$. We write
$$\Sigma_1^{\circ}=\Sigma_1\setminus(\Sigma_2^{\rm dbl}\cup\Sigma_2^{\rm sw}),\quad \Sigma_0^{\circ}=\Pi_F(\tilde L)\setminus\Sigma_1.$$

On a small neighborhood $U_q$ of a point $q\in\tilde L-\Sigma$, $\tilde L|_{U_q}$ is given as the $1$-jet of a {\em height function} $f\colon \Pi_F(U_q)\to\R$.

\begin{defn}
Fix a Riemannian metric on $F$ which agrees with the standard flat metric near the boundary and let $\nabla$ denote the corresponding gradient operator. Let $I$ be a compact interval or a half-line $[0,\infty)$.
\begin{enumerate}
\item A {\em flow line of $\tilde L$ in $F$} is a curve $\gamma\colon I\to F$, together with {\em $1$-jet lifts}
$$\gamma_i: I\to \tilde L\subset J^1 F,\quad \gamma=\Pi_F\circ \gamma_i,\quad i=1,2,$$ such that:
\begin{itemize}
\item for each $t_0\in int(I)$ there is a neighborhood $N_i(t_0)$ of $\gamma_i(t_0)\subset \tilde L$ which is given by a height function $f_i:\Pi_F(N_i(t_0))\to \R$; and
\item on $\Pi_F(N_1(t_0)\cap N_2(t_0))$, $\gamma$ satisfies the downward gradient equation
$$\dot\gamma(t) = -\nabla(f_1-f_2)(\gamma(t)).$$
\end{itemize}
\item The {\em cotangent lifts} of a flow line $(\gamma,\gamma_1,\gamma_2)$ are maps
$$\Pi_{T^*F}\circ\gamma_i:I\to L\subset T^\ast F,\quad i=1,2.$$
\item The {\em flow orientation} is the choice of orientation on the $1$-jet lifts $\gamma_1$, $\gamma_2$ such that locally $\Pi_F\circ\gamma_1$ is oriented by $-\nabla(f_1-f_2)$ and $\Pi_F\circ \gamma_2$ by $-\nabla(f_2-f_1)$.
\end{enumerate}
\end{defn}

\begin{rmk} \label{exception}
A flow line is an immersion except when it is a constant map to a critical point of $f_1-f_2$.
\end{rmk}

\begin{defn}
A \emph{source tree} $\Gamma$ is a metric space which is either a copy of $\R$ or a tree with edges that are compact intervals or half-lines $[0,\infty)$. The endpoints of the edges are the {\em vertices} and the points at infinity of the half-lines or lines are the {\em $1$-valent punctures}.  We also designate certain vertices as the {\em interior punctures}.   Furthermore, at any vertex of the tree there is a cyclic ordering of adjacent edges. Let $E(\Gamma), V(\Gamma), P(\Gamma)$ be the sets of edges, vertices, and punctures (both $1$-valent and interior) of $\Gamma$.
\end{defn}

\begin{defn}[Gradient flow tree] \label{def:flowtree}
Let $\Gamma$ be a source tree.  A {\em gradient flow tree} of $\tilde L$ is a map $\gamma\colon\Gamma\to F$ (called a {\em flow tree map}), together with $1$-jet lifts $\gamma^e_1,\gamma^e_2$ for each $e\in E(\Gamma)$, such that:
\begin{enumerate}
\item for $e\in E(\Gamma)$, $(\gamma|_e,\gamma^e_1,\gamma^e_2)$ is a flow line of $\tilde L$;
\item a neighborhood of each $1$-valent puncture maps to a flow line into or out of a critical point of some height function difference $f_i-f_j$, i.e., a Reeb chord;
\item the {\em cotangent lift $C(\gamma)$ of $\gamma$}, i.e., the union of the closures of the cotangent lifts $\Pi_{T^*F}\circ \gamma_i^e$, $i=1,2$, over all $e\in P(\Gamma)$, is a closed oriented curve;
\item a vertex $v\in V(\Gamma)$ is an interior puncture if and only if there are adjacent cotangent lifts $\Pi_{T^*F}\circ \gamma_{i_1}^{e_1}$ and $\Pi_{T^*F}\circ \gamma_{i_2}^{e_2}$ of $C(\gamma)$ at $v$ such that $\gamma_{i_1}^{e_1}(v)\not=\gamma_{i_2}^{e_2}(v)$ and are connected by a Reeb chord.
\end{enumerate}
\end{defn}

The restriction of a flow tree map to a half-infinite edge may be constant by Remark~\ref{exception}.  See Figure~\ref{punctures} for examples of $1$-valent punctures and interior punctures.

The cotangent and $1$-jet lifts are oriented using the flow orientation. Near a puncture $p$, one of the $1$-jet lifts of the edge $e$ adjacent to $p$ is incoming (i.e., oriented towards the critical point), and the other is outgoing (i.e., oriented away from the critical point).

\begin{defn} \label{def: sign of puncture}
A puncture of $\gamma$ is {\em positive} if the height function of the incoming $1$-jet lift is smaller than that of the outgoing; otherwise it is {\em negative}.
\end{defn}

The $1$-jet lift $\tilde C(\gamma)$ of a flow tree $\gamma$ is a union of paths which connect the endpoints of Reeb chords, just like the boundary of a holomorphic disk. In particular, $\tilde C(\gamma)$ determines $(a_1,\dots,a_k;\bb_1,\dots,\bb_k)$, where $a_1,\dots,a_k\in \mathcal{C}(\tilde L)$ correspond to the positive punctures and $\bb_1,\dots,\bb_k\in \mathcal{A}(\tilde L)$ correspond to the negative punctures.

\begin{defn}
Two flow trees
$$(\gamma:\Gamma\to F;\gamma_i^e, i=1,2, e\in E(\Gamma)),~(\gamma':\Gamma'\to F; (\gamma')_i^e,i=1,2,e\in E(\Gamma'))$$
are {\em equivalent} if there is an isometry $\phi:\Gamma\stackrel\sim\to \Gamma'$ such that $\gamma=\gamma'\circ\phi$, $\gamma_i^e=(\gamma')_i^{\phi(e)}\circ\phi|_e$ for all $e\in E(\Gamma)$, and the cyclic orders around the vertices are preserved.
\end{defn}

\begin{defn}
A {\em partial flow tree} $\gamma:\Gamma\to F$ satisfies Definition~\ref{def:flowtree} with the exception of Condition~(2). A $1$-valent vertex of $\Gamma$ where the cotangent lift of $\gamma$ is non-closed is called a {\em special vertex}. We write $V(\Gamma)$ for the set of vertices of $\Gamma$ and $SV(\Gamma)\subset V(\Gamma)$ for the subset of special vertices. The sign of a special vertex is defined in the same way as Definition~\ref{def: sign of puncture}.
\end{defn}

For example, if we cut a flow tree in two along an edge, then we obtain two partial flow trees, each with one special puncture.

\begin{rmk} \label{leg link}
The definition of a flow tree of a $1$-dimensional Legendrian link is exactly the same as above. Here the situation is simpler in that $\Sigma_1=\Sigma_1^{\circ}$, i.e., generically the only singularities of the front projection are isolated cusps.
\end{rmk}


\subsection{Formal dimension}

Let $\TT_{\tilde L}(a_1,\dots,a_k;\bb_1,\dots,\bb_k)$ be the space of flow trees from $a_1,\dots,a_k$ to $\bb_1,\dots,\bb_k$ modulo equivalence. The formal dimension of $\gamma\in \TT_{\tilde L}:=\TT_{\tilde L}(a_1,\dots,a_k;\bb_1,\dots,\bb_k)$ is given by
\begin{equation} \label{eqn: dim of tree}
\dim(\gamma)=\dim(\TT_{\tilde L})=\sum_{i=1}^k(|a_i|_{\tilde L}-|\bb_i|_{\tilde L})+k-2.
\end{equation}
This formal dimension agrees with the Fredholm index of disks from $a_1,\dots,a_k$ to $\bb_1,\dots,\bb_k$ in the cotangent bundle $T^*F$, which is one less than the Fredholm index in the symplectization $\R\times J^1F$, given by Equation~\eqref{eqn: fredholm index}. This is due to the extra $\R$-translation.

We rewrite Equation~\eqref{eqn: dim of tree} in terms of Morse-theoretic data as in \cite[Definition~3.4]{E1}: Let $n$ be the dimension of the Legendrian submanifold (which in our applications will be $1$ or $2$).  Let $P_\pm(\gamma)$ be the set of positive/negative punctures of $\gamma$ and let $R(\gamma)$ be the set of vertices of $\gamma$ that are not punctures.

\begin{defn}[Morse index $I(p)$]
If $p\in P_\pm(\gamma)$, then let $c$ be the Reeb chord corresponding to $p$ and let $f^+$, $f^-$ be the height functions for the two sheets containing $c^\pm$. Then the {\em Morse index} $I(p)$ is the Morse index of the height function difference $f^+-f^-$ at $p$.
\end{defn}

\begin{defn}[Maslov content $\mu(r)$]
Let $r\in R(\gamma)$. If $x\in \Sigma$ is a cusp point over $r$ which lies in the $1$-jet lift $\tilde C(\gamma)$, then let $\tilde \mu(x)=1$ (resp.\ $\tilde \mu(x)=-1$) if the incoming arc of $\tilde C(\gamma)$ at $x$ lies on the upper (resp.\ lower) sheet and the outgoing arc of $\tilde C(\gamma)$ lies on the lower (resp.\ upper) sheet.  The {\em Maslov content of $r$} is $$\mu(r)=\sum_x \tilde \mu (x),$$ where the sum is over all cusp points $x\in\Sigma$ over $r$.
\end{defn}

Then we have the following:
\begin{equation} \label{eqn: dim of tree alt}
\dim(\gamma)=(n-3) +\sum_{p\in P_+(\gamma)} (I(p)-(n-1)) -\sum_{q\in P_-(\gamma)}(I(q)-1) +\sum_{r\in R(\gamma)}\mu(r).
\end{equation}

\subsection{Generic flow trees on $\tilde L$}

Suppose $\tilde L$ is generic. Then by \cite[Theorem 1.1]{E1} $\TT_{\tilde L}:=\TT_{\tilde L}(a;\bb)$ is a stratified space with strata that are manifolds and its top-dimensional stratum has dimension $\dim(\TT_{\tilde L})$. In particular, $\dim(\TT_{\tilde L})<0$ implies $\TT_{\tilde L}=\varnothing$ and $\dim(\TT_{\tilde L})=0$ implies that $\TT_{\tilde L}$ is a finite collection of flow trees that are transversely cut out. (A flow tree is transverse if it satisfies the {\em preliminary transversality conditions} from \cite[Section 3.1.1]{E1} and the conditions indicated in the proof of \cite[Proposition 3.14]{E1}.)

A flow tree $\gamma:\Gamma\to F$ in $\TT_{\tilde L}$ of dimension zero only has vertices of valency $\le 3$ whose neighborhoods are given as follows:
\begin{enumerate}
\item $1$-valent punctures;
\item $2$-valent interior punctures;
\item ends;
\item $Y_0$-vertices;
\item $Y_1$-vertices; and
\item switches.
\end{enumerate}
See Figures~\ref{punctures}, \ref{vertices}, and \ref{switch} for some examples. Observe that a flow line limits to a cusp edge corresponding to an end in finite time; cf.\ \cite[Lemma~2.8]{E1}.

\begin{figure}[ht]
\begin{center}
\psfragscanon
\psfrag{A}{$-,{\pi\over 2}$}
\psfrag{B}{$+,{\pi\over 2}$}
\psfrag{C}{$-,{3\pi\over 2}$}
\psfrag{D}{$+,{3\pi\over 2}$}
\psfrag{E}{$-,{\pi\over 2}$}
\psfrag{F}{$+,{\pi\over 2}$}
\includegraphics[width=1.0\linewidth]{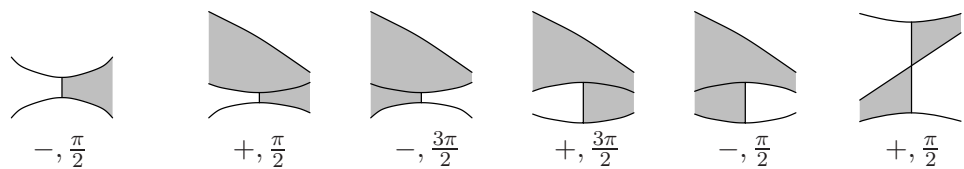}
\end{center}
\caption{A $1$-valent puncture (on the far left) and several $2$-valent interior punctures, drawn in the front projection. The sign of the puncture and the angle made in the Lagrangian projection are given.}
\label{punctures}
\end{figure}

\begin{figure}[ht]
\begin{center}
\includegraphics[width=0.6\linewidth]{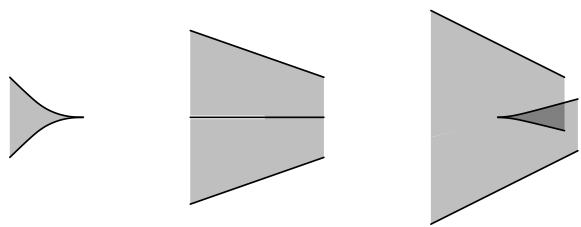}
\end{center}
\caption{From left to right, an end, a $Y_0$-vertex, and a $Y_1$-vertex, drawn in the front projection.}
\label{vertices}
\end{figure}

\begin{figure}[ht]
\begin{center}
\includegraphics[width=0.45\linewidth]{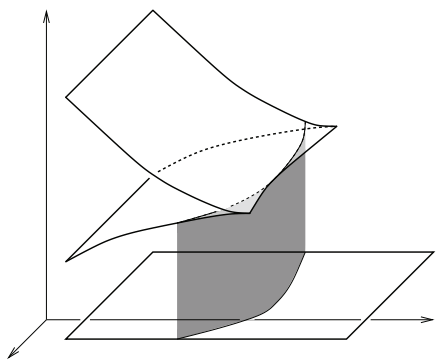}
\end{center}
\caption{A switch, drawn in the front projection.}
\label{switch}
\end{figure}

\subsection{Flow trees in Morse cobordisms}

Let $L^{\Mo}$ be a Morse cobordism in $T^*F$ from $\Lambda_+$ to $\Lambda_-$.

\begin{lemma}\label{lem:bdrytrees}
The $1$-jet lift of a flow tree $\gamma:\Gamma\to F$ of $\tilde L^{\Mo}\subset J^1F$ is contained in $J^{1}F\big|_{\pa F}$ if and only if $\gamma(x)\in \bdry F$ for some $x\in \Gamma$. Furthermore, flow trees of $\tilde L^{\Mo}$ that are contained in $J^{1}F\big|_{\pa_{+}F}$ and $J^{1}F\big|_{\pa_{-}F}$ are in natural one-to-one correspondence with flow trees of $\Lambda_-$ and $\Lambda_+$, respectively.
\end{lemma}

\begin{proof}
This follows from Definition~\ref{defn: Morse cobordism}.
\end{proof}

\n {\em Regularity.} If we view a flow tree in $\TT_{\Lambda_+}(a;\bb)$ as a flow tree $\gamma: \Gamma\to \bdry_+F$ in $\TT_{\tilde L^{\Mo}}(a;\bb)$, then
\begin{equation} \label{eqn: dim of tree part 2}
\dim \TT_{\tilde L^{\Mo}}(a;\bb)=\dim \TT_{\Lambda_+}(a;\bb)+1-\#\bb,
\end{equation}
where $\# \bb$ is the number of ends of $\bb$. If $\# \bb>1$ and $\dim \TT_{\Lambda_+}(a;\bb)=0$, then there exist flow trees $\gamma:\Gamma\to\bdry_+F$ of $\tilde L^{\Mo}\big|_{\pa_{+}F}$ that are not generic as flow trees of $\tilde L^{\Mo}$. This indicates that generically the Reeb chords of $\tilde L^{\Mo}$ in $J^{1}F\big|_{\pa_{+}F}$ would not be ``aligned'', i.e., would not all lie above $\bdry_+F$.

However, for the purposes of this paper, there is no need to perturb out of this situation: A priori we may have a broken flow tree $\gamma_0\cup\gamma_1$ (i.e., a flow tree with more than one level) from $a\in\mathcal{C}(\Lambda_+)$ to $\bb\in \mathcal{A}(\Lambda_-)$, where $\gamma_1$ is a tree from $a$ to $\bb'$ which is contained in $\tilde L^{\Mo}\big|_{\pa_{+}F}$ and $\gamma_0$ is the union of trees from $\bb'$ to $\bb$.  Since $\dim\TT_{\Lambda_+}(a;\bb')\geq 0$ and each of the $\#\bb'$ components of $\gamma_0$ has $\dim\geq 0$,
\begin{align*}
\dim (\gamma_0\cup\gamma_1) &\geq  \dim \TT_{\tilde L^{\Mo}}(a;\bb')+\#\bb'\\
&\geq  \dim\TT_{\Lambda_+}(a;\bb')+\#\bb'\geq 1.
\end{align*}
Therefore, by perturbing outside a neighborhood of $J^{1}F\big|_{\pa F}$, the flow trees of $\dim(\TT_{\tilde L}(a;\bb))=0$ that are not entirely contained in $J^{1}F\big|_{\pa_{+}F}$ are transversely cut out.

\section{Holomorphic disks and flow trees}
\label{Sec:disktree}

The goal of this section is to prove Theorem~\ref{thm:flowtreecomp}, i.e., the correspondence between rigid holomorphic disks and rigid Morse flow trees. The proof is an extension of the corresponding results in the compact case from \cite{E1} to immersed exact Lagrangian submanifolds with cylindrical or conical ends.

This section is organized as follows: In Section~\ref{subsection: morse vs conical} we show that certain rigid flow trees of a Morse cobordism $L^{\rm Mo}$ are in bijection with the rigid long conical flow trees (see Definition~\ref{defn:LC}) of the associated long conical cobordisms $L_{\delta;\sigma}$. In Section~\ref{subsection: ac and Leg links} we make a small perturbation of $L_{\delta;\sigma}$ as in \cite[Section 4]{E1} that allows us to better control holomorphic disks under scaling and introduce an almost complex structure $J_{\delta;\sigma}$ on $X_\delta$ which agrees with the almost complex structure on $T^*F$ induced by a Riemannian metric on $F$ in the sense of \cite[Section 4.4]{E1} and is adjusted to the symplectization of $(\R^3,\alpha_0)$ at the ends.  We then establish a key subharmonic estimate in Section~\ref{subsection: subharm} which allows us to gain control of holomorphic disks under rescaling and thereby prove the convergence of rigid holomorphic disks to rigid flow trees as $\sigma\to 0$ in Section~\ref{subsection: disks to trees}. In Section~\ref{subsection: trees to disks} we construct rigid holomorphic disks near rigid flow trees and show that the construction captures all rigid disks.

\subsection{Morse cobordisms, conical cobordisms, and flow trees} \label{subsection: morse vs conical}

We use the notation from Section \ref{sec:lagcob}. Consider a Morse cobordism $L^{\Mo}\subset T^{\ast}F$. In Section~\ref{subsubsection: morse to conical} we associated a family of long conical cobordisms $(X_\delta,L_{\delta;\sigma})$ to $L^{\Mo}$, parametrized by $\delta>0$ and $\sigma>0$ small.

A {\em Reeb chord flow line} is a flow line $\{x\}\times [a_+,\infty)\subset F_\delta$ or $\{x\}\times (a_-'(\delta),a_-]\subset F_\delta$ which corresponds to a cylinder over a Reeb chord in $\mathcal{C}(\Lambda_\pm)$.

\begin{defn} \label{defn:LC}
A {\em long conical (LC) flow tree $\gamma:\Gamma\to F_\delta$ of $L_{\delta;\sigma}$} satisfies the conditions of Definition \ref{def:flowtree} with (2) and (3) replaced by:
\begin{enumerate}
\item[(2')] a neighborhood of each puncture $p\in P(\Gamma)$ (i.e., a half-line) either limits to a critical point of some height function difference $f_i-f_j$ or is asymptotic to a Reeb chord flow line; and
\item[(3')] if we compactify $T^*F_\delta$ by attaching $\R^3\times\{a_-'(\delta),\infty\}$ and compactify the cotangent lift of $\gamma$ to $\overline\gamma$ by attaching suitable $c\times\{a_-'(\delta)\}$ or $c\times\{\infty\}$, where $c\in \mathcal{C}(\Lambda_\pm)$, then $\overline{\gamma}$ is a closed oriented curve.
\end{enumerate}
\end{defn}

\begin{lemma}\label{lem:long_Morse}
There is a bijection between rigid flow trees of $L^{\Mo}$ that are not contained in $T^{\ast} F\big|_{\pa F}$ and rigid LC flow trees of $L_{\delta;\sigma}$.
\end{lemma}

\begin{proof}
The lemma follows from constructing a diffeomorphism $\phi_\delta: F-\bdry F\stackrel\sim\to F_\delta$ which maps the gradient flow lines of $L^{\rm Mo}$ to the gradient flow lines of $L_{\delta;\sigma}$ after reparametrizations. (Note that fiber scaling does not alter gradient flow lines.)

The following model calculation can be generalized to give $\phi_\delta$: Let $I\subset \R$ be an interval and let $f=f_1-f_2: I\to \R$ be a height function difference.  Consider the height function difference $(1-t^2)f(s)$ on $I\times(-\varepsilon,0)$ with coordinates $(s,t)$ and the height function difference $\tau f(s)$ on $I\times (0,\infty)$ with coordinates $(s,\tau)$. Their gradients are
$$((1-t^2)\nabla f(s), -2tf(s)), \quad (\tau\nabla f(s),f(s)),$$ which are directed by
$$X_1:=\left(\nabla f(s), {2t\over t^2-1} f(s)\right),\quad X_2:=\left(\nabla f(s),{1\over \tau} f(s)\right).$$
We are looking for a function $\tau=g(t)$ such that
$$\phi: I\times (-\varepsilon,0)\to I\times (0,\infty),$$
$$(s,t)\mapsto (s,g(t)),$$
is a diffeomorphism and satisfies $\phi_* X_1=X_2$. Solving the differential equation ${2t\over t^2-1} {d\tau\over dt}= {1\over \tau}$ gives $g(t)=\sqrt{t^2/2 -\log |t| +c}$, and for an appropriate $c$ the desired condition holds. Finally observe that $\phi_0$ is independent of the height function difference $f=f_1-f_2$.
\end{proof}

Given $\epsilon>0$, let $N_{\epsilon,\pm}'$ denote an $\epsilon$-neighborhood of the Reeb chord endpoints in $\Lambda_\pm$ and let $N_{\epsilon}=(N'_{\epsilon,+}\times [a_+,\infty)) \cup (N'_{\epsilon,-}\times(a'_-(\delta),a_-]$.

\begin{cor} \label{cor: canonicalpunctures}
If $\delta>0$ is sufficiently small, then the restriction of any rigid LC flow tree of $L_{\delta;\sigma}$ to $X_\delta- T^*F$ is contained in $N_{C\delta}$ for some $C>0$.
\end{cor}

\begin{proof}
As $\delta\to 0$, the diffeomorphism $\phi_\delta^{-1}$ maps $F_\delta -F$ to smaller and smaller neighborhoods of $\bdry F$.

We first prove the corollary with $N_\epsilon$ instead of $N_{C\delta}$ for any sufficiently small $\epsilon>0$. If this weaker version of the corollary does not hold, then there is a sequence of rigid flow trees in $L^{\rm Mo}$ which converges to a broken flow tree with some level in $T^*F|_{\bdry F}$, which necessarily has dimension at least one. This is a contradiction, which proves the corollary with $N_\epsilon$ instead of $N_{C\delta}$.

Now fix $\varepsilon>0$. If $\gamma$ is a rigid flow tree of $L^{\rm Mo}$ and $\hat \gamma$ is the corresponding rigid LC flow tree of $L_{\delta;\sigma}$, then $\gamma$ and $\hat\gamma$ agree on
$$\check F^\delta:=F- \R\times([a_-,a_-+\delta]\cup[a_+-\delta,a_+]).$$
Moreover the endpoints of $\gamma$ and $\hat\gamma$ on $\bdry \check F^\delta$ are contained in $N_{\epsilon,\pm}'$ for $\delta>0$ sufficiently small. The corollary then follows by looking at the gradient of a height function difference of $L^{\rm Mo}$ near a critical point and the continuation of $\gamma|_{\check F^\delta}$ as it approaches the critical point.
\end{proof}

\subsection{Almost complex structures and deformed Legendrian links} \label{subsection: ac and Leg links}

In this subsection we introduce small deformations of $\Lambda_{\pm}$ and $L_{\delta;\sigma}$ and construct a family of almost complex structures on $X_{\delta}$.

\subsubsection{Deforming Legendrian links} \label{subsubsection: deforming legendrian links}

By a preliminary small Legendrian isotopy, we may assume that the Lagrangian projection of $\Lambda_{\pm}$ consists of straight lines near its double points. We then perturb $\Lambda_{\pm}$ slightly in a neighborhood of all its Reeb chords by making the $z$-coordinates of $\Lambda_\pm$ constant in a small neighborhood of the Reeb chord endpoints without changing the Lagrangian projection. The resulting link $\tilde\Lambda_\pm$ is not necessarily Legendrian.  Similarly, we perturb the ends of $L_{\delta;\sigma}$ so that the resulting totally real submanifold $\tilde L_{\delta;\sigma}$ has conical ends over $\tilde\Lambda_\pm$.

\begin{lemma} \label{lemma: nearby moduli spaces}
Let $J_0$ be the almost complex structure on $\R\times \R^3$ given by Equation~\eqref{e:translinvJ}. If $\Lambda_\pm$ are generic, then for $\tilde \Lambda_\pm$ sufficiently close to $\Lambda_\pm$:
\begin{enumerate}
\item There is a diffeomorphism
$$\mathcal{M}^{(\R\times \R^3,\R\times\Lambda_\pm);J_0}(a;\mathbf{b})\simeq \mathcal{M}^{(\R\times \R^3,\R\times\tilde\Lambda_\pm);J_0}(a;\mathbf{b}).$$
\item Consider DGAs associated to $\tilde\Lambda_{\pm}$ in direct analogy with the DGAs associated to $\Lambda_{\pm}$. Then $(1)$ implies that the DGA of $\tilde\Lambda_{\pm}$ is canonically isomorphic to that of $\Lambda_{\pm}$. There are totally real cobordisms from $\tilde\Lambda_\pm$ to $\Lambda_\pm$ (and vice versa) which induce the identity map on DGAs.
\end{enumerate}
\end{lemma}

\begin{proof}
(1) is immediate from the fact that in both cases the holomorphic disks are determined up to translation by their projections to $\C$. (2) is similar. Note that deforming only the $z$-coordinate in the cobordisms commutes with the projection to $\C$. Hence the only rigid disks of the cobordism are trivial strips.
\end{proof}

\subsubsection{The almost complex structure $J$}

Let $g$ be a Riemannian metric on $F$ which is flat near $\bdry F$ and let $g_\delta$ be its extension to $F_\delta$ which restricts to the standard flat metric on $\R\times(a'_-(\delta),a_-]$ and $\R\times [a_+,\infty)$. Let $J$ be the almost complex structure on $X_\delta$ which is compatible with $g_\delta$ in the sense of \cite[Section 4.4]{E1}.
Although the precise definition of $J$ is not important here, $J$ is given by
\[
\pa_{\xi_j}\mapsto\pa_{\eta_j},\quad \pa_{\eta_j}\mapsto-\pa_{\xi_j},\quad j=1,2.
\]
on the flat parts $T^*(\R\times(a'_-(\delta),a_-])$ and $T^*(\R\times [a_+,\infty))$.

We now apply the coordinate change
$$\Phi_+: [0,\infty)\times \R^3 \to T^*(\R\times [1,\infty)),$$
$$(t,x,y,z)\mapsto (x,e^{t}y,e^{t},z)$$
from Section~\ref{subsection: completion}. (After composing with a $\xi_2$-translation, we may assume that $a_+=1$.) At the positive symplectization end the almost complex structure $J$ can be written as follows:
$$ \pa_x\mapsto e^{-t}\pa_y, \quad \pa_y\mapsto -e^{t}\pa_x,\quad \pa_t\mapsto e^{t}\pa_z - e^{t} y\pa_x,\quad \pa_z\mapsto -e^{-t}\pa_t + e^{-t}y\pa_y.$$
The situation for the negative symplectic end is similar.

\subsubsection{The family of almost complex structures $J_{\delta;\sigma}$}

Next we define a family of almost complex structures $J_{\delta; \sigma}$ on $X_{\delta}$ that interpolates between $J$ and the almost complex structure $J_0$ given by Equation~\eqref{e:translinvJ}.

First consider the positive symplectization end which we take to be $[0,\infty)\times \R^3$ without loss of generality. Fix a sufficiently large constant $K\gg 1$.  For $0\le t\le \tfrac12K \sigma$ and $|y|\leq 1$, $J_{\delta;\sigma}$ is defined as follows:
\begin{equation}\label{e:defJfirst}
\begin{split}
\pa_x\mapsto k(t)^{-1}\pa_y,&\quad
\pa_y\mapsto -k(t)\pa_x,\quad
\pa_t\mapsto k(t)\pa_z - k(t)b(t)y\pa_x,\\
&\pa_z\mapsto -k(t)^{-1}\pa_t + k(t)^{-1} b(t)y \pa_y,
\end{split}
\end{equation}
where $k(t)=e^{t}$ near $t=0$ and $k(t)=1$ near $t=\frac12 K\sigma$, and $b(t)=1$ near $t=0$ and $b(t)=0$ near $t=\tfrac12 K\sigma$. Here $k(t)$ and $b(t)$ (as well as $a(t)$ below) depend on $\sigma>0$. Hence $J_{\delta;\sigma}$ is given as follows near $t=\frac12 K\sigma$:
\begin{equation}\label{e:defJsecond}
\pa_x\mapsto \pa_y,\quad
\pa_y\mapsto -\pa_x,\quad
\pa_t\mapsto \pa_z,\quad
\pa_z\mapsto-\pa_t.
\end{equation}
For $\tfrac12 K\sigma\le t\le K\sigma$ and $|y|\leq 1$, $J_{\delta;\sigma}$ is defined as follows:
\begin{equation}\label{e:defJthird}
\pa_x\mapsto\pa_y+a(t)y\pa_t,\quad
\pa_y\mapsto-\pa_x-a(t)y\pa_z,\quad
\pa_t\mapsto\pa_z,\quad
\pa_z\mapsto-\pa_t,
\end{equation}
where $a(t)=0$ near $t=\tfrac12 K\sigma$ and $a(t)=1$ near $t= K\sigma$.

\begin{lemma}
On the region $\{0\leq t\leq K\sigma, |y|\leq 1\}$, $J_{\delta;\sigma}$ is tamed by the canonical symplectic form $\omega$ on $T^*F_\delta$
\end{lemma}

\begin{proof}
We verify the lemma for $0\leq t\leq {1\over 2}K\sigma, |y|\leq 1$. By the identification $\Phi_\pm$ from Section~\ref{subsection: completion} we may assume that $\omega=e^t(dt\wedge (dz-ydx)+dx\wedge dy)$. We compute that:
\begin{align*}
& e^{-t}\omega(a_1\bdry_x +a_2\bdry_y +a_3\bdry_t+a_4\bdry_z, J_{\delta;\sigma}(a_1\bdry_x +a_2\bdry_y +a_3\bdry_t+a_4\bdry_z))\\
=&(dx\wedge dy+dt\wedge (dz-ydx))(a_1\bdry_x +a_2\bdry_y +a_3\bdry_t+a_4\bdry_z,a_1 k(t)^{-1}\bdry_y-a_2 k(t)\bdry_x \\
& + a_3(k(t) \bdry_z-k(t)b(t)y\bdry_x) + a_4(-k(t)^{-1}\bdry_t +k(t)^{-1}b(t)y \bdry_y))\\
=& a_1^2 k(t)^{-1} + a_1 a_4 k(t)^{-1} b(t) y + a_2^2 k(t) + a_2a_3k(t)b(t)y\\
& + a_3^2k(t) + a_2 a_3 k(t)y + a_3^2 k(t) b(t)y^2 + a_4^2k(t)^{-1}-a_1a_4k(t)^{-1}y.
\end{align*}
Since $|y|\leq 1$ and we may take $|b(t)-1|\leq 1$, the term $a_1^2+(b(t)-1)y a_1a_4 +a_4^2$ is positive. Similarly, since we may take $|b(t)+1|\leq 2$, the term $a_2^2+(b(t)+1)y a_2a_3 +a_3^2$ is also positive.  The positivity of the whole expression follows.
\end{proof}

For $0\leq t\leq K\sigma$ and $|y|\geq 1$, we arbitrarily extend $J_{\delta;\sigma}$ such that $J_{\delta;\sigma}$ is tamed by the canonical symplectic form $\omega$ on $T^*F_\delta$.  The precise form of $J_{\delta;\sigma}$ on this region is not important since the holomorphic curves that we consider do not enter this region by Sections~\ref{subsection: monotonicity} and \ref{subsection: subharm}.

Finally, $J_{\delta;\sigma}$ is $t$-translation invariant for $t\geq  K\sigma$:
\begin{equation}\label{e:defJfourth}
\pa_x\mapsto\pa_y+y\pa_t,\quad
\pa_y\mapsto-\pa_x-y\pa_z,\quad
\pa_t\mapsto\pa_z,\quad
\pa_z\mapsto-\pa_t.
\end{equation}

Similarly, we define $J_{\delta;\sigma}$ on the negative end using a completely analogous interpolation in a slice of width $K\sigma$.

\subsection{A priori energy bounds and monotonicity} \label{subsection: monotonicity}

\begin{lemma} \label{energy bound}
If $\tilde L_{\delta;\sigma}$ is sufficiently close to $L_{\delta;\sigma}$, then there exists $C>1$ such that
$$E_{\omega_0}(u)\leq \mathfrak{A}(a),\quad E_{\beta_\pm}(u)\leq C \mathfrak{A}(a)$$
for any $u\in \mathcal{M}^{(X_{\delta},\tilde L_{\delta;\sigma});J_{\delta;\sigma}}(a;{\bf b})$.
\end{lemma}

\begin{proof}
If $v\in \mathcal{M}^{(X_{\delta},L_{\delta;\sigma});J_{\delta;\sigma}}(a;{\bf b})$, then $E_{\omega_0}(v)\leq \mathfrak{A}(a)$ and $E_{\beta_\pm}(v)\leq \mathfrak{A}(a)$ by Lemma~\ref{lemma: energy bound}. Let $u\in \mathcal{M}^{(X_{\delta},\tilde L_{\delta;\sigma});J_{\delta;\sigma}}(a;{\bf b})$. The first inequality $E_{\omega_0}(u)\leq \mathfrak{A}(a)$ is immediate from the definition of $\tilde L_{\delta;\sigma}$ in Section~\ref{subsubsection: deforming legendrian links}.  By the proof of Lemma~\ref{lemma: nearby moduli spaces}, there is a diffeomorphism
$$\phi:\mathcal{M}^{(X_{\delta},L_{\delta;\sigma});J_{\delta;\sigma}}(a;{\bf b})\stackrel\sim\to\mathcal{M}^{(X_{\delta},\tilde L_{\delta;\sigma});J_{\delta;\sigma}}(a;{\bf b})$$
such that $\phi(u)$ is $C^1$-close to $u$. This implies that $E_{\beta_\pm}(u)\leq C \mathfrak{A}(a)$.
\end{proof}

Next we recall the monotonicity lemma for holomorphic curves; see for example \cite[Proposition 4.3.1]{Si}. Let $h_\delta$ be a metric which agrees with the standard Euclidean metric on the ends $[0,\infty)\times \R^3$ and $(-\infty,0]\times \R^3$ and is commensurate with the metric on $T^*F$ induced by $g$. All the distances with be measured with respect to $h_\delta$. In particular, $B(p,r)$ is the ball of radius $r$ around $p\in X_{\delta}$ with respect to $h_\delta$.

\begin{lemma}[Monotonicity lemma] \label{lemma: monotonicity}
There exist $r_0>0$ and $C>0$ such that
\begin{equation}\label{e:monint}
E_{\omega_0}(u)+E_{\beta_+}(u)+E_{\beta_-}(u) \ge C\cdot r^2
\end{equation}
for any $0<r<r_0$, $J_{\delta,\sigma}$-holomorphic map $u\colon M\to X_{\delta}$ with $u(0)=p$ and $u(\pa M)\subset \pa B(p,r)$.
\end{lemma}

We use the monotonicity lemma to obtain the following a priori $\Ordo(\sigma^{1/2})$ bounds on $|y|$ and $|z|$, where $|y|$ and $|z|$ are measured with respect to the standard Euclidean metric on $\R^3$.  In Section~\ref{subsection: subharm} we improve the $O(\sigma^{1/2})$ bound that comes from monotonicity to an $O(\sigma)$ bound that comes from the maximum principle for harmonic functions.

\begin{lemma} \label{lemma: bounds}
If $u\in \mathcal{M}^{(X_{\delta},\tilde L_{\delta;\sigma}); J_{\delta;\sigma}}(a;\mathbf{b})$, then $|y\circ u|=\Ordo(\sigma^{1/2})$ and $|z\circ u|=\Ordo(\sigma^{1/2})$ on the region $\{0\leq t\leq K\sigma\}$.
\end{lemma}

\begin{proof}
Arguing by contradiction, suppose there are sequences $\sigma_i\to 0$, $c_i\to\infty$, $u_i\in \mathcal{M}^{(X_{\delta},\tilde L_{\delta;\sigma_i}); J_{\delta;\sigma_i}}(a;\mathbf{b})$, and $p_i$ in the domain of $u_i$, such that $|y\circ u_i|(p_i)> c_i\sigma_i^{1/2}$ and $u_i(p_i)\in \{0\leq t\leq K\sigma_i\}$. Then there is a ball $B(q_i,{c_i\over 3}\sigma_i^{1/2})$ such that:
\begin{enumerate}
\item[(i)] $u_i$ passes through $q_i$, where $q_i\in \{0\leq t\leq K\sigma_i\}$;
\item[(ii)] $B(q_i,{c_i\over 3}\sigma^{1/2})\cap \tilde L_{\delta;\sigma_i}=\varnothing$;
\end{enumerate}
Here (ii) follows from the fact that the distance between $\tilde L_{\delta;\sigma}$ and the $0$-section is in $O(\sigma)$. By the Monotonicity Lemma~\ref{lemma: monotonicity},
$$E_{\omega_0}(u)+E_{\beta_+}(u)+E_{\beta_-}(u) \geq C c_i^2 \sigma/9.$$ On the other hand, by Lemma~\ref{energy bound}, $E_{\omega_0}(u)+E_{\beta_+}(u)+E_{\beta_-}(u)$ is bounded above by a constant times $\sigma$, a contradiction.  This proves the lemma for $|y\circ u|$. The argument for $|z\circ u|$ is the same.
\end{proof}

{\em In what follows we will work with $\tilde\Lambda_\pm$ and $\tilde L_{\delta;\sigma}$ and omit the tildes in the notation.}

\subsection{Subharmonicity} \label{subsection: subharm}

Define a function $p^{2}_{\delta;\sigma}\colon X_{\delta}\to\R$ as follows:
\[
p^{2}_{\delta;\sigma}(q)=
\begin{cases}
\eta^{2}_1+\eta_{2}^{2} &\text{ for }q\in T^{\ast} F,\\
(k(t)y)^{2}+z^{2} &\text{ for }q\notin T^{\ast} F,
\end{cases}
\]
where $k(t)=k_\sigma(t)$ is the interpolation function used in the definition of $J_{\delta;\sigma}$. Here $k(t)$ was defined on $[0,{1\over 2}K\sigma]$ and is extended by setting $k(t)=1$ on $t\geq {1\over 2}K\sigma$.

\begin{lemma}\label{lem:subharm}
There exist $K\gg 1$, $\sigma_0>0$, and $\delta_0>0$ such that, for any $0<\sigma<\sigma_0$, $0<\delta<\delta_0$, and $u\in \mathcal{M}^{(X_{\delta},L_{\delta;\sigma}); J_{\delta;\sigma}}(a;\mathbf{b})$, the function $p^2_{\delta;\sigma}\circ u$ is subharmonic with respect to the metric $g_\delta$. In particular, $p^2_{\delta;\sigma}\circ u$ achieves its maximum on $\bdry D$.
\end{lemma}

\begin{proof}
On the region $t\geq K\sigma$, $k(t)=1$ is constant. By a straightforward calculation the $y$- and $z$-coordinates are harmonic. This in turn implies that $y^2$ and $z^2$ are both subharmonic. On the other hand, $\eta_1^2+\eta_2^2$ is subharmonic on $T^{\ast}F$ by \cite[Section 5.1.2 and Lemma 5.5]{E1}. It remains to verify the subharmonicity on the interpolation region $0\leq t\leq K\sigma$.

Consider the region $0\le t\le\frac12K\sigma$. The Cauchy--Riemann equations
$$\pa_{\tau_1} u + J_{\delta;\sigma}\pa_{\tau_2} u =0$$
imply that:
\begin{align}\label{library}
\pa_{\tau_1}(k(t)y)&=k(t)\pa_{\tau_1} y+ yk'(t)\pa_{\tau_1} t\\
\notag &=-k(t)\left(k(t)^{-1}\left(\pa_{\tau_2} x+b(t)y \pa_{\tau_2} z\right)\right) -yk'(t)\left(-k(t)^{-1}\pa_{\tau_2} z\right)\\
\notag &=-\pa_{\tau_2} x-b(t)y \pa_{\tau_2} z+ yk'(t)k(t)^{-1} \pa_{\tau_2} z,\\
\label{library2}\pa_{\tau_2} (k(t)y) &= \pa_{\tau_1} x +b(t)y \pa_{\tau_1} z-yk'(t)k(t)^{-1} \pa_{\tau_1} z.
\end{align}

Next we claim that
\begin{equation}\label{subharmonic 1}
\pa^{2}_{\tau_1}(k(t)y)+\pa^{2}_{\tau_2}(k(t)y)= Q_1(t,y)\left(\pa_{\tau_1} y,\pa_{\tau_2} y,\pa_{\tau_1} z,\pa_{\tau_2} z\right),
\end{equation}
where $Q_{1}(t,y)$ is a quadratic form whose coefficients satisfy an $\Ordo(K^{-1}\sigma^{-1/2})$-bound, where $K\gg1$ is the constant in the definition of $J_{\delta;\sigma}$.  By differentiating Equations~\eqref{library} and \eqref{library2} and using the Cauchy--Riemann equations to express $\pa_{\tau_1} t$ and $\pa_{\tau_2} t$ in terms of $\pa_{\tau_1} z$ and $\pa_{\tau_2} z$, we obtain
\begin{align*}
\pa^{2}_{\tau_1}(k(t)y)+\pa^{2}_{\tau_2}(k(t)y) =& y(-b'(t) +(\log k(t))'') k(t)^{-1}((\bdry_{\tau_1}z)^2+(\bdry_{\tau_2}z)^2)\\
& + (-b(t)+(\log k(t))')(\bdry_{\tau_1}y\bdry_{\tau_2} z -\bdry_{\tau_1}z \bdry_{\tau_2}y).
\end{align*}
To see the bound on coefficients, note that we can choose $b(t)$ and $k(t)$ so that $b'(t)=\Ordo(K^{-1}\sigma^{-1})$, $c_0<k(t)<c_1$ where $c_0,c_1>0$ are independent of $\sigma$, $(\log k(t))'=1$ near $t=0$ and $(\log k(t))'=0$ near $t={1\over 2} K\sigma$, and $(\log k(t))''=\Ordo(K^{-1}\sigma^{-1})$. Since $|y|=\Ordo(\sigma^{1/2})$ by Lemma~\ref{lemma: bounds}, the claim follows.

A similar but easier calculation gives
\begin{equation} \label{subharmonic 2}
\pa^{2}_{\tau_1} z+\pa^{2}_{\tau_2} z= 0.
\end{equation}

Using Equations~\eqref{subharmonic 1} and \eqref{subharmonic 2} we compute:
\begin{align*}
\left(\bdry_{\tau_1}^2+\bdry_{\tau_2}^2\right)\left((k(t)y)^{2}+z^{2}\right)=& 2\left(|\nabla(k(t)y)|^{2}+|\nabla z|^{2}\right.\\
&\left.+k(t)y\;Q_1(t,y)\left(\pa_{\tau_1} y,\pa_{\tau_2} y,\pa_{\tau_1} z,\pa_{\tau_2} z\right) \right),
\end{align*}
where $\nabla=(\bdry_{\tau_1},\bdry_{\tau_2})$.
The bounds on the quadratic forms imply that there is a constant $C>0$ such that
\[
Q_j(t,y)\left(\pa_{\tau_1} y,\pa_{\tau_2} y,\pa_{\tau_1} z,\pa_{\tau_2} z\right)\le CK^{-1}\sigma^{-1/2}(|\nabla y|^{2}+|\nabla z|^{2}),
\]
for $j=1,2$. By Lemma~\ref{lemma: bounds}, $|y|=\Ordo(\sigma^{1/2})$ and $|z|=\Ordo(\sigma^{1/2})$. Finally, since
$$\nabla (k(t)y)=k(t)\nabla y + (\log k(t))'(\bdry_\sigma z, -\bdry_\tau z) y,$$
the dominant term is $k(t)\nabla y$ and we conclude that
\[
(\bdry_{\tau_1}^2+\bdry_{\tau_2}^2)\left((k(t)y)^{2}+z^{2}\right)\ge 0,
\]
provided $K\gg1$ is sufficiently large.

Similarly, $p^2_{\delta;\sigma}\circ u$ is subharmonic on the region $\tfrac12K\sigma\le t\le K\sigma$. The lemma follows.
\end{proof}

\begin{lemma} \label{lemma: bound}
If $u\in \mathcal{M}^{(X_\delta,L_{\delta;\sigma});J_{\delta;\sigma}}(a;\mathbf{b})$, then
$$|p_{\delta;\sigma}\circ u|=\sqrt{p^2_{\delta;\sigma}\circ u}=\Ordo(\sigma),$$
for all sufficiently small $\sigma>0$.
\end{lemma}

\begin{proof}
This is immediate from the second assertion of Lemma \ref{lem:subharm} by observing the following: when restricted to $L_{\delta;\sigma}\cap (\{0\leq t<\infty\}\cup \{-\infty< t\leq 0\})$, the function $p^2_{\delta;\sigma}=(k_\sigma(t)y)^2+z^2$ is bounded above by $C y^2+z^2$ for a positive constant $C$ which is independent of $\sigma$ and $|y|,|z|=\Ordo(\sigma)$.
\end{proof}

\subsection{From rigid disks to rigid trees} \label{subsection: disks to trees}


Consider a sequence
$$u_{\delta;\sigma}\colon (D_{m+1},\pa D_{m+1})\to (X_{\delta},L_{\delta;\sigma}),\quad \delta,\sigma\to 0,$$
of rigid disks in $\mathcal{M}^{(X_\delta,L_{\delta;\sigma});J_{\delta;\sigma}}(a;\mathbf{b})$.

\begin{lemma}\label{lem:disktotree}
After passing to a subsequence, there exists $\delta_0>0$ such that for any $\delta\in (0,\delta_0)$ there exists $\sigma_0=\sigma_0(\delta)>0$ such that for any $\sigma\in(0,\sigma_0)$, there exists a rigid LC flow tree $\hat\gamma_\delta$ of $L_{\delta;\sigma}$ such that the following hold:
\begin{enumerate}
\item $u_{\delta;\sigma}(D_{m+1})\cap T^*F$ lies in an $\Ordo(\sigma\log(\sigma^{-1}))$-neighborhood of the cotangent lift of $\hat \gamma_\delta$;
\item $\hat\gamma_\delta$ lies in a $\phi(\delta)$-neighborhood of the Reeb chord flow lines over $a$ and ${\bf b}$ and $u_{\delta;\sigma}(D_{m+1})\cap (X_{\delta}-T^{\ast}F)$ lies in a $\phi(\delta)$-neighborhood of the strips over $a$ and ${\bf b}$, where $\displaystyle\lim_{\delta\to 0}\phi(\delta)=0$.
\end{enumerate}
\end{lemma}

\begin{proof}
The lemma is almost a consequence of \cite[Theorem 1.2]{E1}.

\s\n
(1) First observe that $|p_{\delta;\sigma}\circ u_{\delta;\sigma}|=\Ordo(\sigma)$ by Lemma~\ref{lemma: bound}. Then we obtain local (= away from points where the derivative blows up) flow tree convergence at a rate of $\Ordo(\sigma\log(\sigma^{-1}))$ as in \cite[Sections 5.2--5.3]{E1}. As in the usual proof of Gromov compactness, we add punctures at points where the derivative blows up on the $\sigma$-scale, while noting that the number of such points is controlled by the average linking number (cf.\ \cite[Section 5.4]{E1}). This implies the flow tree convergence when restricted to $T^{\ast}F$.

\s\n
(2) By Corollary~\ref{cor: canonicalpunctures}, the restriction of $\hat\gamma_\delta$ to $X_\delta- T^*F$ is contained in $N_{\phi(\delta)}$ for some $\phi$. The argument in \cite[Lemma 5.7]{E1} gives $\Ordo(\sigma)$ bounds on the derivative of $u_{\delta;\sigma}$ on $\{0\leq t\leq K\sigma\}$. Hence, for $\sigma>0$ sufficiently small, the restriction of $u_{\delta;\sigma}$ to $\{0\leq t\leq K\sigma\}$ is close to cylinders over Reeb chords.

Finally we claim that the restriction of $u_{\delta;\sigma}$ to $\{t\geq K\sigma\}$ is close to cylinders over Reeb chords for $\delta,\sigma>0$ sufficiently small. Arguing by contradiction, there is a sequence $u_{\delta_i;\sigma_i}$ such that some point of $u_{\delta_i;\sigma_i}|_{\{t\geq K\sigma_i\}}$ maps outside a fixed $\varepsilon$-neighborhood of cylinders over Reeb chords. Recalling that $\{t\geq K\sigma\}$ coincides with the symplectization of $\R^3$, the flow tree convergence applied the $1$-dimensional Legendrian submanifold $\Lambda_\pm$ yields a flow tree of $\Lambda_\pm$ of dimension $\geq 0$. This implies that $\ind(u_{\delta_i;\sigma_i})\geq 1$, which is a contradiction.
\end{proof}

\subsection{From rigid trees to rigid disks}
\label{subsection: trees to disks}

In this subsection we construct holomorphic disks near rigid LC flow trees as in \cite[Section 6]{E1}.

The construction can be summarized as follows: As we take the limit $\sigma\to 0$, the Lagrangian boundary condition $L_{\delta;\sigma}$ degenerates to the $0$-section and the domains $S_\sigma$ of rigid $u_\sigma\in \mathcal{M}^{(X_\delta,L_{\delta;\sigma});J_{\delta;\sigma}}(a;\mathbf{b})$ converge to a ``broken domain'' at the boundary of the space of conformal structures; see \cite[Remark 5.35]{E1}. After subdividing $S_\sigma$ for $\sigma>0$ small, we find explicit local solutions in a neighborhood of the rigid LC flow tree $\hat\gamma$, except in small shrinking regions where the neighboring local solutions are patched together. This yields a disk in $(X_\delta,L_{\delta;\sigma})$ which is close to being $J_{\delta;\sigma}$-holomorphic as in \cite[Section 6.2]{E1}. Finally we use Newton iteration in order to produce actual solutions as in \cite[Section 6.4]{E1}.

In order to extend the argument to the case of a Lagrangian with cylindrical ends, we first construct explicit local solutions in the cylindrical ends near Reeb chords.

\subsubsection{Construction of a truncated local solution} \label{subsubsection: truncated}

Let $J_0$ be the almost complex structure on $\R\times \R^3$ given by Equation~\eqref{e:translinvJ}.  In this subsection we will write down an explicit local $J_0$-holomorphic map, i.e., a ``truncated local solution'',
$$u=(t,x,y,z)\colon [0,\infty)\times[0,1]\to\R\times\R^{3},$$
where $[0,\infty)\times[0,1]$ is a half-strip with coordinates $\zeta=\tau_1+i\tau_2$ and $u$ is asymptotic to the Reeb chord $c\in\mathcal{C}(\Lambda_+)$ at the positive end. The situation at the negative end is analogous.

We may assume that there exist $-{\pi\over 2}<\alpha_+<\alpha_-<{\pi\over 2}$, $\alpha_\pm=\alpha_\pm(\sigma)$, such that the Lagrangian projections $\Pi_{\C}$ of the two branches of $\Lambda=\Lambda_+$ near $c$ are contained in $e^{i\alpha_+}\R$ and $e^{i\alpha_-}\R$ and that $\Pi_\C$ of the positive end of $u$ sweeps out the (small) sector from $\alpha_+$ to $\alpha_-$. As $\sigma\to 0$, we have $|\alpha_+|,|\alpha_-|=\Ordo(\sigma)$. Moreover, by the deformation of $\Lambda$ from Section~\ref{subsubsection: deforming legendrian links}, the $z$-coordinates of the two branches are constant, i.e., $z=h_+$ and $z=h_-$, where $h_\pm=h_\pm(\sigma)$ and $|h_\pm|=\Ordo(\sigma)$ as $\sigma\to 0$. Here the subscript $+$ (resp.\ $-$) refers to the upper (resp.\ lower) strand.

The map $u$ satisfies the Cauchy--Riemann equations if and only if $v=x+iy$ and $w=t+iz$ satisfy
\begin{align}\label{eqn: holomorphic}
\overline\bdry v&=\tfrac{1}{2} (\bdry_{\tau_1}+ i \bdry_{\tau_2}) v =0,\\
\label{eq:holinsympl}
\overline\bdry w&=\tfrac{1}{2} (\bdry_{\tau_1}+ i \bdry_{\tau_2}) w  = -\tfrac{1}{2} y \bdry_{\tau_2}\overline{v}.
\end{align}

If we choose the solution
\begin{equation} \label{eqn: v}
v(\zeta)=e^{i\alpha_-}e^{(\alpha_+-\alpha_-)\zeta},
\end{equation}
for Equation~\eqref{eqn: holomorphic} and abbreviate $\theta=\alpha_+-\alpha_-$, then Equation~\eqref{eq:holinsympl} becomes:
\begin{equation}\label{eq:explholsympl}
\bar\pa w = \tfrac{\theta}{4}e^{2\theta {\tau_1}}\left(1-e^{-2i(\alpha_-+\theta\tau_2)}\right)=\tfrac{\theta}{4} e^{\theta(\zeta+\overline\zeta)} -\tfrac{\theta}{4} e^{-2i\alpha_+}e^{2\theta\overline\zeta}.
\end{equation}
In order to solve for $w$, first let
\begin{equation}
w_0(\zeta)=(h_+-h_-)\zeta+ih_-+c_0
\end{equation}
be a holomorphic map from $[0,\infty)\times[0,1]$ to the strip over the Reeb chord $c$, postcomposed with the projection to the $(t,z)$-coordinates. Here $c_0$ is a constant.  Next observe that
\begin{equation}
w_1(\zeta)=\tfrac{1}{4} e^{\theta(\zeta+\overline\zeta)}-\tfrac{1}{8}e^{-2i\alpha_+}e^{2\theta\overline\zeta}+ \tfrac{1}{8}e^{2i\alpha_+} e^{2\theta\zeta}
\end{equation}
solves Equation~\eqref{eq:explholsympl} and satisfies real boundary conditions, and $\displaystyle\lim_{\tau_1\to\infty}w_1(\zeta)=0$. Then we take $w=w_0+w_1$.

The truncated local solution is then given by:
\begin{equation} \label{local solution}
u(\zeta)=(v(\zeta), w(\zeta)).
\end{equation}

\subsubsection{Gluing}

\begin{lemma}\label{lem:treetodisk}
If $\delta>0$ is sufficiently small, then there exists $\sigma_0>0$ such that for any $0<\sigma<\sigma_0$ and any rigid LC flow tree $\hat\gamma$ of $L_{\delta;\sigma}$, there is a unique rigid $J_{\delta;\sigma}$-holomorphic disk with boundary on $L_{\delta;\sigma}$ in a neighborhood of $\hat\gamma$.
\end{lemma}

\begin{proof}
This follows from the proof of \cite[Theorem 1.3]{E1}, where the only new ingredient is the presence of the cylindrical end. The local solutions are constructed as in \cite[Section 6.1]{E1}, where the solutions near the $1$-valent punctures are replaced by the truncated local solutions $u$ from Section~\ref{subsubsection: truncated}. We then construct the approximate holomorphic disks by gluing as in \cite[Section 6.2]{E1}. The weight functions on the domains are still as in \cite[Section 6.3.1]{E1}. In particular, we impose small positive exponential weights near the  punctures and the weight function is equal to $1$ on the strip regions of length $\Ordo(K\sigma)$ and width $\Ordo(\sigma)$ that map into the interpolation region $\{0\leq t\leq K\sigma\}$.\footnote{The interpolation region interpolates between the local solution on the symplectization part and a local solution on $T^*F$ which is ``approximated'' by a flow line.}



The necessary conditions for applying Newton iteration are still satisfied after these modifications are made. The uniform invertibility of the differential is argued as in \cite[Proposition 6.21]{E1}, the quadratic estimate is argued as in \cite[Proof of Theorem 1.3, p.\ 1216]{E1}. The surjectivity of the construction, i.e., the fact that Newton iteration captures all rigid holomorphic disks, follows from an analog of the calculation in \cite[p.\ 1218]{E1}, as follows. Let $u$ be the truncated local solution from Section~\ref{subsubsection: truncated}.  It is not hard to check that the $C^0$-norm near the ends control the Fourier coefficients of $u$ and the Fourier coefficients control the weighted Sobolev norm in the interpolation region $\{0\leq t\leq K\sigma\}$ via the $C^{0}$-norm. The $C^0$-norm is controlled by Lemma \ref{lem:disktotree}.
\end{proof}

\begin{proof}[Proof of Theorem \ref{thm:flowtreecomp}]
Let $L^{\Mo}$ be a Morse cobordism in $T^{\ast}F$ and let $(X_{\delta},L_{\delta;\sigma})$ be the associated conical exact Lagrangian cobordisms parametrized by $\delta,\sigma>0$. For $\delta,\sigma>0$ sufficiently small, any $u\in\mathcal{M}^{(X_\delta,L_{\delta;\sigma});J_{\delta;\sigma}}(a;\mathbf{b})$ is ``approximated by'' a rigid LC flow tree $\hat \gamma_\delta$ in the sense of Lemma~\ref{lem:disktotree}.

On the other hand, by Lemma~\ref{lem:treetodisk}, for $\sigma,\delta>0$ sufficiently small, there is a unique rigid holomorphic disk in $\mathcal{M}^{(X_\delta,L_{\delta;\sigma});J_{\delta;\sigma}}(a;\mathbf{b})$ which ``approximates'' the rigid LC flow tree $\hat\gamma_\delta$. The theorem now follows from Lemma~\ref{lem:long_Morse}, which gives a bijection between rigid flow trees of $L^{\Mo}$ and rigid LC flow trees of $L_{\delta;\sigma}$.
\end{proof}

\section{Elementary exact Lagrangian cobordisms and their DGA maps}
\label{Sec:DGAmaps}

In this section we introduce the elementary exact Lagrangian cobordisms and compute the induced DGA maps. An {\em elementary exact Lagrangian cobordism $L$} is one of the following:
\begin{enumerate}
\item a cobordism induced by a $\Pi_\C$-simple Legendrian isotopy;
\item a cobordism induced by a Legendrian Reidemeister move;
\item a minimum cobordism; or
\item a saddle cobordism.
\end{enumerate}
The cobordisms will be discussed in Sections~\ref{sec:nomovescob}, \ref{sec:movescob}, \ref{subsub:minimum}, and \ref{subsub:saddle}, respectively. An exact Lagrangian cobordism $L$ is {\em decomposable} if it is exact Lagrangian isotopic to a concatenation of elementary exact Lagrangian cobordisms.

In this section the coefficient ring of the DGAs is $\F$, unless stated otherwise.

\subsection{Lagrangian cobordisms from Legendrian isotopies}

Let $\Lambda_\tau\subset \R^3$, $\tau\in[0,1]$, be a $1$-parameter family of Legendrian links from $\Lambda_0$ to $\Lambda_1$. Choose a parametrization $\gamma_{\tau}\colon S\to\R^{3}$ of $\Lambda_\tau$, where $S$ is a (not necessarily connected) closed $1$-manifold. We reparametrize the $\tau$-parameter via a map $f\colon\R\to [0,1]$ with small derivative so that $f(t)=1$ for $t\ge t_0\gg 0$ and $f(t)=0$ for $t\le -t_0$. Consider the trace
$$\Gamma\colon\R\times S\to \R\times\R^{3},$$
$$(t,s)\mapsto (t, \gamma_{f(t)}(s)),$$
of the isotopy $\gamma_{f(t)}$. Then
\begin{align*}
\Gamma([t_0,\infty)\times S)& =[t_0,\infty)\times \Lambda_1,\\
\Gamma((-\infty,-t_{0}]\times S)&=(-\infty,-t_{0}]\times\Lambda_0.
\end{align*}

The following lemma is a version of a standard result; compare e.g.~\cite[Lemma A.1]{E2}.

\begin{lemma}\label{lemma:isotop-cobord}
For any $\epsilon>0$, there exists $\delta>0$ such that if $\Gamma$ satisfies $\left|{\bdry \Gamma\over \bdry t}\right|_{C^0}<\delta$, then there is a cylindrical exact Lagrangian cobordism $L$ from $\Lambda_1$ to $\Lambda_0$ which is $\epsilon$-close (in the $C^0$-metric) to the image of $\Gamma$.
\end{lemma}

\begin{proof}
Let $H(t,s)= \alpha_0\left({\bdry\Gamma\over \bdry t}(t,s)\right)$, where $\alpha_0=dz-ydx$. We then write
\[
\Gamma(t,s)=\left(t,x(t,s),y(t,s),z(t,s)\right),
\]
and consider the deformed map
\begin{equation}\label{Eq:Gammamodif}
\Gamma'(t,s)=\left(t,x(t,s),y(t,s),z(t,s)+H(t,s)\right).
\end{equation}
We calculate
\begin{eqnarray*}
(\Gamma')^*(e^t\alpha_0)&=& (\Gamma')^*(e^t (dz-ydx))\\
&=& e^t\left( {\bdry z\over \bdry t}dt+{\bdry z\over \bdry s}ds -y\left({\bdry x\over \bdry t}dt +{\bdry x\over \bdry s}d s\right)+{\bdry H\over \bdry t}dt +{\bdry H\over \bdry s}ds\right)\\
&=& e^t \left( \left( {\bdry z\over \bdry t}-y{\bdry x\over \bdry t}\right)d t + {\bdry H\over \bdry t}d t +{\bdry H\over \bdry
s}ds\right)\\
&=& e^t \left(H(t,s)dt+ {\bdry H\over \bdry t}dt +{\bdry H\over\bdry s}ds\right)\\
&=& d(e^t H(t,s)).
\end{eqnarray*}
Hence $\Gamma'$ is exact Lagrangian. Since $\Gamma\colon\R\times S\rightarrow \R\times M$ is an embedding, $\Gamma'$ is also an embedding, provided the modification $H(t,s)$ is sufficiently small. Finally, Condition (i) in Definition~\ref{defn: exact Lagrangian with cylindrical ends} is satisfied since $H(t,s)=0$ at the ends of $\Gamma'$. The lemma follows.
\end{proof}

\subsection{Simple Legendrian isotopies}
\label{sec:nomovescob}

Let $\Lambda_\tau\subset\R^{3}$, $\tau\in[0,1]$, be a $1$-parameter family of Legendrian links from $\Lambda_0$ to $\Lambda_1$. After a small perturbation, $\Lambda_\tau$ is chord generic except at isolated instances $0<\tau_1<\dots<\tau_k<1$, where the Lagrangian projection $\Pi_\C(\Lambda_{\tau_i})$ has a self-tangency or a triple point and passing from $\Pi_\C(\Lambda_{\tau_i-\epsilon})$ to $\Pi_\C(\Lambda_{\tau_i+\epsilon})$, $\epsilon>0$ small, corresponds to a Legendrian Reidemeister move; see \cite[Figure 6]{Ka1}.

\begin{defn}
A Legendrian isotopy $\Lambda_{\tau}$, $\tau\in[0,1]$, is \emph{$\Pi_\C$-simple} if the Lagrangian projection of $\Lambda_{\tau}$ has only transverse double points for all $\tau\in[0,1]$, i.e., there are no Reidemeister moves during the isotopy.
\end{defn}

In particular, if $\Lambda_\tau$, $\tau\in[0,1]$, is $\Pi_\C$-simple, then there is a natural identification $\phi_{\tau,\tau^*}:\mathcal{C}(\Lambda_\tau)\stackrel\sim\to \mathcal{C}(\Lambda_{\tau^*})$ for all $\tau,\tau^*\in[0,1]$.

We now use Lemma \ref{lemma:isotop-cobord} to construct cobordisms of $\Pi_\C$-simple isotopies and compute the corresponding cobordism maps.

\begin{lemma}\label{l:noRM->id}
Let $\Lambda_\tau$, $\tau\in[0,1]$, be a $\Pi_\C$-simple Legendrian isotopy. Then there exist $\epsilon>0$ and a subdivision of $[0,1]$ into intervals $[a,a+\epsilon]$ of length $\epsilon$ such that the Lagrangian cobordisms that correspond to $\Lambda_\tau$, $\tau\in[a,a+\epsilon]$, all induce the DGA isomorphisms
$$(\A(\Lambda_a),\bdry_a)\stackrel\sim\to (\A(\Lambda_{a+\epsilon}),\bdry_{a+\epsilon})$$
which map $c\in \mathcal{C}(\Lambda_a)$ to the corresponding chord $\phi_{a,a+\epsilon}(c)\in \mathcal{C}(\Lambda_{a+\epsilon})$.
\end{lemma}

\begin{proof}
Arguing by contradiction, suppose there exist $\epsilon_i\to 0$ and intervals $[a_i,a_i+\epsilon_i]$ such that the concordance corresponding to $\Lambda_\tau$, $\tau\in[a_i,a_i+\epsilon_i]$, contains an $\op{ind}=0$ holomorphic disk which is not close to a trivial strip. By passing to a subsequence we may assume that $a_i\to a$. By Gromov compactness, the trivial cylinder over $\Lambda_a$ has an $\op{ind}=0$ disk which is not a trivial strip. Such a disk projects to a nontrivial disk of Fredholm index $-1$ in $\C$ with boundary on $\Pi_\C(\Lambda_a)$.

On the other hand, no such disk exists by the argument principle: If $u: D_m \to \C$ is a holomorphic disk whose boundary maps to $\Pi_\C(\Lambda_a)$, then an easy calculation shows that the Fredholm index $\op{ind}(u)$ of $u$ is given by $\mu(\bdry u) -2$, where $\mu(\bdry u)$ is the Maslov index along $\bdry u$ with positive ${\pi\over 2}$-rotations at the corners.  By the argument principle $\mu(\bdry u)\geq 2$.  Hence $\op{ind}(u)\geq 0$ and the lemma follows.
\end{proof}

\subsection{Cobordisms corresponding to Reidemeister moves}
\label{sec:movescob}

Following \cite{EK}, we consider three Reidemeister moves:
\begin{itemize}
\item[(L1)] a triple point move;
\item[(L2)] pair cancellation of Reeb chords;
\item[(L3)] pair creation of Reeb chords.
\end{itemize}
The (L1)-, (L2)-, and (L3)-moves are depicted in Figures~\ref{fig:L1a}, ~\ref{fig:L1b}, and ~\ref{fig:L2}. The Morse cobordisms $L^{\Mo}\subset T^*F$, $F=\R\times[0,1]$, corresponding to the (L1)-, (L2)-, and (L3)-moves are referred to as the {\em triple point}, {\em death}, and {\em birth} cobordisms. Note that Morse cobordisms are called ``Legendrian submanifolds with standard ends'' in \cite[Section 3.2]{EK}.

We remark that the exact Lagrangian cobordisms corresponding to the triple point, death, and birth Morse cobordisms, as well as the saddle and minimum cobordisms from Sections~\ref{subsub:minimum} and \ref{subsub:saddle}, satisfy Definition~\ref{defn: exact Lagrangian with cylindrical ends}(i) since we may assume that the modifications occur inside a small region and outside this region the isotopy is trivial.

\subsubsection{Abstract and geometric perturbations}

We would like to apply Theorem~\ref{thm:flowtreecomp} to compute the DGA morphisms corresponding the above Morse cobordisms $L^{\Mo}$ from $\Lambda_+$ to $\Lambda_-$. For simplicity we will be using $\F$-coefficients.

Let $L\subset T^*F$ be an exact Lagrangian cobordism which satisfies the following:
\begin{itemize}
\item $L$ is exact Lagrangian isotopic to $L^{\Mo}$ relative to $\Lambda_-\times\{0\}$ and $\Lambda_+\times\{1\}$; and
\item $L$ restricts to $\Lambda_-\times[0,\epsilon]$ and $\Lambda_+\times [1-\epsilon,1]$ over $\R\times[0,\epsilon]$ and $\R\times[1-\epsilon,1] \subset F$.
\end{itemize}

A {\em geometric perturbation} of $L$ is a perturbation of the exact Lagrangian $L$ relative to $\bdry L$ together with a perturbation of the Riemannian metric on $F$ which is used to define gradients. We assume additionally that the perturbation of $L$ is a Morse cobordism.  For a generic geometric perturbation, the formal dimension $\leq 1$ moduli spaces of flow trees with one positive puncture on $\Lambda_+$ and arbitrarily many negative punctures on $\Lambda_-$ are transversely cut out by \cite[Theorem 1.1]{E1}.   A count of rigid flow trees from $a$ to $\mathbf{b}$ with respect to a generic geometric perturbation immediately gives $|\mathcal{T}(a;\mathbf{b})|$ in Theorem~\ref{thm:flowtreecomp}.

However, it is easier to compute the DGA morphisms using Morse-Bott type considerations since $L$ is close to being Morse-Bott degenerate. In order to rigorously treat Morse-Bott theory for flow trees, we use an {\em abstract perturbation} scheme which counts {\em perturbed Morse-Bott flow tree cascades}.  They are defined in \cite[Section 3.4]{EK} under the name ``perturbed generalized flow trees''. At this point the reader is encouraged to review Sections 3.3 and 3.4 of \cite{EK}, including the notions of a {\em Morse-Bott flow tree cascade} (=``generalized flow tree''), a {\em slice tree}, a {\em connector}, and the {\em level} of a cascade. In this paper we will simply refer to a ``Morse-Bott flow tree cascade'' as a ``cascade''. For a generic abstract perturbation, the formal dimension $\leq 1$ moduli spaces of perturbed cascades with one positive puncture and arbitrarily many negative punctures are transversely cut out by \cite[Lemma~3.9]{EK}. In \cite[Lemmas~6.6, 6.7, and 6.8]{EK}, for each of the moves, all the rigid perturbed cascades were determined for a certain generic abstract perturbation of $L$.

Let
$$\Phi_{\rm a},\Phi_{\rm g}:\A(\Lambda_{+})\to\A(\Lambda_{-})$$
be the maps defined by counting rigid perturbed cascades or rigid flow trees of $\tilde L$ using an abstract perturbation and a geometric perturbation, respectively.

\begin{lemma}\label{lem:geomvsabstr}
The maps $\Phi_{\rm a}$ and $\Phi_{\rm g}$ are DGA morphisms and are chain homotopic, i.e., there is a degree $+1$ map $K$ which takes generators of $\A(\Lambda_{+})$ to $\A(\Lambda_{-})$ so that
\[
\Phi_{\rm a}-\Phi_{\rm g}=\Omega_{K}\circ\pa_{+}+\pa_{-}\circ\Omega_{K}.
\]
Here $\Omega_K$ is as in Lemma \ref{lem:chainhomotopy}.
\end{lemma}

\begin{proof}
The fact that $\Phi_{\rm a}$ and $\Phi_{\rm g}$ are DGA morphisms follows from the compactness and transversality properties of perturbed cascades in $J^{1}F$. The proofs require only standard finite-dimensional arguments given in \cite[Lemma 3.9]{EK}.

The proof of the chain homotopy is similar to that of \cite[Lemma 3.13]{EK} and is the usual chain homotopy argument in disguise.  We make the following simplifying assumptions:
\begin{enumerate}
\item[(i)] the abstract and geometric perturbations are close and are connected by a $1$-parameter family of abstract perturbations $\mathfrak{P}_t$, $t\in[0,1]$;
\item[(ii)] there is a single $\op{ind}=0$ disk in the $1$-parameter family of moduli spaces of rigid disks $\bigsqcup_{t\in[0,1]} \mathcal{M}_t$, where $\mathcal{M}_t$ is with respect to $\mathfrak{P}_t$.\footnote{The reason we consider $\op{ind}=0$ disks is that $\Phi_{\rm a}$ and $\Phi_{\rm g}$ count $\op{ind}=1$ disks in $\R\times J^1 F$.}
\end{enumerate}
Note that we can view a geometric perturbation as an instance of an abstract perturbation. 

Let $F_I= F\times[0,1]/\sim$, where $(x,t)\sim(x,t')$ for all $x\in \bdry F$ and $t,t'\in[0,1]$, and let $L_I=L\times [0,1]/\sim$ be a Legendrian submanifold of $J^1(F_I)$. Also let $D=[0,1]\times[0,1]/\sim$, where $(s,t)\sim (s,t')$ for $s\in\{0,1\}$ and $t,t'\in[0,1]$, and let $\pi: T^*F_I\to D$ be the corresponding projection.  We assume that $L_I\cap \pi^{-1}(D_\pm)$ is a trivial cobordism over $\Lambda_\pm$, where $$D_-=\{0\leq s\leq \tfrac{1}{3}\},\quad D_+= \{\tfrac{2}{3}\leq s\leq 1\}$$
are subsets of $D$.

We think of $D$ and $F_I$ as smooth surfaces in the obvious way and choose a Morse function $f$ on $D$ such that:
\begin{itemize}
\item $f$ has one saddle point $h=\{s=1\}$, one maximum $e_+$ in the interior of $D_+$, one minimum $e_-=\{s=0\}$, and no other critical points;
\item $-\nabla f=-\bdry_s$ on $D-D_+-D_-$; and
\item $t=0$ and $t=1$ are gradient trajectories from $h$ to $e_-$.
\end{itemize}
The set $\mathcal{C}(L_I)$ of Reeb chords of $L_I$ is given by
\[
\CC(L_I)=\CC(\Lambda_-)\cup \hat\CC(\Lambda_+)\cup\tilde \CC(\Lambda_+),
\]
where $\CC(\Lambda_-)$, $\hat\CC(\Lambda_+)$, and $\tilde\CC(\Lambda_+)$ are the sets of Reeb chords that lie above $e_-$, $h$, and $e_+$, respectively.  Each of $\hat\CC(\Lambda_+)$ and $\tilde\CC(\Lambda_+)$ is in one-to-one correspondence with $\CC(\Lambda_+)$, except that $|\hat c|=|c|+1$ and $|\tilde c|=|c|+2$. Here $c\in \CC(\Lambda_+)$ and $\hat c$ and $\tilde c$ are the corresponding elements in $\hat\CC(\Lambda_+)$ and $\tilde \CC(\Lambda_+)$; also let $d\in \CC(\Lambda_-)$. Let $\hat{\mathcal{A}}(\Lambda_+)$ and $\tilde {\mathcal{A}}(\Lambda_+)$ be the algebras generated by $\hat\CC(\Lambda_+)$ and $\tilde\CC(\Lambda_+)$.

Assume that the cobordism $L$ with the perturbations $\frak{P}_0$ and $\frak{P}_1$ correspond to the two flow lines connecting $h$ to $e_-$. As in the proof of \cite[Lemma 3.13]{EK}, we claim that the differential $\Delta$ of $\mathcal{A}(L_I)$ is given as follows:
\begin{align}
\label{eq1} \Delta d &= \pa_- d,\\
\label{eq2} \Delta \hat c &= \Phi_{\rm a}(c)+\Phi_{\rm g}(c)+\hat{\mathbf{O}}(1),\\
\label{eq3} \Delta \tilde c &= \hat c + K(\hat{c}) + Q (\pa_+ c)+\mathbf{O}(2).
\end{align}
Here:
\begin{itemize}
\item $\hat{\mathbf{O}}(n)$ (resp.\ $\mathbf{O}(n)$) is a sum of words, each of which has at least $n$ letters in $\hat\CC(\Lambda_+)$ (resp.\ $\hat\CC(\Lambda_+)\cup\tilde\CC(\Lambda_+)$);
\item $K:\hat{\CC}(\Lambda_+)\to \mathcal{A}(\Lambda_-)$ is the chain homotopy corresponding to the $\op{ind}=0$ disk; and
\item  $Q (b_1\dots b_m) = \tilde b_1 \Phi_{\rm g}(b_2\dots b_m) + \Phi_{\rm a}(b_1)\tilde b_2 \Phi_{\rm g}(b_3\dots b_m)$
$$\qquad \quad + \dots +\Phi_{\rm a}(b_1\dots b_{m-1})\tilde b_m.$$
\end{itemize}

We briefly indicate how the terms of Equation~\eqref{eq3} are obtained.  We first enumerate the rigid unperturbed cascades which contribute to $\mathbf{O}(0)$ and $\mathbf{O}(1)$. In the case of $\mathbf{O}(0)$, the chain homotopy term $K$ is obtained by viewing $f$ as a perturbation of a Morse function $\tilde f$ such that:
\begin{itemize}
\item $\tilde f$ has one maximum $e_+=\{s=1\}$ and one minimum $e_-=\{s=0\}$ and no critical points in $int(D)$; and
\item $t=0$ and $t=1$ are gradient trajectories from $e_+$ to $e_-$.
\end{itemize}
In the case of $\mathbf{O}(1)$, a rigid cascade $\Gamma$ can have at most one level by an index computation.  If the level of $\Gamma$ is zero, then it is a connector from $\tilde c$ to $\hat c$.  If the level of $\Gamma$ is one, then it consists of a slice tree $\Gamma_{e_+}$ from $\tilde c$ to $\tilde b_1,\dots,\tilde b_m$, together with $m-1$ rigid trees $\gamma_i$, $i\not=i_0$, from $\tilde b_i$ to some component $\mathbf{a}_i=a_{i1}\dots a_{ij_i}$ of $\Phi_{\rm a}(b_i)$ or $\Phi_{\rm g}(b_i)$. (In the case where $L_I$ is not a product $\Lambda\times D$, we substitute rigid trees for connectors.)

An abstract perturbation analogous to the time-ordered, domain-dependent abstract perturbation from the proof of Lemma~\ref{lem:chainhomotopy} then gives the term $Q(b_1\dots b_m)$, as follows: Let $N(e_+)\subset D$ be a small neighborhood of $e_+$, $q\in N(e_+)$, $\Gamma_q$ a parallel copy of the slice tree $\Gamma_{e_+}$ over $q$, and $c(q)$ the Reeb chord over $q$ corresponding to $c\in\mathcal{C}(\Lambda_+)$. Suppose that
\begin{enumerate}
\item[(*)] the $\op{ind}=0$ disk that contributes to $K$ lies over the gradient trajectory $t={1\over 2}$ of $f$.
\end{enumerate}
For $q\in N(e_+)$, we define the {\em perturbation function} (cf.\ \cite[Section 3.4]{EK})
$$v(\Gamma_q): C(\Gamma_q)\to T_q D\simeq \R^2,$$ where $C(\Gamma_q)$ is the cotangent lift of $\Gamma_q$, so that the following holds for all rigid slice trees $\Gamma_q$:
\begin{enumerate}
\item $v(\Gamma_q)$ is independent of $q$;
\item $v(\Gamma_q)$ is zero near $c(q)$ and takes distinct constant values $v_i\in \R^2$ near $b_i(q)$;
\end{enumerate}
We specify $v_i$ further:
\begin{enumerate}
\item[(3)] Let $\overline{v}_i=(s_i,-\varepsilon)$, where $\varepsilon>0$ is small and $s_1<\dots < s_m$.  Then $v_i$ is a generic point which is $\delta$-close to $\overline{v}_i$, where $0<\delta\ll \varepsilon$.
\end{enumerate}
We will write $\Gamma_q+ v(\Gamma_q)$ for $\Gamma_q$ shifted in the $D$-direction via $v(\Gamma_q)$.
A rigid perturbed cascade $\tilde \Gamma$ that corresponds to $\Gamma$ consists of the following:
\begin{itemize}
\item[(a)] a connector from $\tilde c$ to $c(q)$ for some point $q\in N(e_+)$;
\item[(b)] a perturbed flow tree $\Gamma_q+v(\Gamma_q)$, where $b_i(q+v_i)$ is the perturbed negative puncture corresponding to $b_i(q)$; and
\item[(c)] for each $i\not=i_0$, a flow tree starting at $b_i(q+v_i)$.
\end{itemize}
In order for such a cascade to be rigid and contribute to a term that is linear in the $\tilde{\CC}(\Lambda_+)$-variables, we must have $q+v_{i}=e_+$ when $i=i_0$; this uniquely determines $q$. 
By the choice (3) of the $v_i$, together with (*), there is exactly one perturbed cascade $\tilde\Gamma$ corresponding to $\Gamma$ and it contributes to $\Phi_{\rm a}(b_1\dots b_{i_0-1})\tilde b_{i_0} \Phi_{\rm g}(b_{i_0+1}\dots b_m)$.  

Applying $\Delta$ to Equation~\eqref{eq3} and using $\Delta^{2}=0$ we find that
\[
\Phi_{\rm g}(c)+\Phi_{\rm a}(c)=\pa_-(K(\hat c))+\Omega_K(\pa_+ c).
\]
This follows from restricting to the terms without letters in $\hat\CC(\Lambda_+)$ and $\tilde\CC(\Lambda_+)$. This proves the lemma.
\end{proof}

\begin{rmk}
In \cite{EK}, the full DGA differentials of the Morse cobordisms corresponding to (L1)--(L3) were computed using abstractly perturbed flow trees. The calculation needed here is simpler: we only need to consider trees with one positive puncture at the maximum and all other punctures at the minimum. In particular, the calculations for (L1) and (L2) can be carried out with only geometric perturbations.
\end{rmk}

\subsubsection{Triple point cobordisms}

Let $L^{\Mo}_{\rm tr}$ be a triple point cobordism from $\Lambda_{+}$ to $\Lambda_{-}$ and let $(X_{\rm tr},L_{\rm tr})$ be the corresponding exact Lagrangian cobordism with cylindrical ends. Then there is a canonical identification $\CC(\Lambda_{+})\simeq\CC(\Lambda_{-})$. There are two types of (L1)-moves, denoted by (L1a) and (L1b); see Figures~\ref{fig:L1a} and \ref{fig:L1b}.

\begin{figure}[ht]
\begin{center}
\psfragscanon
\psfrag{a}{\tiny $a$}
\psfrag{b}{\tiny $b$}
\psfrag{c}{\tiny $c$}
\includegraphics[width=0.7\linewidth]{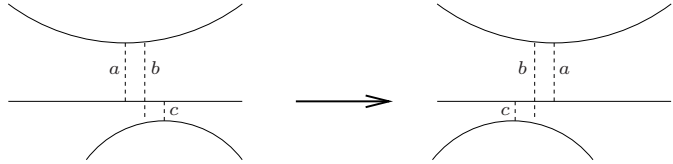}
\end{center}
\caption{An (L1a)-isotopy.}\label{fig:L1a}
\end{figure}

\begin{figure}[ht]
\begin{center}
\psfragscanon
\psfrag{a}{\tiny $a$}
\psfrag{b}{\tiny $b$}
\psfrag{c}{\tiny $c$}
\includegraphics[width=0.7\linewidth]{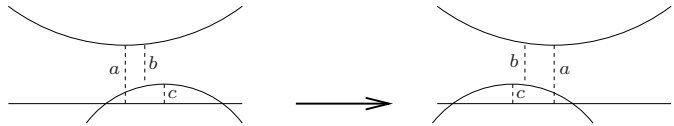}
\end{center}
\caption{An (L1b)-isotopy.}\label{fig:L1b}
\end{figure}

We define two DGA morphisms
$$\phi_{(\rm L1a)}, \phi_{(\rm L1b)}: \A(\Lambda_+)\to \A(\Lambda_-)$$ as follows:  $\phi_{(\rm L1a)}$ maps $x\in \CC(\Lambda_+)$ to the corresponding $x\in\CC(\Lambda_-)$ and $\phi_{(\rm L1b)}$ maps $a\mapsto a+bc$ and all other $x\in \CC(\Lambda_+)$ to the corresponding $x\in \CC(\Lambda_-)$.

\begin{lemma}\label{lem:triple}
The cobordism map $\Phi_{(X_{\rm tr},L_{\rm tr})}\colon\A(\Lambda_{+})\to\A(\Lambda_{-})$ is given by $\phi_{(\mathrm{L1a})}$ or $\phi_{(\mathrm{L1b})}$, as appropriate.
\end{lemma}

\begin{proof}
Follows from \cite[Lemma 6.6]{EK} in combination with Lemma~\ref{lem:geomvsabstr} and Theorem \ref{thm:flowtreecomp}.
\end{proof}

\subsubsection{Death cobordisms}

Let $L^{\Mo}_{\rm de}$ be a death cobordism from $\Lambda_{+}$ to $\Lambda_{-}$ and let $(X_{\rm de},L_{\rm de})$ be the corresponding exact Lagrangian cobordism with cylindrical ends. Then there is a canonical identification $\CC(\Lambda_{+})\simeq\CC(\Lambda_{-})\cup\{a,b\}$, where $a$ and $b$ are the canceling Reeb chords; see Figure~\ref{fig:L2}.

We define the DGA morphism
$$\phi_{(\rm L2)}: \A(\Lambda_+)\to \A(\Lambda_-)$$ as follows: Suppose $\bdry_+ a= b+v$, where $\bdry_+$ is the differential for $\mathcal{A}(\Lambda_+)$. Observe that $v$ has no terms that contain $a$ or $b$.  Then $\phi_{(\rm L2)}$ maps $a\mapsto 0$, $b\mapsto v$, and all other $x\in \CC(\Lambda_+)$ to the corresponding $x\in \CC(\Lambda_-)$.

\begin{figure}[ht]
\begin{center}
\psfragscanon
\psfrag{a}{\tiny $a$}
\psfrag{b}{\tiny $b$}
\psfrag{c}{\tiny $c$}
\includegraphics[width=0.7\linewidth]{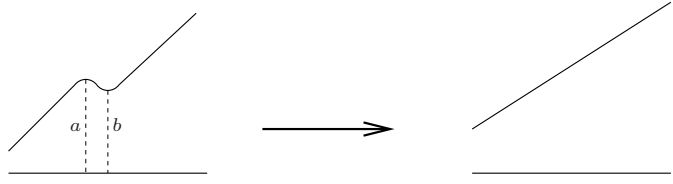}
\end{center}
\caption{An (L2)-isotopy from left to right. The (L3)-isotopy is from right to left.}\label{fig:L2}
\end{figure}

\begin{lemma}\label{lem:death}
The cobordism map $\Phi_{(X_{\rm de},L_{\rm de})}\colon\A(\Lambda_{+})\to\A(\Lambda_{-})$ is given by $\phi_{(\mathrm{L2})}$.
\end{lemma}

\begin{proof}
Follows from \cite[Lemma 6.7]{EK} in combination with  Lemma~\ref{lem:geomvsabstr} and Theorem~\ref{thm:flowtreecomp}.
\end{proof}

\subsubsection{Birth cobordisms}

Let $L^{\Mo}_{\rm bi}$ be a birth cobordism from $\Lambda_{+}$ to $\Lambda_{-}$ and let $(X_{\rm bi},L_{\rm bi})$ be the corresponding exact Lagrangian cobordism with cylindrical ends. Then there is a canonical identification $\CC(\Lambda_{+})\cup\{a,b\}\simeq\CC(\Lambda_{-})$, where $a$ and $b$ are the newly created Reeb chords.

We define the DGA morphism
$$\phi_{(\rm L3)}: \A(\Lambda_+)\to \A(\Lambda_-)$$
inductively as follows (cf.\ \cite[Remark 3.4]{Ka1}): Suppose $\bdry_- a=b+v$, where $\bdry_-$ is the differential for $\mathcal{A}(\Lambda_-)$, and
$$\mathcal{C}(\Lambda_-)=\{b_1,\dots,b_m,b,a,a_1,\dots,a_l\},$$
arranged in action-nondecreasing order.  We first set $\phi_{(\rm L3)}(b_i)=b_i$.  Suppose that
$$\bdry_- a_1= \sum B_1 bB_2b \dots B_k b A,$$
where $B_i$ is a monomial in $b_1,\dots,b_m$, $A$ is a monomial in $b_1,\dots,b_m,b,a$, and every $b$ in $A$ is preceded by an $a$ in $A$. Then
\begin{align*}
\phi_{(\rm L3)}(a_1) = & a_1 + \sum \left(B_1 a B_2b \dots B_k bA + B_1 v B_2 a B_3b \dots B_k b A\right.\\
& + B_1 v B_2 v B_3 a B_4 b \dots B_k b A + \dots + \left.B_1 v B_2 v\dots B_k a A \right).
\end{align*}
Next suppose that
$$\bdry_- a_i= \sum B_1 b B_2 b\dots B_k b A,$$
where $B_i$ is a monomial in $b_1,\dots,b_m,a_1,\dots,a_{i-1}$, $A$ is a monomial in $b_1,\dots,b_m,$ $b,a,a_1,\dots,a_{i-1}$, and every $b$ in $A$ is preceded by an $a$ in $A$.  Then
\begin{align*}
\phi_{(\rm L3)}(a_i) = & a_i + \sum \left(\overline{B}_1 a B_2b \dots B_k bA + \overline{B}_1 v \overline{B}_2 a B_3b \dots B_k b A\right.\\
& + \overline{B}_1 v \overline{B}_2 v \overline{B}_3 a B_4 b \dots B_k b A + \dots + \left. \overline{B}_1 v \overline{B}_2 v\dots \overline{B}_k a A \right),
\end{align*}
where $\overline{B}_j$ is obtained from $B_j$ by replacing each occurrence of $a_1,\dots,a_{i-1}$ by $\phi_{(\rm L3)}(a_1),\dots, \phi_{(\rm L3)}(a_{i-1})$.

\begin{lemma}\label{lem:birth}
The cobordism map $\Phi_{(X_{\rm bi},L_{\rm bi})}\colon\A(\Lambda_{+})\to\A(\Lambda_{-})$ is given by $\phi_{(\mathrm{L3})}$.
\end{lemma}

\begin{proof}
Follows from \cite[Lemma 6.8]{EK} in combination with Lemma~\ref{lem:geomvsabstr} and Theorem~\ref{thm:flowtreecomp}.
\end{proof}

\begin{rmk}
The chain maps $\phi_{(\mathrm{L1a})}$, $\phi_{(\mathrm{L1b})}$, $\phi_{(\mathrm{L2})}$, and $\phi_{(\mathrm{L3})}$ above are precisely the chain maps used by Chekanov \cite{Ch} to prove the invariance of Legendrian contact homology under the Legendrian Reidemeister moves. The above exact Lagrangian cobordisms can be interpreted as providing a geometric context where these maps arise naturally.
\end{rmk}

\subsection{Minimum cobordisms} \label{subsub:minimum}

In this subsection and the next we consider two types of exact Lagrangian cobordisms that correspond to single Morse modifications of a Legendrian link.

Let $\Lambda_{+}\subset\R^{3}$ be a Legendrian link, one of whose components is the standard Legendrian unknot $U$ such that $\Pi_\C(U)$ is contained in a disk which is disjoint from $\Pi_\C(\Lambda_+-U)$. Let $\Lambda_{-}=\Lambda_{+}-U$.

\begin{defn}\label{def:mincob}
A \emph{minimum cobordism} from $\Lambda_{+}$ to $\Lambda_{-}$ is a Morse cobordism which is the union of a trivial Morse cobordism from $\Lambda_{+}-U$ to $\Lambda_{-}$ and a disk with boundary~$U$ in $J^{1}(\R\times[0,1])$ with front in $J^{0}(\R\times[0,1])$ as shown in Figure \ref{fig:mincob}.
\end{defn}

\begin{figure}[ht]
\labellist
\small
\pinlabel $a$ at 530 310
\pinlabel $\xi_1$ at -14 20 
\pinlabel $\xi_2$ at 720 60
\pinlabel $z$ at 50 530
\endlabellist
\centering
\includegraphics[width=0.5\linewidth]{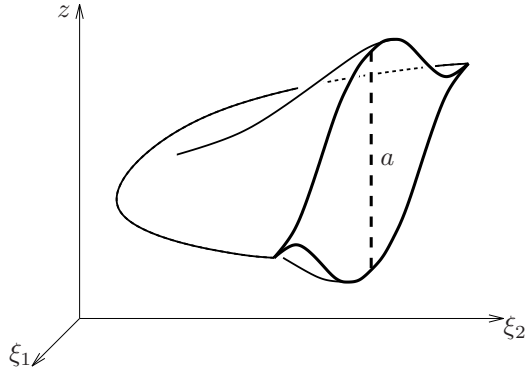} 
\caption{A minimum cobordism.}\label{fig:mincob}
\end{figure}

Let $L_{\rm mi}^{\Mo}$ be a minimum cobordism from $\Lambda_{+}$ to $\Lambda_{-}$ and let $(X_{\rm mi}, L_{\rm mi})$ be the corresponding exact Lagrangian cobordism with cylindrical ends. By assumption there are no Reeb chords of $\Lambda_{+}$ that connect $U$ to any other component. Let $a$ denote the unique Reeb chord from $U$ to itself. Then $\CC(\Lambda_{+})=\CC(\Lambda_{-})\cup\{a\}$, where the identification is induced by the trivial cobordisms on $\Lambda_{+}-U=\Lambda_{-}$.

\begin{lemma} \label{lem:minimum}
The cobordism map $\Phi_{(X_{\rm mi},L_{\rm mi})}\colon\A(\Lambda_{+})\to\A(\Lambda_{-})$ is given by
\[
\Phi_{(X_{\rm mi},L_{\rm mi})}(c)=
\begin{cases}
0 &\text{if } c=a,\\
c &\text{if } c\ne a,
\end{cases}
\]
where $c\in\CC(\Lambda_{+})$.
\end{lemma}

\begin{proof}
This is an immediate consequence of $|a|=1$ and the fact that the rigid holomorphic disks in a trivial cobordism are strips over Reeb chords.
\end{proof}

\subsection{Saddle cobordisms} \label{subsub:saddle}

In this subsection we treat {\em saddle cobordisms}.

\subsubsection{Contractibility}

\begin{defn}[Contractible Reeb chord] \label{defn:contractible}
Let $\Lambda\subset\R^{3}$ be a Legendrian knot. A Reeb chord $a\in\CC(\Lambda)$ is \emph{contractible} if there exists a homotopy $\Lambda_{\tau}$, $0\le \tau\le 1$, of Legendrian immersions such that:
\begin{itemize}
\item $\Lambda_0=\Lambda$;
\item $\Lambda_{\tau}$, $\tau\in [0,1]$, is a $\Pi_\C$-simple, i.e., $\Pi_\C (\Lambda_\tau)$ has only transverse double points for all $\tau\in[0,1]$; and
\item $\Lambda_{1}$ has a transverse self-intersection which is obtained by sending $\mathfrak{A}(a_\tau)\to 0$ as $\tau\to 1$, where $a_\tau\in \mathcal{C}(\Lambda_\tau)$ is the Reeb chord corresponding to $a$.
\end{itemize}
\end{defn}

Let $\Lambda_+'\subset\R^{3}$ be a Legendrian link with a contractible Reeb chord $a$. Then, after Legendrian isotopy, we obtain $\Lambda_+\subset \R^3$ with a contractible Reeb chord $a$, whose neighborhood is as shown on the left-hand side of Figure \ref{fig:saddleends}. Let $\Lambda_{-}$ denote the Legendrian link obtained by modifying the front of $\Lambda_{+}$ as shown on the right-hand side of Figure \ref{fig:saddleends}.

\begin{figure}[ht]
\labellist
\small
\pinlabel $a$ at 210 110
\pinlabel $\Lambda_+$ at 190 0
\pinlabel $\Lambda_-$ at 860 0
\endlabellist
\centering
\includegraphics[width=0.7\linewidth]{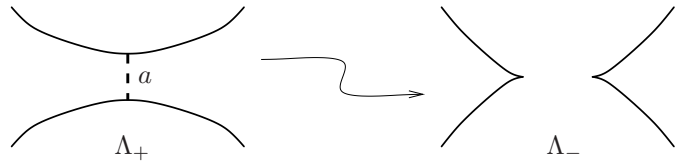}
\caption{Legendrian links at the ends of a saddle cobordism, viewed in the front projection.}
\label{fig:saddleends}
\end{figure}

Now we consider the Lagrangian projection $\Pi_\C$. The modification of $\Lambda_+$ given in Figure~\ref{fig:saddleends} corresponds to a $0$-resolution of $\Pi_\C(\Lambda_+)$ at the crossing $\Pi_\C(a)$.  Also the immersed Legendrian isotopy can be continued so that a crossing of the opposite sign emerges where $\Pi_\C(a)$ used to be (see Figure~\ref{crossingchange}).

\begin{figure}[ht]
\begin{overpic}[width=0.7\linewidth] {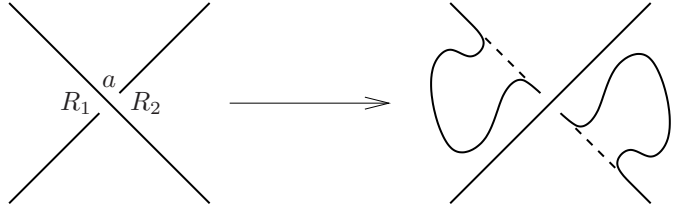}
\put(8,14.4){\small $R_1$}
\put(18.5,14.4){\small $R_2$}
\put(14.3,17.8) {\small $a$}
\end{overpic}
\caption{Crossing change at a contractible Reeb chord in the Lagrangian projection.}\label{crossingchange}
\end{figure}

We state a simple necessary and sufficient condition for contractibility with respect to the Lagrangian projection. We refer to a component of $\C-\Pi_\C(\Lambda_+)$ as a ``region''.

\begin{lemma}[Criterion for contractibility] \label{condition for contractibility}
Let the quadrants with positive Reeb sign at $a\in \mathcal{C}(\Lambda_+)$ belong to the regions $R_1$ and $R_2$. The Reeb chord $a$ is contractible if and only if it is possible to apply a $\Pi_\C$-simple isotopy to $\Lambda_+$ so that the areas of both $R_1$ and $R_2$ exceed the action $\mathfrak{A}(a)$. (If $R_1=R_2$, the condition becomes $A(R_1)>2\cdot\mathfrak{A}(q)$.)
\end{lemma}

\begin{proof}
If $a$ contracts, then its action $\mathfrak{A}(a)$ approaches $0$ and eventually becomes smaller than the area of any region in the diagram. If $R_1$ and $R_2$ have large enough areas, then it is possible to carry out the isotopy shown in Figure~\ref{crossingchange}.
\end{proof}

\subsubsection{Saddle cobordisms and simplicity}

\begin{defn}\label{def:saddlecob}
Let $a$ be a contractible Reeb chord. Then a \emph{saddle cobordism} $L_{\rm sa}^{\rm Mo}$ from $\Lambda_{+}$ to $\Lambda_{-}$ corresponding $a$ is a Morse cobordism for which there exists an open set $V\subset L^{\rm Mo}_{\rm sa}$ such that:
\begin{enumerate}
\item the front of $V$, viewed as a subset of $J^{0}(\R\times[0,1])$, is as in Figure~\ref{fig:saddlecob}; and
\item $L^{\rm Mo}_{\rm sa}-V$ is {\em locally trivial}, i.e., there is an open cover of $L^{\rm Mo}_{\rm sa}-V$ such that each open set of the cover is the restriction of a trivial Morse cobordism.
\end{enumerate}
\end{defn}

\begin{figure}[ht]
\labellist
\small
\pinlabel $a$ at 560 340
\pinlabel $\xi_1$ at -14 20 
\pinlabel $\xi_2$ at 780 60
\pinlabel $z$ at 50 530
\endlabellist
\centering
\includegraphics[width=0.5\linewidth]{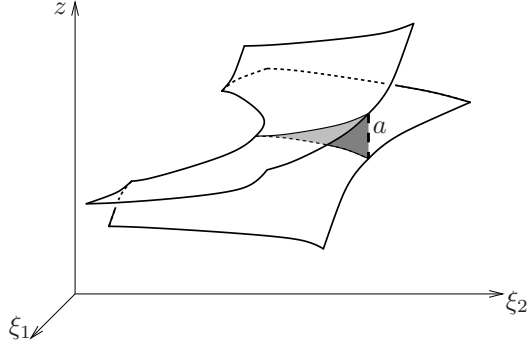} 
\caption{A saddle cobordism and a gradient flow line $\rho_0$ from $a$ to $1$.}
\label{fig:saddlecob}
\end{figure}

If $L_{\rm sa}^{\rm Mo}$ is a saddle cobordism from $\Lambda_+$ to $\Lambda_-$ corresponding to $a$, then we write $(X_{\rm sa},L_{\rm sa})$ for the corresponding exact Lagrangian cobordism with cylindrical ends.

For $c\ne a$, let
$$\MM(c,a^k;\bb)=\MM^{(\R\times \R^3,\R\times\Lambda_+)}(c,a^k;\bb)$$
be the moduli space of holomorphic disks in $\R\times \R^3$ with boundary on $\R\times\Lambda_+$, one positive puncture at $c$, $k$ positive punctures at $a$, and negative punctures at $\bb$.

We make the following simplifying definition:

\begin{defn} \label{defn: simple contractible}
A contractible Reeb chord $a\in \CC(\Lambda_+)$ is \emph{simple} if $\op{ind}(u)\geq k$ for all broken holomorphic disks $u$ in $\R\times \R^3$ with boundary on $\R\times \Lambda_+$, one positive puncture at a chord $c\ne a$, and $k>1$ positive punctures at $a$. A saddle cobordism for a simple contractible Reeb chord is a {\em simple saddle cobordism}.
\end{defn}

\subsubsection{The DGA morphism for a simple saddle cobordism}

In view of the identification $\CC(\Lambda_+)\simeq \CC(\Lambda_{-})\cup\{a\}$, we define an algebra map
$$\Psi_0:\A(\Lambda_{+})\to\A(\Lambda_{-})$$
on the generators by setting $\Psi_0(a)=1$, $\Psi_0(b)=b$ for $b\in\CC(\Lambda_{-})$, and $\Psi_0(A)=A$ for $A\in H_{1}(L_{\rm sa})$. We also define
$$\Psi_1:\mathcal{C}(\Lambda_{+})\to\A(\Lambda_{-})$$
by setting
$$\Psi_1(c)=\sum_{\dim(\MM(c,a;\bb))=1}|\MM(c,a;\bb)/\R|\cdot\Psi_0(\bb),$$
for $c\in\CC(\Lambda_{-})$ and $\Psi_1(a)=0$.

\begin{rmk}
We may assume that the moduli spaces $\MM(c,a;\bb)$ are transversely cut out, since the disks in $\MM(c,a;\bb)$ are not multiply-covered. Moreover, the bijection between flow trees with boundary on $\R\times\Lambda_+$ (resp.\ $\tilde L^{\Mo}_{\rm sa}$) and disks with boundary on $\R\times\Lambda_+$ (resp.\ $\tilde L^{\Mo}_{\rm sa}$) holds also for disks in $\MM(c,a;\bb)$ of index $1$.
\end{rmk}

\begin{prop}\label{prop:saddle}
If $a\in \mathcal{C}(\Lambda_+)$ is a simple contractible Reeb chord, then the corresponding cobordism map $\Phi_{(X_{\rm sa},L_{\rm sa})}\colon\A(\Lambda_+)\to\A(\Lambda_-)$ is given by
$$\Phi_{(X_{\rm sa},L_{\rm sa})}(c)=\Psi_0(c)+\Psi_1(c).$$
\end{prop}

We often refer to the map $\Phi_{(X_{\rm sa},L_{\rm sa})}$ as the {\em $0$-resolution map at $a$.}

\begin{proof}
Let $L^{\Mo}_{\rm sa}$ be a saddle cobordism from $\Lambda_{+}$ to $\Lambda_{-}$, let $a$ be the simple contractible chord in $\Lambda_{+}$, and let $\Lambda_1$ be the immersed Legendrian with a double point $a'$ corresponding to the Reeb chord $a$ of length $0$. Without loss of generality suppose that $F=\R\times[-1,1]$. We write $\hat c\in \mathcal{C}(\Lambda_+)$ for the Reeb chord corresponding to $c\in\mathcal{C}(\Lambda_-)$.

First observe that there is a unique flow tree $\rho_0$ of $L^{\Mo}_{\rm sa}$ from $a$ to $1$ --- it is the flow line given in Figure~\ref{fig:saddlecob}. The flow tree $\rho_0$ will be called the {\em basic tree} and the corresponding holomorphic disk the {\em basic disk}.

\s\n {\em Step 1.}
We describe a degeneration $L_t$, $t\in[0,1]$, of $L_0=L^{\Mo}_{\rm sa}$ to the immersed exact Lagrangian $L_1=\Lambda_1\times[-1,1]$.

Let $\Lambda_t$, $t\in [0,1]$, be a regular homotopy of Legendrian immersions with $\Lambda_0=\Lambda_+$, which is guaranteed by Definition~\ref{defn:contractible}. We may assume that:
\begin{itemize}
\item $\Pi_{J^0\R}(\Lambda_t)$, $t\in[0,1]$, is independent of $t$ outside a small rectangle $R\subset J^0\R$;
\item $\Pi_{J^0\R}(\Lambda_t)\cap R$, $t<1$, is as shown on the left-hand side of Figure~\ref{fig:saddleends};
\item $\Pi_{J^0\R}(\Lambda_t)$ limits to $\Pi_{J^0\R}(\Lambda_1)$ as $t\to 1$; in particular, $\mathfrak{A}(a_t)\to 0$ as $t\to 1$, where $a_t$ is the contractible chord corresponding to $a$.
\end{itemize}

Next let $L_t$, $t<1$, be a family of saddle cobordisms from $\Lambda_t$ to $\Lambda_-$ satisfying the following:
\begin{itemize}
\item the $L_t$ converge to $L_1$ as $t\to 1$;
\item if $V_t$, $t<1$, is the open set $V$ for $L_t$ which appears in Definition~\ref{defn:contractible}, then the $V_t$ shrink to a point as $t\to 1$.
\end{itemize}

\s\n {\em Step 2.} Let $\rho_0(t)$, $t<1$, be the basic tree for $L_t$.

\begin{claim} \label{claim: seq of flow trees}
A sequence of rigid flow trees $\gamma_t\not=\rho_0(t)$ of $L_t$, $t<1$, with one positive end $\hat c\not=a_t$ and $k$ negative ends $b_1,\dots,b_k$, has a subsequence which converges to a cascade $\Gamma$ of $L_1$ as $t\to 1$, which consists of the following:
\begin{itemize}
\item a connector from $\hat c$ to $c\times\{0\}$;
\item a slice tree $\Gamma_1\subset \Lambda_1\times\{0\}\subset L_1$; and
\item connectors from $b_i\times\{0\}$ to $b_i$ for $i=1,\dots,k$.
\end{itemize}
\end{claim}

\begin{proof}
By the flow tree analog of Gromov compactness, there is a subsequence of $\gamma_t$ which converges to a cascade $\Gamma$ of $L_1$.  (Here the boundary condition $L_t$ varies with $t$ and Gromov compactness corresponds to local convergence of flow lines of the vector fields defined by $L_t$ to the flow lines of the limiting vector field as $t\to 1$, provided an energy bound is satisfied.)  If a level of $\Gamma$ contains a slice tree of $\Lambda_1\times\{\xi_2\}$ with $\xi_2\not=0$, then $\gamma_t$ is not rigid.  The claim follows.
\end{proof}

\s\n {\em Step 3.} Consider the slice tree $\Gamma_1$. Since the Reeb chords $a_t$ limit to $a_1$ with $\frak{A}(a_1)=0$, we can view $\Gamma_1$ as a tree with punctures at $a_1$, denoted by $\Gamma_1^\circ$. For each negative (resp.\ positive) puncture at $a_1$, there is a corresponding end (resp.\ switch) of $L_t$, $t<1$; see Figure~\ref{fig:pushdown}.

\begin{figure}[ht]
\labellist
\pinlabel {\small end} at 120 170
\pinlabel {\small switch} at 900 190
\pinlabel {\small $\Pi_F(\Sigma)$} at 175 10
\pinlabel {\small $\Pi_F(\Sigma)$} at 880 10
\endlabellist
\centering
\includegraphics[width=0.6\linewidth]{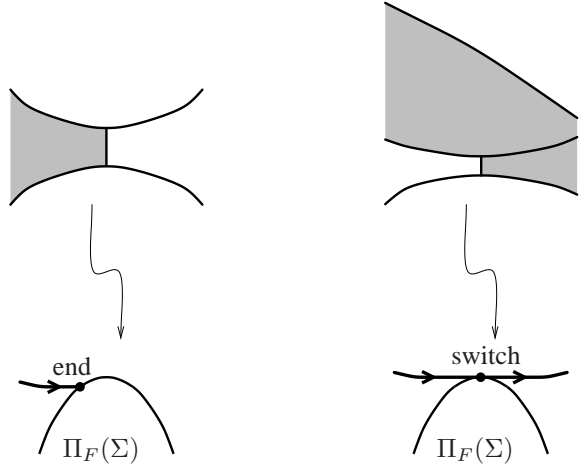}
\caption{Producing rigid trees in a saddle cobordism. Here $\Sigma$ is the singular set of $\tilde L^{\rm Mo}_{\rm sa}$. The arrows in the bottom figures represent the negative gradients of the positive function differences.}
\label{fig:pushdown}
\end{figure}

We now apply Equation~\eqref{eqn: dim of tree alt} to compute $\dim (\Gamma_1^\circ)$ and $\dim (\gamma_t)$: Since $n=1$,
$$\dim (\Gamma_1^\circ) = -2 + I(c) -\sum_{i=1}^k (I(b_i)-1) +\sum_r \mu(r).$$
On the other hand, since $n=2$, $I(\hat c)=I(c)+1$, and each positive end at $a_t$ is converted into a switch which contributes $-1$ to the Maslov content,
$$\dim (\gamma_t) =-1 + I(c)-\sum_{i=1}^k (I(b_i)-1)+\sum_r \mu(r) -k.$$
Since $\dim (\gamma_t)=0$, it follows that $\dim (\Gamma_1^\circ)= k-1$.  By the simplicity of the chord $a_t$, there is no such holomorphic disk when $k>1$.

Note that for $t$ close to $1$, the rigid flow tree $L_t$ must intersect the basic flow tree $\rho_0$ as well as $\Pi_F(\Sigma)$.

\s\n {\em Step 4.}
It remains only to construct a unique rigid tree $\gamma_t$ in $L_t$, $1-\varepsilon<t<1$, corresponding to the cascade $\Gamma$.  When we perturb $\Gamma$ to $\gamma_t$, the unique positive puncture of the slice tree $\Gamma_1^\circ$ at $a$ becomes a switch of $\gamma_t$ as depicted on the right-hand side of Figure~\ref{fig:pushdown}.  The tree $\gamma_t$ can be split into two partial flow trees $\gamma_t^+$ and $\gamma_t^-$ at $d$, where $\gamma_t^+$ has a positive puncture at $\hat c$ and $\gamma_t^-$ has no positive punctures; similarly, $\Gamma$ can be split into partial flow trees $\Gamma^+$ and $\Gamma^-$. At the switching point $d$, the incoming and outgoing gradient trajectories are uniquely determined.  Hence there is a unique $\gamma_t^+$ which is close to $\Gamma^+$ and ends at $d$ and there is a unique $\gamma_t^-$ which is close to $\Gamma^-$ and starts at $d$.  This proves the lemma.
\end{proof}

\begin{rmk}
Consider a rigid tree $\gamma$ in the cobordism $L^{\rm Mo}_{\rm sa}$ with one positive end $c\not=a$. By applying ``boundary gluing'' to $\gamma$ and the basic tree $\rho_0$, we create a tree with two positive ends $c$ and $a$. Here a ``boundary gluing'' of two flow trees is the flow tree analog of a gluing of two disks whose boundaries intersect at a point on the Lagrangian. The end and the switch of this tree cancel as in Figure \ref{fig:bdryglue}, the tree moves upwards, and we eventually obtain a $1$-dimensional family of trees in $\R\times \Lambda_+$ with two positive ends $c$ and $a$.
\end{rmk}

\begin{figure}[ht]
\labellist
\pinlabel $1$ at 41 943
\pinlabel $2$ at 816 943
\pinlabel $3$ at 41 367
\pinlabel $4$ at 816 367
\pinlabel \small{$\Pi_F(\Sigma)$} at 90 -30
\pinlabel \small{$\Pi_F(\Sigma)$} at 90 545
\pinlabel \small{$\Pi_F(\Sigma)$} at 870 -30
\pinlabel \small{$\Pi_F(\Sigma)$} at 870 545
\pinlabel \small{boundary glue} at 400 990
\pinlabel \small{$\rho_0$} at 200 900
\pinlabel \small{$\gamma$} at 50 700
\pinlabel \small{sw} at 150 660
\pinlabel \small{e} at 232 660
\pinlabel \small{sw} at 930 660
\pinlabel \small{e} at 1012 660
\pinlabel \small{collide} at 1020 520
\pinlabel \small{can now move up} at 570 150
\endlabellist
\centering
\includegraphics[width=0.63\linewidth]{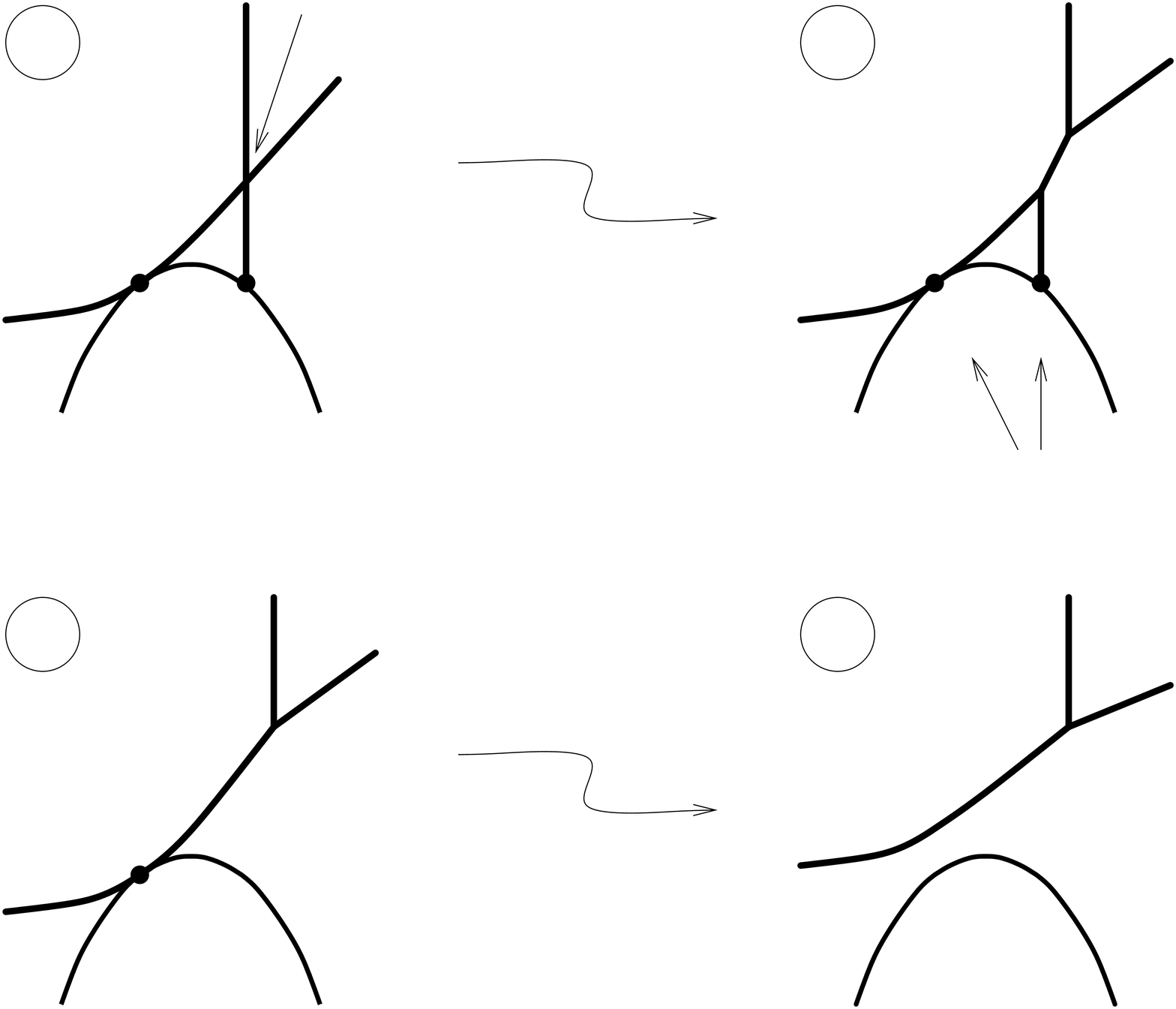}
\s
\caption{Boundary gluing rigid trees in the cobordism. ``sw'' refers to a switch and ``e'' refers to an end.}
\label{fig:bdryglue}
\end{figure}

\begin{cor} \label{surjective}
If $a\in \mathcal{C}(\Lambda_+)$ is a simple contractible Reeb chord, then the corresponding cobordism map $\Phi_{(X_{\rm sa},L_{\rm sa})}\colon\A(\Lambda_+)\to\A(\Lambda_-)$ is surjective.
\end{cor}

\begin{proof}
This follows from Proposition~\ref{prop:saddle}, together with the observation that
$$\mathfrak{A}(\Psi_0(c))> \mathfrak{A}(\Psi_1(c)),\quad c\not=a.$$ Here $\mathfrak{A}(\Psi_1(c))$ is the supremum of $\mathfrak{A}({\bf b})$ for all the nonzero monomial summands ${\bf b}$ of $\Psi_1(c)$.
\end{proof}

\subsubsection{Pushing forward augmentations}

If $\Phi:\mathcal{A}(\Lambda_+)\to \mathcal{A}(\Lambda_-)$ is a DGA morphism and $\varepsilon'$ is an augmentation of $\Lambda_-$, then the pullback $\Phi^*\varepsilon':=\varepsilon'\circ \Phi$ is naturally an augmentation of $\Lambda_+$.

For saddle cobordisms corresponding to a simple contractible Reeb chord, augmentations behave well under {\em pushforward}:

\begin{lemma}\label{newaug}
Let $a\in \mathcal{C}(\Lambda_+)$ be a simple contractible Reeb chord, let $\varepsilon$ be an augmentation of $\Lambda_+$ with $\varepsilon(a)=1$, and let $\Phi=\Phi_{(X_{\rm sa},L_{\rm sa})}$ be the corresponding DGA morphism. Then there is a unique algebra map $\varepsilon':\mathcal{A}(\Lambda_-)\to \F$ such that $\varepsilon=\varepsilon'\circ \Phi$. Furthermore, $\varepsilon'$ is an augmentation of $\Lambda_-$.
\end{lemma}

\begin{proof}
The existence and uniqueness of $\varepsilon'$ uses the same key idea as that of Corollary~\ref{surjective}.  Since $\Phi(x)=\Psi_0(x)+\Psi_1(x)$ with $\mathfrak{A}(\Psi_0(x))>\mathfrak{A}(\Psi_1(x))$, we can write $\varepsilon'(\Psi_0(x))=\varepsilon(x) -\varepsilon'(\Psi_1(x))$ and use induction on the action.

To prove that $\varepsilon'\bdry'=0$, we compute
\begin{align*}
(\varepsilon'\bdry')(\Psi_0(x))&=(\varepsilon'\bdry') (\Phi(x)-\Psi_1(x))=(\varepsilon'\Phi\bdry)(x)-(\varepsilon'\bdry')(\Psi_1(x))\\
&= (\varepsilon\bdry)(x)-(\varepsilon'\bdry')(\Psi_1(x))= (\varepsilon'\bdry')(-\Psi_1(x)),
\end{align*}
and use induction on the action.
\end{proof}

\subsubsection{Dipped diagrams}
\label{subsub:dipped}

In this subsection we explain how to modify a Legendrian link $\Lambda$ with a contractible Reeb chord $a$ so that $a$ becomes simple.  The modification is called {\em dipping}, which first appeared in \cite{F} under the name ``splashing'' and was used extensively in \cite{Sa2}.

\begin{lemma} \label{dipping}
Let $\Lambda$ be a Legendrian link with a contractible Reeb chord $a$. Then there exist a Legendrian isotopy $\Lambda_\tau$, $\tau\in[0,1]$, and a $1$-parameter family $a_\tau$, $\tau\in[0,1]$, of contractible Reeb chords such that:
\begin{enumerate}
\item  $\Lambda_0=\Lambda$ and $a_0=a$;
\item $\Pi_\C(\Lambda_\tau)$ is a transverse intersection at $\Pi_\C(a_\tau)$ for all $\tau\in[0,1]$; and
\item $a_1\in \mathcal{C}(\Lambda_1)$ is simple.
\end{enumerate}
\end{lemma}

\begin{proof}
We start with a $\Pi_\C$-simple isotopy $\Lambda_\tau$, $\tau\in [0,{1\over 3}]$, with $\Lambda_0=\Lambda$, such that the restriction of the front of $\Lambda_{1/3}$ to $J^0 I$, where $I$ is a small interval around the $x$-coordinate of $a_{1/3}$, is as in Figure \ref{fig:dippedinitial}, with the exception of (*):
\begin{itemize}
\item[(*)] in the front projection, the slopes of the Legendrian arcs above (resp.\ below) $\Pi_{J^0F} (a_{1/3})$ are negative (resp.\ positive) and decrease as the $z$-coordinate increases.
\end{itemize}
\begin{figure}[ht]
\labellist
\pinlabel \small{$a_{1/3}$} at 230 363
\endlabellist
\centering
\includegraphics[width=0.15\linewidth]{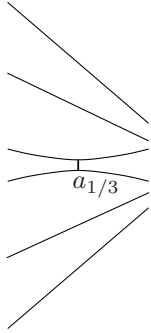}
\caption{Aligning all the local gradients. }
\label{fig:dippedinitial}
\end{figure}
The isotopy is obtained from the $\Pi_\C$-simple isotopy given by Definition~\ref{defn:contractible}, by applying a $C^0$-small perturbation near the Reeb chord $a_{1/3}$. Next we shrink $I$ to $I'\ni x(a_{1/3})$ and apply a $C^0$-small isotopy to $\Pi_{J^0F}(\Lambda_{1/3})\cap J^0 I'$ so that (*) holds. This yields $\Lambda_\tau$, $\tau\in[{1\over 3},{2\over 3}]$.

Finally we apply a $C^0$-small isotopy to the front projection $\Pi_{J^0F}(\Lambda_{2/3})\cap J^0I'$ so that the resulting front $\Pi_{J^0F}(\Lambda_1)\cap J^0I'$ is as in Figure \ref{fig:dippedfront}. The strands are numbered $$-m,\dots,-1,0_-,0_+,1,\dots,n$$ from bottom to top. Let $I_{L,i}\subset I'$ (resp.\ $I_{R,i}\subset I'$) be the support of the perturbation of strand $i$ to the left (resp.\ right) of $x(a_1)$.  Then $I_{L,i}\cap I_{L,j}=\varnothing$ for $i\not=j$ and the $I_{L,i}$ move from left to right as $i$ increases; the same holds for $I_{R,i}$. The resulting Legendrian isotopy will be denoted by $\Lambda_\tau$, $\tau\in[{2\over 3},1]$.
\begin{figure}[ht]
\begin{center}
\psfragscanon
\psfrag{A}{\tiny $n$}
\psfrag{B}{\tiny $n-1$}
\psfrag{C}{\tiny $0_+$}
\psfrag{D}{\tiny $0_-$}
\psfrag{E}{\tiny $-m+1$}
\psfrag{F}{\tiny $-m$}
\includegraphics[width=0.25\linewidth]{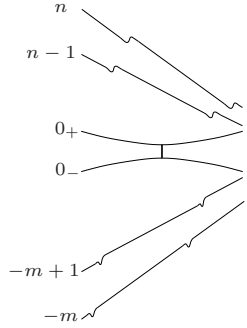}
\end{center}
\caption{Isolating the contractible chord by dipping.}
\label{fig:dippedfront}
\end{figure}
The reason this isotopy $\Lambda_\tau$, $\tau\in[{2\over 3},1]$, is called a {\em dipping} is self-evident when the dipping is viewed in the Lagrangian projection; see Figure~\ref{fig:lagdips}.
\begin{figure}[ht]
\begin{overpic}[width=1.0\linewidth]{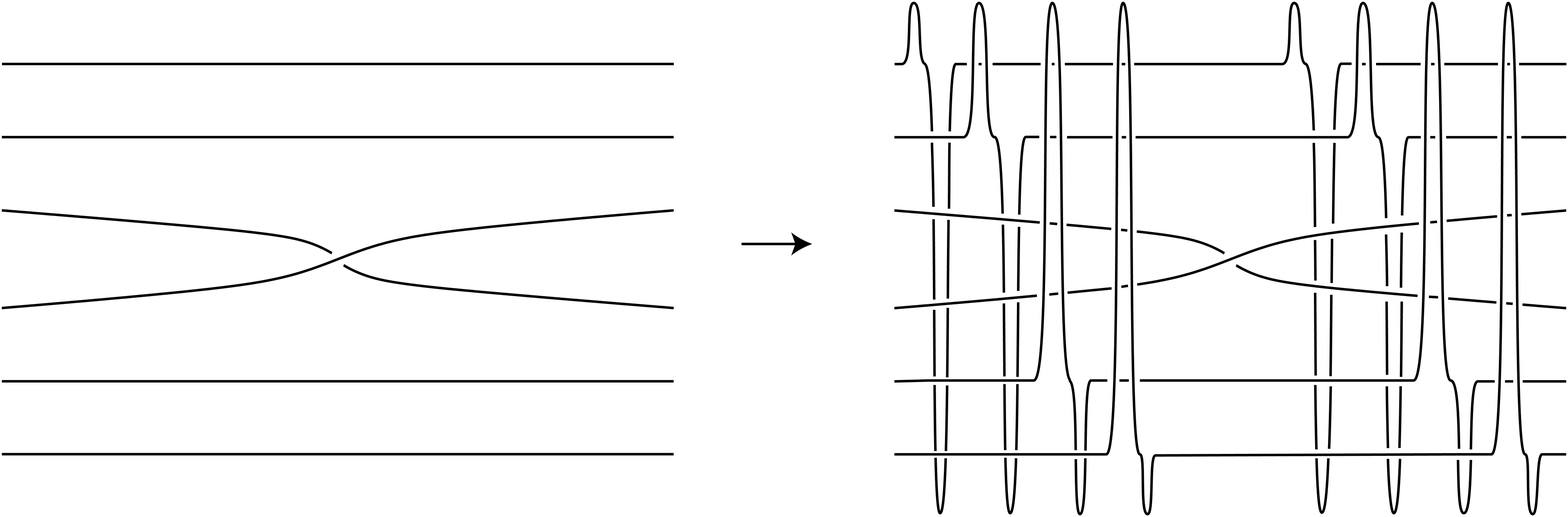}
\end{overpic}
\caption{A dipping in the Lagrangian projection.}
\label{fig:lagdips}
\end{figure}
Conditions (1) and (2) immediately hold by construction.

Next we claim that if a tree $\gamma$ of $\Lambda_1$ has one positive puncture at $c\not=a_1$ and $k$ positive punctures at $a_1$, then its formal dimension is at least $k-1$. Recall that a tree has formal dimension $k-1$ if and only if the Fredholm index of the corresponding disk is $k$. The reader might prefer to translate the proof below to the situation of a Lagrangian projection. For each positive puncture of $\gamma$ at $a_1$ with a ${\pi\over 2}$-angle, there is a partial flow tree of the following type:
\begin{itemize}
\item $\gamma'$, which consists of a partial flow line $(\gamma|_e,\gamma^e_1,\gamma^e_2)$ where $\Pi_{J^0F}$ maps $\gamma^e_2$ to the upper sheet $S_2$ of $a_1$ and $\gamma^e_1$ to a sheet $S_3$ above $a_1$ for $x\leq x(a_1)$ and a partial flow line $(\gamma|_{e'},\gamma^{e'}_1,\gamma^{e'}_2)$ where $\Pi_{J^0F}$ maps $\gamma^{e'}_2$ to the lower sheet $S_1$ of $a_1$ and $\gamma^{e'}_1$ to $S_3$ for $x\geq x(a_1)$. $\gamma'$ is given by the upper right diagram of Figure~\ref{fig:pushdown}.
\item $\gamma''$, which is obtained from $\gamma'$ by reflecting across the $x$-axis.
\end{itemize}
Note that there are no $Y_1$-vertices, switches, and ends on $J^0I'$, since there are no cusps on $J^0I'$. A $Y_0$-vertex and a $2$-valent interior puncture with a ${3\pi\over 2}$-angle both increase the formal dimension by $1$.

Let $\gamma'\subset \gamma$ be a partial flow tree with a positive $2$-valent interior puncture at $a_1$.  We consider the continuation of $\gamma'$ to the left; the case of $\gamma''\subset\gamma$ is similar.  If there is no $Y_0$-vertex and no $2$-valent interior puncture with a ${3\pi\over 2}$-angle, then by a case-by-case analysis there is always a positive puncture $c\not=a_1$ whose corresponding Reeb chord starts at $0_+$.  Hence the continuation of $\gamma'$ to the left has one of the following: (i) a positive puncture $c\not=a_1$, (ii) a $Y_0$-vertex above $0_+$, when viewed in the front projection, or (iii) a $2$-valent interior puncture with a ${3\pi\over 2}$-angle, whose corresponding Reeb chord starts at or above $0_+$. Moreover, for each positive $2$-valent puncture of $\gamma$ at $a_1$, the corresponding (i), (ii) or (iii) is distinct.  Since $\gamma$ has only one positive puncture $c\not=a_1$ and each $Y_0$-vertex or ${3\pi\over 2}$-angle increases the formal dimension by $1$, it follows that $\dim(\gamma)\geq k-1$. This proves the claim and the lemma.
\end{proof}

\subsubsection{The DGA morphism for a general saddle cobordism}

Let $L$ be a general saddle cobordism from $\Lambda_+$ to $\Lambda_-$ with contractible $a\in \mathcal{C}(\Lambda_+)$. The DGA morphism $\Phi$ for $L$ can be computed as follows: Let $\Lambda_{\tau}$ and $a_\tau$, $\tau\in[0,1]$, be the isotopies from the proof of Lemma~\ref{dipping} such that $\Lambda_0=\Lambda_+$, $a_0=a$, and $a_1\in \mathcal{C}(\Lambda_1)$ is simple and contractible. We then resolve $a_1$ to obtain a simple saddle cobordism from $\Lambda_1$ to $\Lambda_1'$.  Finally we ``undo'' the isotopy $\Lambda_\tau$, $\tau\in[0,1]$, to obtain $\Lambda'_\tau$, $\tau\in[0,1]$, such that $\Lambda'_0=\Lambda_-$. This gives a composition of three cobordisms with corresponding DGA morphism
$$\Phi_0\circ \Phi_1\circ \Phi_2\colon \A(\Lambda_{+})\to\A(\Lambda_{-}),$$
where $\Phi_0$ and $\Phi_2$ correspond to Legendrian isotopies and can be computed using Lemmas~\ref{lem:triple},~\ref{lem:death}, and~\ref{lem:birth}, and $\Phi_{\rm sa}$ corresponds to a simple saddle cobordism and can be computed using Proposition~\ref{prop:saddle}. Since the composition of the three cobordisms is isotopic to $L$,  $\Phi$ is chain homotopic to $\Phi_0\circ \Phi_1\circ \Phi_2$ by Lemma \ref{lem:chainhomotopy}.

\section{Exact Lagrangian fillings and augmentations}
\label{section: exact Lagrangian fillings}

In this section we collect some general results on exact Lagrangian fillings and augmentations.

\subsection{$tb$, $r$, and slice genus}

The following is due to Chantraine~\cite[Theorem~1.3]{Cha}:

\begin{thm}[Chantraine]
If the Legendrian knot $\Lambda$ admits an orientable Lagrangian filling $L$, then:
\begin{enumerate}
\item $g(L)=g_s(\Lambda)$; and
\item $tb(\Lambda)=2g_s(\Lambda)-1=2g(L)-1$ and $r(\Lambda)=0$.
\end{enumerate}
Here $g(L)$ is the genus of $L$, $g_s(\Lambda)$ is the slice genus of $\Lambda$, and $tb(\Lambda)$ and $r(\Lambda)$ are the Thurston-Bennequin invariant and rotation number of $\Lambda$.
\end{thm}

\begin{proof}[Sketch of proof.]
This is obtained by combining the fact that the tangent bundle of $L$ is isomorphic to its normal bundle, the relation between linking on the boundary and intersections in the interior, and the slice Thurston--Bennequin inequality.
\end{proof}

\subsection{Restriction on linearized Legendrian contact homology}

Suppose $\Lambda$ admits an exact Lagrangian filling $(\R\times\R^3,L)$. Then the DGA morphism $$\varepsilon=\Phi_{(\R\times\R^3,L)}:\mathcal{A}(\Lambda)\to \F$$
is an augmentation of $\Lambda$. The augmentation $\varepsilon$ gives a linearization of $\A(\Lambda)$ and we denote the corresponding $\varepsilon$-linearized Legendrian contact homology group (resp.\ chain complex) by $HC^{\varepsilon}(\Lambda)$ (resp.\ $CC^\varepsilon(\Lambda)$).

The following theorem is essentially due to Seidel.  The proof, given in more detail in \cite{E3}, uses wrapped Floer homology (cf.\ \cite{AS,FSS}).

\begin{thm}[Seidel] \label{Seidel theorem}
$HC^\varepsilon(\Lambda)\simeq H_*(L)$.
\end{thm}

\begin{proof}[Sketch of proof.]
Let $CF^{\rm wr}(L)$ be the wrapped Floer chain complex of $L$.  In analogy with the Morse--Bott description of symplectic homology from \cite{BO}, there is a chain map
$$\Psi:CC^{\varepsilon}(\Lambda)\to CF^{\rm wr}(L)$$
that counts perturbed holomorphic strips, along which a Hamiltonian term is turned on as we go from the positive end to the negative end. The chain map  $\Psi$ is a quasi-isomorphism from $CC^{\varepsilon}(\Lambda)$ to the positive energy part $CF^{{\rm wr},+}(L)$ of the wrapped Floer complex by an action filtration argument. In particular, the leading term of $\Psi(a)$, $a\in\mathcal{C}(\Lambda)$, in the action filtration is given by the reparametrized trivial strip over $a$.

Now consider the exact triangle
\begin{equation} \label{triangle in wrapped Floer}
\stackrel{\delta_*}\to H_*(L)\to HF^{\rm wr}(L) \to HF^{{\rm wr},+}(L)\simeq HC^{\varepsilon}(\Lambda)\stackrel{\delta_*}\to.
\end{equation}
Since $L$ is displaceable, $HF^{\rm wr}(L)=0$ and the theorem follows.
\end{proof}

\subsection{Nontriviality of augmentations}

We state a related result concerning the nontriviality of augmentations.
Let $(\R\times \R^3,L)$ be an exact Lagrangian filling of $\Lambda$ and let
$$\tilde \varepsilon:\mathcal{A}(\Lambda; \F[H_1(L;\Z)])\to \F[H_1(L;\Z)]$$
be the induced augmentation map with twisted coefficients.

\begin{lemma}
If $H_{1}(L;\Z)\ne 0$, then $\tilde\varepsilon\not\equiv 0$.
\end{lemma}

\begin{proof} [Sketch of proof]
Let $C_*(L;\F)$ be the Morse chain complex of $L$ with respect to a Morse function $f$.  Let
$$\delta:CC^{\tilde\varepsilon}(\Lambda;\F)\to C_*(L;\F)$$
be the chain map which induces the map $\delta_*$ in Exact Triangle~\eqref{triangle in wrapped Floer}. Since $L$ is displaceable, $\delta_*$ is an isomorphism. Let $p$ be a critical point of $f$ and let $x\in \mathcal{C}(\Lambda)$. If we write $\delta(x)=\sum_p \langle \delta(x),p\rangle \cdot p$, then $\langle \delta(x),p\rangle$ is the count of rigid pairs $(u,\gamma)$, where $u$ is a holomorphic disk with a positive end at $x$ and boundary on $L$ and $\gamma$ is a gradient flow line of $L$ emanating from the boundary of $u$ and ending at $p$.

Arguing by contradiction, if $\tilde\varepsilon\equiv 0$, then, for each $x\in\mathcal{C}(\Lambda)$ of degree zero and $\gamma\in H_1(L;\Z)$, the number of rigid holomorphic disks with a positive end at $x$ and boundary on $L$ which represent $\gamma$ is even. Hence $\delta(x)=0$, which is a contradiction.
\end{proof}

\subsection{The fundamental class}
\label{subsection: fundamental class}

Let $\varepsilon:\mathcal{A}(\Lambda)\to \F$ be an augmentation of $\Lambda\subset \R^3$.

\begin{defn}
An {\em $\varepsilon$-augmented holomorphic disk} in $(\R\times \R^3,\R\times \Lambda)$ from $a$ to ${\bf b}$ is a holomorphic disk
$$u:(D_{m+1},\bdry D_{m+1})\to (\R\times \R^3,\R\times \Lambda)$$ from $a$ to ${\bf b}'=\tau_0b_1\tau_1\dots\tau_{m-1}b_m\tau_m$, together with a subset ${\frak c}$ of $\{1,\dots,m\}$ such that applying $\varepsilon$ to all the $b_i$, $i\in{\frak c}$, yields ${\bf b}$.

If $(X,L)$ is an exact Lagrangian cobordism from $\Lambda_+$ to $\Lambda_-$ and $\varepsilon_-$ is an augmentation of $\Lambda_-$, then an {\em $\varepsilon_-$-augmented holomorphic disk} in $(X,L)$ is defined similarly.
\end{defn}

Given $y\in \R\times \Lambda$ and $a\in \mathcal{C}(\Lambda)$, let $\langle\bdry_y a,{\bf b}\rangle$ be the count of $\varepsilon$-augmented $\op{ind}=1$ holomorphic disks $u$ from $a$ to ${\bf b}$ that pass through $y$. We then write $\bdry_y a= \sum_{i=0}^\infty \bdry_y^i a$, where $\bdry_y^i$ counts the $\varepsilon$-augmented disks from $a$ to ${\bf b}$ and ${\bf b}$ has $i$ negative ends.

\begin{defn}
An {\em $\varepsilon$-fundamental class} is an element $\eta_\varepsilon$ of $HC^\varepsilon(\Lambda)$ such that $\bdry_y^0x=1$ for any generic $y\in \R\times \Lambda$ and representative $x$ of $\eta_\varepsilon$.
\end{defn}

The following is a theorem of Sabloff~\cite{Sa2} and Ekholm-Etnyre-Sabloff~\cite[Theorem 5.5]{EESa}.

\begin{thm}
For any augmentation $\varepsilon$ of $\Lambda\subset \R^3$, there exists a $\varepsilon$-fundamental class in $HC^\varepsilon(\Lambda)$.
\end{thm}

\begin{thm}\label{thm: mapping augmentations}
Let $(X,L)$ be a exact Lagrangian cobordism from $\Lambda_+$ to $\Lambda_-$ with corresponding chain map $\Phi_{(X,L)}:\mathcal{A}(\Lambda_+)\to \mathcal{A}(\Lambda_-)$, let $\varepsilon_-$ be an augmentation of $\Lambda_-$, and let $\varepsilon_+=\Phi_{(X,L)}^*(\varepsilon_-)$ be the induced augmentation of $\Lambda_+$. Assume $L$, $\Lambda_+$, and $\Lambda_-$ are connected. Then the linearization
$$\Phi_{(X,L)}:HC^{\varepsilon_+}(\Lambda_+)\to HC^{\varepsilon_-}(\Lambda_-)$$
maps an $\eta_{\varepsilon_+}$-fundamental class to an $\eta_{\varepsilon_-}$-fundamental class.
\end{thm}

\begin{proof}
Theorem~\ref{thm: mapping augmentations} follows from observing that $(X,L)$ induces a chain map from $\mathcal{A}(\Lambda_+)$ to $\mathcal{A}(\Lambda_-)$, where the differentials are ``twisted by point conditions''; the proof is similar to the usual case.

More precisely, we have the following: Let $T>0$ be sufficiently large and let $y_\pm \in \Lambda_\pm$ be generic. Take a generic path $y:\R\to L$ such that $y(t)=(t,y_+)$ for $t\geq T$ and $(t,y_-)$ for $t\leq -T$.  Define $\mathcal{M}_{y(t)}(a)$, $a\in \mathcal{C}(\Lambda_+)$, as the moduli space of $\varepsilon_-$-augmented $\op{ind}=1$ holomorphic disks in $L$ from $a$ to $\varnothing$ that pass through $y(t)$. As $t\to +\infty$, $\mathcal{M}_{y(t)}(a)$ limits to the moduli space of $\varepsilon_+$-augmented $\op{ind}=1$ disks in $\R\times \Lambda_+$ from $a$ to $\varnothing$ that pass through $(0,y_+)$. (Note that the $\varepsilon_+$-augmentation subsumes the level that is contained in $X$.) As $t\to -\infty$, $\mathcal{M}_{y(t)}(a)$ limits to the moduli space of two-level buildings $u_{-1}\cup u_0$, where $u_0$ is an $\varepsilon_-$-augmented $\op{ind}=0$ disk from $a$ to some $b\in \mathcal{C}(\Lambda_-)$ and $u_{-1}$ is an $\varepsilon_-$-augmented $\op{ind}=1$ disk from $b$ to $\varnothing$ that passes through $(0,y_-)$.  The resulting chain homotopy implies that if $x\in \eta_{\varepsilon_+}$ then the mod $2$ count of $\varepsilon_-$-augmented disks from $\Phi_{(X,L)}(x)$ to $\varnothing$ that pass through $(0,y_-)$ is equal to the mod $2$ count of $\varepsilon_+$-augmented disks from $x$ to $\varnothing$ that pass through $(0,y_+)$, which in turn is $1$.  This proves the theorem.
\end{proof}

\section{Applications}
\label{sec:examples}

\subsection{Lagrangian fillings of $\bm{(2,n)}$-torus links}
\label{sec:2n_links}

In this subsection we consider the Legendrian $(2,n)$-torus link $\Lambda_{n}$ whose Lagrangian projection is given by Figure~\ref{fig:2n}.  Each component of $\Lambda_n$ has  Maslov number $0$.  When $n$ is even, $\Lambda_n$ consists of two unknots with Thurston--Bennequin number $tb(\Lambda_n)=-1$. Let $a_1$ and $a_2$ be the two rightmost Reeb chords in Figure~\ref{fig:2n} with grading $|a_j|=1$, $j=1,2$. The remaining Reeb chords, from left to right, are denoted by $b_1,\ldots,b_n$ and satisfy $|b_j|=0$. (If $n$ is even, we choose the reference path $\delta_{12}$ to be one of the Reeb chords connecting the two components of $\Lambda_n$. We then choose the path of lines in the contact planes along $\delta_{12}$ so it makes a $\frac{\pi}{2}$-rotation with respect to the trivialization induced by $\C$.)

\begin{claim} \label{claim: contractibility}
The chords $a_1$ and $a_2$ are non-contractible and the chords $b_j$, $j=1,\dots, n$, are all contractible. Moreover, all pairs $b_j,b_k$, $|k-l|>1$, of non-adjacent degree $0$ Reeb chords of $\Lambda_{n}$ are simultaneously contractible.
\end{claim}

\begin{proof}
Follows from Lemma~\ref{condition for contractibility}.
\end{proof}

By Claim~\ref{claim: contractibility} there is a sequence of saddle cobordisms that resolves the crossings of $\Lambda_{n}$ in each of the $n!$ possible orders. We denote a permutation $\sigma\in S_n=\mbox{Aut}(\{1,\dots,n\})$ by $(i_1,\dots,i_n)$ if $\sigma(j)=i_j$.  For each permutation $\sigma=(i_{1},\dots, i_n)$, let $L_{\sigma}$ be a Lagrangian cobordism in $\R\times\R^{3}$ which is deformation-equivalent to the composition of the saddle cobordisms that resolve the degree $0$ crossings in the order $b_{i_1},\dots,b_{i_n}$ from top to bottom.

Let $\sigma=(i_1,\dots,i_n)$ be a permutation and $i_j, i_{j+1}$ be adjacent entries such that there is some $i_p$, $p>j+1$, with $i_j< i_p< i_{j+1}$ or $i_j>i_p> i_{j+1}$. If
$$\sigma'=(i_1,\dots,i_{j-1},i_{j+1},i_j,i_{j+2},\dots,i_n),$$
then $L_{\sigma}$ and $L_{\sigma'}$ are exact Lagrangian isotopic. Two permutations $\sigma$ and $\sigma'$ are {\em isotopy equivalent} if they are related by a sequence of transpositions of the above type.

\begin{rmk}\label{rem:smoothisotopy}
If we drop the Lagrangian condition on the cobordism, then any two chords become simultaneously contractible. Thus all of our Lagrangian fillings $\Lambda_\sigma$ are smoothly isotopic.
\end{rmk}

\begin{lemma}
The number of isotopy equivalence classes of permutations is the {\em Catalan number} $C_n={2n\choose n}/(n+1)$.
\end{lemma}

\begin{proof}
It is well-known that the Catalan numbers are determined by the initial value $C_0=1$ and the recurrence relation
\begin{equation} \label{eqn: recurrence}
C_{n+1}=\sum_{k=0}^nC_kC_{n-k},
\end{equation}
for $n\geq 0$.
We prove that the number $a_n$ of isotopy equivalence classes of permutations satisfies the same recurrence relation.

First note that $a_0=1$. If two permutations $(i_1,\dots,i_{n+1}), (i'_1,\dots,i'_{n+1})\in S_{n+1}$ are equivalent, then $i_{n+1}=i'_{n+1}$ since the last entry can never be part of an isotopy move. Let $S_n^{k+1}$ be the set of permutations such that $i_{n+1}=k+1$ and let $\sigma=(i_1,\dots, i_{n})\in S_n^{k+1}$. If $i_{j}>k+1$ and $i_{j+1}<k+1$, then $i_j$ and $i_{j+1}$ can be swapped by an isotopy equivalence. Hence any $\sigma\in S_n^{k+1}$ is equivalent to a permutation in normal form $\lambda=(\lambda_1,\lambda_2,k+1)$, where $\lambda_1\in \mbox{Aut}(\{1,\dots,k\})$ and $\lambda_2\in\mbox{Aut}(\{k+2,\dots,n+1\})$. Now, two permutations in $S_n^{k+1}$ in normal form are isotopy equivalent if and only if the corresponding permutations~$\lambda_1$ and $\lambda_2$ are isotopy equivalent. Hence there are $a_{k}a_{n-k}$ isotopy equivalence classes in $S_{n+1}$ with $i_{n+1}=k+1$. Summing over $k$ gives the recurrence relation \eqref{eqn: recurrence}.
\end{proof}

\begin{example}
When $n=3$, we have $C_3=5$ and the five sequences of resolutions produce the five different augmentations of the Legendrian right-handed trefoil. Hence the corresponding exact Lagrangian fillings are pairwise not exact Lagrangian isotopic. The only two isotopy equivalent permutations in this case are $(1,3,2)$ and $(3,1,2)$, and they both give rise to the augmentation $\varepsilon(b_1)=\varepsilon(b_2)=\varepsilon(b_3)=1$. In Figure \ref{fig:szetesik} we illustrated the sequence of resolutions and hence the Lagrangian cap corresponding to $(2,1,3)$. It yields the augmentation $\varepsilon(b_1)=0$, $\varepsilon(b_2)=\varepsilon(b_3)=1$.
\end{example}

\begin{figure}[ht]
\labellist
\pinlabel $b_1$ at 1200 1200
\pinlabel $b_2$ at 1500 1200
\pinlabel $b_3$ at 1800 1200
\pinlabel $b_1+1$ at 1200 870
\pinlabel $1$ at 1500 870
\pinlabel $b_3+1$ at 1800 870
\pinlabel $1+1$ at 1200 500
\pinlabel $1$ at 1500 500
\pinlabel $b_3+1+1$ at 1800 500
\pinlabel $0$ at 1200 130
\pinlabel $1$ at 1500 130
\pinlabel $1+1+1=1$ at 1850 130
\endlabellist
   \raggedright
   \includegraphics[height=3in]{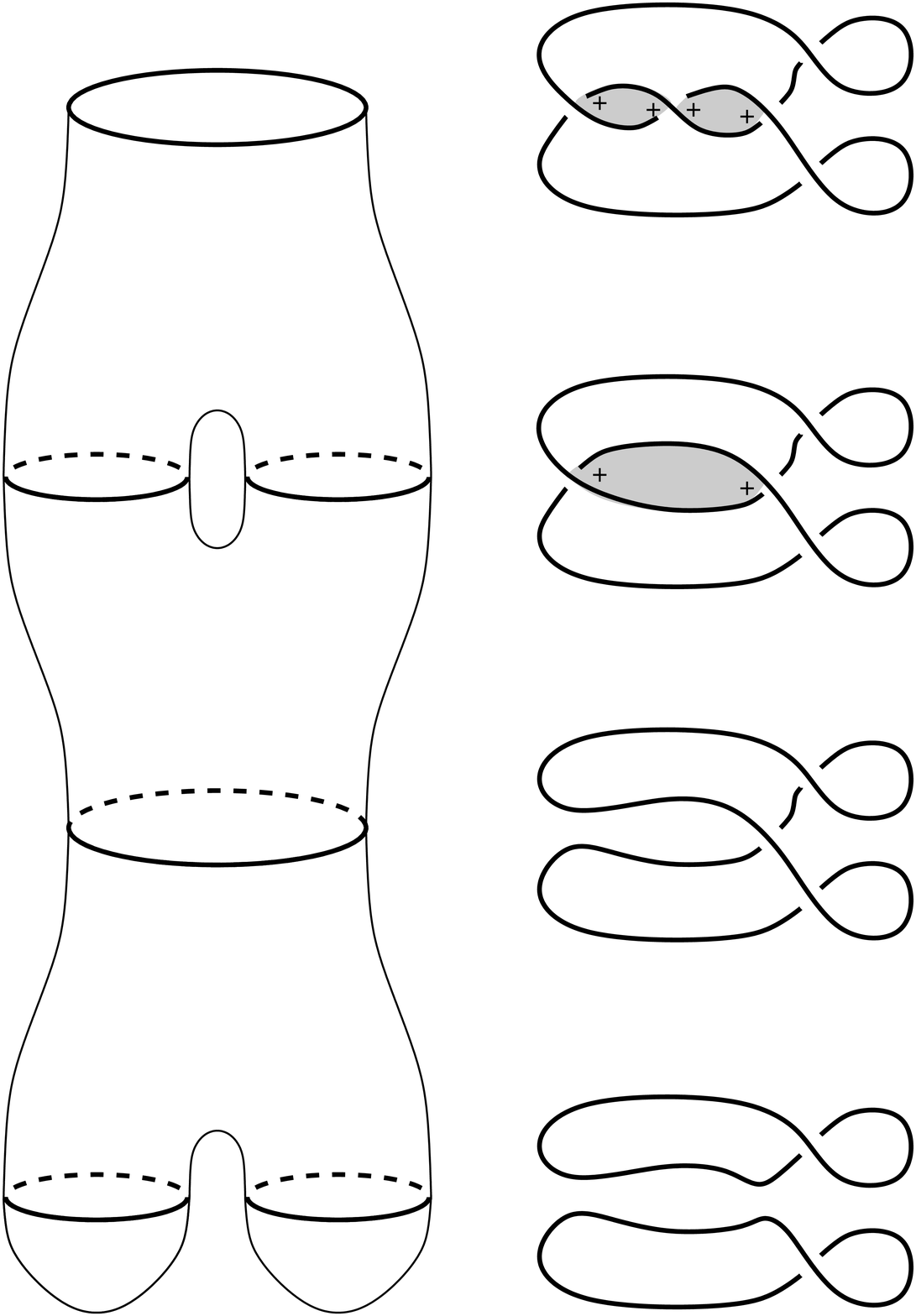}
   \caption{Constructing a Lagrangian cap and computing its induced augmentation.}\label{fig:szetesik}
\end{figure}

For larger $n$, however, the pattern is less clear. The total number of augmentations is $A_n=(2^{n+1}-(-1)^{n+1})/3$ by \cite[Proposition 7.1]{Ka2}, which is much smaller than $C_n$. When $n=5$, we have $C_5=42$ and $A_5=21$. Computing the augmentations belonging to the $42$ sequences of resolutions, we find that all $21$ possibilities occur at least once.

We can prove by a relatively easy construction that this is the case for general $\Lambda_n$. More precisely, we have:

\begin{prop} \label{prop: torus link}
Any augmentation of the $(2,2k+1)$-torus knot is induced by a sequence of crossing resolutions of the standard diagram, and so are all nonzero augmentations of the $(2,2k)$-torus link.
\end{prop}

An augmentation $\varepsilon: \mathcal{A}(\Lambda)\to \F$ is the {\em zero augmentation} (written as $\varepsilon\equiv 0$) if $\varepsilon|_{\mathcal{C}(\Lambda)}=0$.

\begin{proof}
View $\Lambda_n$, $n=2k$ or $2k+1$, as the closure of a positive braid on $2$ strands which is given by $\sigma^n$. Here $\sigma$ is a positive generator of the braid group on $2$ strands.  Define the matrix
$$B_n=\begin{bmatrix} b_1 & 1\\ 1& 0\end{bmatrix} \begin{bmatrix} b_2 & 1\\ 1& 0\end{bmatrix} \dots \begin{bmatrix} b_n & 1\\ 1& 0\end{bmatrix},$$
where the indeterminates $b_1,\dots,b_n$ are the Reeb chord generators of $\mathcal{A}(\Lambda_n)$. Let $\varepsilon:\mathcal{A}(\Lambda_n)\to \F$ be an algebra map defined by sending each degree $0$ generator $b_i$ to $0$ or $1$ and each degree $\not=0$ generator to $0$.  According to  \cite[Theorem~5.3]{Ka2}, the algebra map $\varepsilon$ is an augmentation of $\Lambda_n$ if and only if $\varepsilon(B_n)$ is of the form $\begin{bmatrix} 1 & * \\ * & * \end{bmatrix}$.

Given an augmentation $\varepsilon$ for $\Lambda_n$, resolving a crossing $q$ with $\varepsilon(q)=1$ and changing the $\varepsilon$-value (from $1$ to $0$ or from $0$ to $1$) associated to the one or two adjacent degree zero crossings yields $\Lambda_{n-1}$ with an algebra map $\varepsilon':\mathcal{A}(\Lambda_{n-1})\to \F$. We claim that $\varepsilon'$ is an augmentation. Indeed, the Boolean algebra identity
\[\begin{bmatrix}x&1\\1&0\end{bmatrix}\cdot\begin{bmatrix}1&1\\1&0\end{bmatrix}\cdot\begin{bmatrix}y&1\\1&0\end{bmatrix}=\begin{bmatrix} x+1&1\\1&0\end{bmatrix}\cdot\begin{bmatrix} y+1&1\\1&0\end{bmatrix}\]
implies that $\varepsilon'(B_{n-1}^j)$ is of the form $\begin{bmatrix} 1 & * \\ * & * \end{bmatrix}$,
where
$$B_{n-1}^j=\begin{bmatrix} b_1 & 1\\ 1& 0\end{bmatrix}\dots \begin{bmatrix} b_{j-1}+1 & 1\\ 1& 0\end{bmatrix} \begin{bmatrix} b_{j+1}+1 & 1\\ 1& 0\end{bmatrix}\dots \begin{bmatrix} b_n & 1\\ 1& 0\end{bmatrix}.$$

Next we claim that $\varepsilon'\not\equiv 0$ by a suitable choice of $q$: If there are index $0$ crossings $r$ with $\varepsilon(r)=0$, then choose $q$ with $\varepsilon(q)=1$ so that it has a neighbor $r$ with $\varepsilon(r)=0$. If $\varepsilon\equiv 1$ on the index $0$ crossings, then let $q$ be the leftmost or rightmost crossing. We have $\varepsilon'\not\equiv 0$ as long as $n>2$. Finally observe that $\varepsilon(b_1)=\varepsilon(b_2)=1$ is not an augmentation for the Hopf link $\Lambda_2$.

The proof then proceeds by induction.
\end{proof}

\begin{rmk}
Although $\varepsilon(b_1)=\varepsilon(b_2)=0$ is an augmentation of the Hopf link $\Lambda_2$, $\varepsilon$ cannot come from an embedded Lagrangian by Theorem~\ref{Seidel theorem}.
\end{rmk}

Not much is known about exact Lagrangian fillings of Legendrian links.  We conclude this subsection with some open questions.

\begin{q}
If $\sigma_1,\sigma_2\in S_n$ are not isotopy equivalent but correspond to exact Lagrangian fillings $L_{\sigma_1},L_{\sigma_2}$ of $\Lambda_n$ which induce the same augmentation, then are $L_{\sigma_1}$ and $L_{\sigma_2}$ isotopic through exact Lagrangians?
\end{q}

\begin{q}
Is any exact Lagrangian filling of $\Lambda_n$ which induces a nonzero augmentation exact Lagrangian isotopic to one coming from a sequence of resolutions of the ``standard model'' of $\Lambda_n$ given by Figure~\ref{fig:2n}? This is not known even for $n=3$.
\end{q}

\begin{q}
If $\Lambda$ is a Legendrian knot and $g_s$ is its slice genus, then is every augmentation $\varepsilon$ of $\Lambda$ geometric, provided $HC^\varepsilon(\Lambda)\simeq H_*(\Sigma_{g_s})$, where $\Sigma_{g_s}$ is the once-punctured oriented surface of genus $g_s$? In other words, is Theorem~\ref{Seidel theorem} the only obstruction to the existence of an exact Lagrangian filling that induces $\varepsilon$?
\end{q}

\begin{q}
Is every exact Lagrangian filling of a Legendrian knot $\Lambda$ decomposable?
\end{q}

In the case of a concave filling, there is an example of Sauvaget~\cite{Sau}, which after some modification, yields an exact Lagrangian cobordism of genus $2$ from $\varnothing$ to a stabilized unknot.  A concave filling cannot be decomposable.

\subsection{Connected sums}

\subsubsection{Comultiplication}

We now introduce the exact Lagrangian cobordism and the accompanying comultiplication map induced by (the inverse of) a connected sum.  Let $\Lambda_1,\Lambda_2\subset (\R^3,\xi_0)$ be Legendrian knots and let $\Lambda_1\#\Lambda_2$ be their connected sum, schematically given in Figure~\ref{connectedsum}.

\begin{figure}[ht]
\s
\begin{center}
\psfragscanon
\psfrag{A}{$\Lambda_1$}
\psfrag{B}{$\Lambda_2$}
\includegraphics[width=0.3\linewidth]{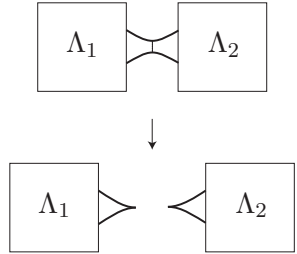}
\end{center}
\s
\caption{The inverse of a connected sum of $\Lambda_1$ and $\Lambda_2$ in the front projection.}
\label{connectedsum}
\end{figure}

Suppose without loss of generality that $\Lambda_1\sqcup \Lambda_2$ is obtained from $\Lambda_1\#\Lambda_2$ by resolving a simple contractible Reeb chord $a$.  Then
$$\mathcal{C}(\Lambda_1\#\Lambda_2)\simeq \mathcal{C}(\Lambda_1) \cup \mathcal{C}(\Lambda_2)\cup\{a\}$$
and the corresponding simple saddle cobordism $L$ induces the coproduct map
\begin{equation} \label{coprod map}
\Phi_L: \mathcal{A}(\Lambda_1\#\Lambda_2)\to \mathcal{A}(\Lambda_1\sqcup \Lambda_2)
\end{equation}
which maps $x\mapsto x$ if $x\in \Lambda_1$ or $\Lambda_2$ and $a\mapsto 1$.  This is due to the fact that there are no holomorphic disks in $\Lambda_1\#\Lambda_2$ that have $a$ at the positive end.

\begin{example}
Let $\Lambda_1$ and $\Lambda_2$ both be standard $tb=-1$, $r=0$ Legendrian unknots $U$ and let $x_i$, $i=1,2$, be the generators of $\mathcal{A}(\Lambda_i)$. Then $\Lambda_1\#\Lambda_2$ is isotopic to $U$ with generator $x$.  Let $\F\{S\}$ be the free $\F$-algebra generated by the set $S$.  On the level of homology the coproduct map is:
$$\F\{x\}\to \F\{x_1,x_2\},\quad x\mapsto x_1+x_2.$$
If we abelianize the DGA morphism $\mathcal{A}(\Lambda_1\#\Lambda_2) \to \mathcal{A}(\Lambda_1\sqcup\Lambda_2)$, then on the level of homology the coproduct map is:
$$\F[x]\to \F[x_1]\otimes \F[x_2], \quad x\mapsto x_1\otimes 1 + 1\otimes x_2,$$
which is the coproduct on the $S^1$-equivariant cohomology of a point.
\end{example}

\subsubsection{Multiplication}

Next we introduce the exact symplectic cobordism induced by a stabilized connected sum $S_+S_-(\Lambda_1\#\Lambda_2)$, where $S_+$ (resp.\ $S_-$) denotes a single positive (resp.\ negative) stabilization. This is given by Figure~\ref{connectedsum2}. Unfortunately, the multiplication map
$$\mathcal{A}(\Lambda_1)\otimes \mathcal{A}(\Lambda_2)\to \mathcal{A}(S_+S_-(\Lambda_1\# \Lambda_2))$$ is the zero map and Legendrian contact homology gives no information.

\begin{figure}[ht]
\begin{center}
\psfragscanon
\psfrag{A}{$\Lambda_1$}
\psfrag{B}{$\Lambda_2$}
\includegraphics[width=0.3\linewidth]{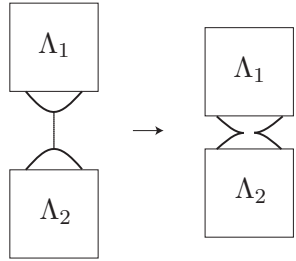}
\end{center}
\caption{The stabilized connected sum of $\Lambda_1$ and $\Lambda_2$ in the front projection.}
\label{connectedsum2}
\end{figure}

The above operation nonetheless can be used to prove the following theorem:

\begin{thm}
Given any Legendrian link $\Lambda\subset \R^3$, there is an exact Lagrangian cobordism from $\Lambda$ to some stabilized unknot.
\end{thm}

\begin{proof}
Consider the image of $\Lambda$ under the front projection $\Pi_{J^0\R}: \R^3\to\R^2$, $(x,y,z)\mapsto (x,z)$. We slice $\Pi_{J^0\R}(\Lambda)$ along $x=x_i$ for $x_1<\dots< x_n$ such that there is at most one crossing on each $\{x_i<x<x_{i+1}\}$ and there are no cusps or crossings on each $x=x_i$.  (Here we are assuming that $\Lambda$ is sufficiently generic.) We then repeatedly apply the pinching procedure of Figure~\ref{fig:saddleends}, making sure that at each step we pinch strands that are pointing in opposite directions. The result of applying the procedure to the slice $\{x_i<x<x_{i+1}\}$ is given in Figure~\ref{iteratedpinching},
\begin{figure}[ht]
\begin{center}
\includegraphics[width=0.45\linewidth]{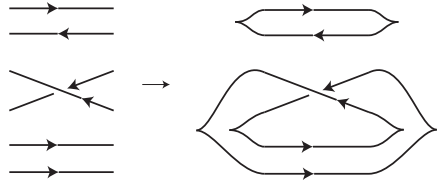}
\end{center}
\caption{Iterated pinching.}
\label{iteratedpinching}
\end{figure}
and is a disjoint union of possibly stabilized unknots.  Finally we apply the stabilized connected sum operation to obtain a stabilized unknot.
\end{proof}

\section{Khovanov homology}
\label{sec:Khovanov}

In this section we assume that the reader is familiar with basic facts about Khovanov homology \cite{Kh}.  We refer the reader to \cite{Ja} for a discussion of the map on Khovanov homology induced by a cobordism of links.

The Khovanov homology of a link $\Lambda\subset \R^3$ will be denoted by $Kh(\Lambda)=\oplus_{i,j} K^{i,j} (\Lambda)$, where $i$ is the homological grading and $j$ is the $q$-grading. (Strictly speaking, when $\Lambda$ is Legendrian, we consider the Lagrangian projection $\pi_\C(\Lambda)$.) The following is due to Jacobsson~\cite[Section 4.2]{Ja}.

\begin{thm}[Jacobsson] \label{thm: jacobsson}
If $\Lambda_+,\Lambda_-\subset \R^3$ are links and $L\subset [-1,1]\times \R^3$ is a smooth cobordism from $\{1\}\times \Lambda_+$ to $\{-1\}\times \Lambda_-$, then it induces a map
$$\Psi_L\colon Kh(\Lambda_-)\rightarrow Kh(\Lambda_+)$$
that is well-defined up to $\pm 1$ and is invariant under smooth isotopies relative to $\bdry L$.
\end{thm}

The isotopy is required to fix $\Lambda_+$ and $\Lambda_-$ pointwise to avoid monodromy issues. Note that any isotopy invariant derived from $\Psi_L$ will have this $\pm 1$ ambiguity as well.

Let $\overline{K}$ denote the mirror of the link $K$. Given a Legendrian link $\Lambda\subset \R^3$ and an exact Lagrangian cobordism $L$ from $\Lambda$ to $\varnothing$, Theorem~\ref{thm: jacobsson} gives a map
$$\Psi_{\overline{L}}: Kh(\varnothing)\simeq \Z\to Kh(\overline{\Lambda}).$$

Let $\Lambda\subset\R^{3}$ be a Legendrian link. Let $\mathcal{S}$ (resp.\ $\mathcal{S}^{\rm dec}$) be the set of isotopy classes of exact (resp.\ decomposable exact) Lagrangian cobordisms $L$ from $\Lambda$ to $\varnothing$.  Then we define
$$ \mathcal{L}_\Lambda=\bigsqcup_{L\in\mathcal{S}^{\rm dec}}\{\Psi_{\overline{L}}(\pm 1)\}\subset Kh(\overline{\Lambda}),\quad
\mathcal{L}'_\Lambda=\bigsqcup_{L\in\mathcal{S}}\{\Psi_{\overline{L}}(\pm1)\}\subset Kh(\overline{\Lambda}).$$

By definition, $\mathcal{L}_\Lambda$ and $\mathcal{L}'_\Lambda$ are invariants of $\Lambda$, although not particularly computable.

Let $L$ be a decomposable exact Lagrangian cobordism from $\Lambda_+$ to $\Lambda_-$ which is the composition of elementary cobordisms $L_i$ from $\Lambda_i$ to $\Lambda_{i+1}$, $i=1,\dots,n-1$, where $\Lambda_+=\Lambda_1$ and $\Lambda_-=\Lambda_n$. Then we define
\[
\Psi'_{\overline{L}_i}\colon Kh(\overline{\Lambda}_{i+1})\rightarrow Kh(\overline{\Lambda}_i)
\]
and $\Psi'_{\overline{L}}$ as the composition of the $\Psi'_{\overline{L}_i}$ as follows:
If $K_1$ is the $1$-resolution of $K$, then there exists a map $Kh(K_1)\rightarrow Kh(K)$.  In our case, mirroring transforms a $0$-resolution to a $1$-resolution and we have a map $Kh(\overline{\Lambda}_{i+1})\rightarrow Kh(\overline{\Lambda}_i)$ if $\Lambda_{i+1}$ is a $0$-resolution of $\Lambda_i$. If $\Lambda_{i+1}$ is obtained from $\Lambda_i$ by a Legendrian isotopy, then the isotopy induces the map $Kh(\overline{\Lambda}_{i+1})\rightarrow Kh(\overline{\Lambda}_i)$.

\begin{lemma}
The map $\Psi'_{\overline{L}}\colon Kh(\overline{\Lambda}_n)\rightarrow Kh(\overline{\Lambda}_1)$ induced by the decomposable exact Lagrangian cobordism $L$ agrees with $\Psi_{\overline{L}}$, where we view $\overline{L}$ as a smooth surface from $\overline{\Lambda}_n$ to $\overline{\Lambda}_1$.
\end{lemma}

\begin{proof}
In \cite{Ja}, the $1$-resolution map $Kh(K_1)\rightarrow Kh(K)$ is not used in the definition of the cobordism map $\Psi_{L}$. Instead, the $1$-resolution $K_1\rightarrow K$ must be factored into the composition of a Reidemeister~I move~$K_1 \rightarrow K'$ followed by a Morse saddle move $K'\rightarrow K$.  One easily verifies that the composition $Kh(K_1)\rightarrow Kh(K')\rightarrow Kh(K)$ agrees with the $1$-resolution map.
\end{proof}

Let $L$ be an exact Lagrangian filling of $\Lambda$.  We view $L$ as a composition of a cobordism $L_1$ from $\Lambda$ to the unknot $U$, followed by a minimal cobordism $L_2$. Let $\varepsilon_-$ and $\varepsilon_+$ be the augmentations induced by $L_2$ and $L$, respectively. The augmentation $\varepsilon_+$ gives rise to an automorphism of $\mathcal{A}(\Lambda)$ which takes $a\in\mathcal{C}(\Lambda)$ to $\overline{a}=a+\varepsilon_+(a)$, so that the differential respects the filtration by word length in the $\overline{a}$'s. Let $E^1(\mathcal{A}(\Lambda),\varepsilon_+)$ be the $E^1$-term of the associated spectral sequence; one easily sees that $E^1(\mathcal{A}(\Lambda),\varepsilon_+)$ is the unital tensor algebra generated by $HC^{\varepsilon_+}(\Lambda)$. The we can define a map:
$$\Theta: E^1(\mathcal{A}(\Lambda),\varepsilon_+)\to Kh(\overline{\Lambda}),$$
obtained by composing the maps
$$\Phi_{L_1}: E^1(\mathcal{A}(\Lambda),\varepsilon_+)\to E^1(\mathcal{A}(U),\varepsilon_-),$$
$$\Psi_{\overline{L}_1}: Kh(\overline{U})\to Kh(\overline{\Lambda}_+),$$
in Legendrian contact homology and Khovanov homology, together with
$$E^1(\mathcal{A}(U),\varepsilon_-)\simeq \F[x]\to Kh(\overline{U})\simeq V,$$
where $1\mapsto \mathbf{v}_+$, $x\mapsto \mathbf{v}_-$, and $V$ is a $2$-dimensional $\F$-vector space generated by $\mathbf{v}_+$ of degree $1$ and $\mathbf{v}_-$ of degree $-1$.

We close this subsection with some questions and a remark.

\begin{q}
Is the map $\Theta$ good for anything?
\end{q}

\begin{q}
Can the set $\mathcal{L}_\Lambda$ and $\mathcal{L}'_\Lambda$ have more than two elements? Are $\mathcal{L}_\Lambda$ and $\mathcal{L}'_\Lambda$ finite?
\end{q}

Observe that $\mathcal{L}_\Lambda$ and $\mathcal{L}'_\Lambda$ are empty if there are no exact Lagrangian cobordisms from $\Lambda$ to $\varnothing$. By the $\pm 1$ ambiguity, the cardinalities of $\mathcal{L}_\Lambda$ and $\mathcal{L}_\Lambda'$ are even.

\begin{q}
Are there only finitely many exact Lagrangian isotopy classes of exact Lagrangians which bound a given Legendrian knot $\Lambda$?
\end{q}

Since there are no local Lagrangian $2$-knots in $4$-space \cite{EP}, one might expect that the answer is yes.

\begin{example}
Let $\Lambda_3$ be the Legendrian right-handed trefoil knot with $tb=1$ from Section~\ref{sec:2n_links}.  For each of the 5 exact Lagrangian cobordisms $L_1,\dots,L_5$ constructed in Section~\ref{sec:2n_links}, we can assign an element of $Kh(\overline{\Lambda})$.  For each $L_i$, the element $\Psi_{\overline{L}_i}(1)$ coincides with Plamenevskaya's transverse knot invariant \cite{Pl}, and also lies on Ng's \cite{Ng} line $j-i=C$, where
\[
C=\{\mbox{max}(j-i)~|~ Kh^{i,j}(\overline{L})\not=0\}.
\]
By Remark~\ref{rem:smoothisotopy}, the exact Lagrangians $L_i$ are all smoothly isotopic.
\end{example}

\begin{q}
Clarify the relationship among $\mathcal{L}_\Lambda$ and $\mathcal{L}_\Lambda'$, Plamenevskaya's transverse knot invariant, and Ng's line.
\end{q}

\s\n {\em Acknowledgements.} We thank Baptiste Chantraine, Klaus Mohnke, and Paul Seidel for useful discussions. KH also thanks Stanford University and the Simons Center for Geometry and Physics for their hospitality.

\end{document}